\documentclass[12pt,reqno]{amsart}%
\usepackage{amsfonts}
\usepackage{amsmath}
\usepackage{amssymb}
\usepackage{graphicx}%
\providecommand{\U}[1]{\protect\rule{.1in}{.1in}}
\newtheorem{theorem}{Theorem}
\theoremstyle{plain}

\newtheorem{corollary}{Corollary}

\newtheorem{definition}{Definition}

\newtheorem{lemma}{Lemma}

\newtheorem{proposition}{Proposition}
\newtheorem{remark}{Remark}

\numberwithin{theorem}{section} \numberwithin{equation}{section}
\numberwithin{lemma}{section}
\numberwithin{remark}{section}
\numberwithin{example}{section}
\numberwithin{proposition}{section}
\numberwithin{definition}{section}
\numberwithin{corollary}{section}
\begin{document}
\title[Noncommutative localizations]{Noncommutative localizations of a contractive quantum plane}
\author{Anar Dosi}
\address{College of Mathematical Sciences, Harbin Engineering University, Nangang
District, Harbin, 150001, China}
\email{dosiev@yahoo.com dosiev@hrbeu.edu.cn}
\date{December 1, 2025}

\begin{abstract}
In the present paper we investigate the localizations in the sense of J. L.
Taylor of the Arens-Michael-Fr\'{e}chet algebras associated with
noncommutative analytic spaces of a contractive $q$-plane representing its
formal geometry. It turns out that all noncommutative Fr\'{e}chet algebras
obtained by the Fr\'{e}chet algebra structure sheaves over open subsets from
the topology bases are indeed localizations. That topological homology
property of the structure sheaves results in the key properties of Taylor and
Putinar spectra of the left Banach $q$-modules over the algebras of global sections.

\end{abstract}
\maketitle

\section{Introduction\label{secInt}}

The localizations or (weak) homological epimorphisms introduced by J. L.
Taylor \cite{Tay2} play one of the central roles in the multivariable
functional calculus problem for the operator tuples, and noncommutative
complex analytic geometry. They have also significant impact in computations
of the homological dimensions of topological algebras, and many other problems
of topological homology developed by J. L. Taylor \cite{TayHom} and A. Ya.
Helemskii \cite{HelHom}. The key idea of a localization is to draw the pure
algebra homology construction over a base algebra $B$ to topological homology
of a polynormed algebra $A$ along the given canonical algebra homomorphism
$B\rightarrow A$. Since every pure algebra can be equipped with the finest
polynormed topology, the localization can be treated as a polynormed algebra
homomorphism $B\rightarrow A$, which identifies the category of polynormed
$A$-modules with a full subcategory of the category of polynormed $B$-modules
so that topological homology over $A$ is reduced to topological homology over
$B$ along the homomorphism $B\rightarrow A$. In the purely algebraic case, we
used to say that $B\rightarrow A$ is a homological epimorphism \cite{GL}, or
$A$ is stably flat over $B$ \cite{NR}, which is equivalent to the condition
$\operatorname{Tor}_{n}^{B}\left(  A,A\right)  =0$ for all $n>0$. The stably
flatness plays an important role in some problems of representation theory of
finite-dimensional algebras \cite{GBW}, and algebraic $K$-theory \cite{NR}.
The localization property within the content of noncommutative algebras is
referred as noncommutative localization similar to the algebraic case (see for
example \cite{NR}).

In noncommutative complex analytic geometry, the base algebra $B$ stands for a
finitely generated algebra of noncommutative variables, whereas its
Fr\'{e}chet algebra localizations represent the algebras of local sections of
germs of noncommutative holomorphic functions in elements of $B$ defined on a
noncommutative space $\mathfrak{X}$ of $B$. A noncommutative analytic space
(or a geometric model) associated with $B$ is related to its certain
multinormed envelope $\widehat{B}$ that represents the
Arens-Michael-Fr\'{e}chet algebra of some (noncommutative) entire functions in
elements of $B$. For example, so are its Arens-Michael envelope \cite[5.2.21]%
{Hel} (all entire functions in elements of $B$), the PI-envelope \cite{Aris}
of $B$, and there are others obtained as the universal objects in some classes
of Banach algebra homomorphisms of $B$. Upon specifying the multinormed
envelope $\widehat{B}$ so that $B\rightarrow\widehat{B}$ is a localization,
the underlying topological space $\mathfrak{X}$ is defined as the spectrum of
$\widehat{B}$ to be the set of all its irreducible Banach space
representations \cite[6.2]{Hel} equipped with a certain (used to be
non-Hausdorff) topology. It contains the space $\operatorname{Spec}\left(
\widehat{B}\right)  $ of all continuous characters (or one-dimensional
irreducible representations) of $\widehat{B}$. By \textit{a noncommutative
analytic space of} $\widehat{B}$ we mean a ringed space $\left(
\mathfrak{X,F}\right)  $ equipped with a Fr\'{e}chet algebra (pre)sheaf
$\mathfrak{F}$ on the spectrum $\mathfrak{X}$ such that $\mathfrak{F}\left(
\mathfrak{X}\right)  =\widehat{B}$ and $U\supseteq\operatorname{Spec}\left(
\mathfrak{F}\left(  U\right)  \right)  $ for every open subset $U\subseteq
\mathfrak{X}$. We also say that $\left(  \mathfrak{X,F}\right)  $ is \textit{a
}(\textit{noncommutative})\textit{ complex analytic geometry that stands for}
$\widehat{B}$. If $\mathfrak{X=}\operatorname{Spec}\left(  \widehat{B}\right)
$, then $\left(  \mathfrak{X,F}\right)  $ can be referred as \textit{a
standard geometry }of $\widehat{B}$, so is the commutative case of $B$ (see
below). The \textit{noncommutative localization property of }$\left(
\mathfrak{X,F}\right)  $ is to select of a countable topology base $\left\{
U_{i}\right\}  $ for $\mathfrak{X}$ so that all restriction homomorphisms
$\widehat{B}\rightarrow\mathfrak{F}\left(  U_{i}\right)  $ (or $B\rightarrow
\mathfrak{F}\left(  U_{i}\right)  $) are localizations. The related
noncommutative $\mathfrak{F}$-functional calculus problem for the Banach space
representations of $\widehat{B}$ is in turn associated with noncommutative
$\widehat{B}$-spectral theory and the related topics.

In the case of the (commutative) polynomial base algebra $B=\mathbb{C}\left[
x_{1},\ldots,x_{n}\right]  $, its Arens-Michael envelope $\widehat{B}$
(including other envelopes) is reduced to the algebra $\mathcal{O}\left(
\mathbb{C}^{n}\right)  $ of all entire functions in $n$ complex variables, and
the commutative analytic space $\left(  \mathbb{C}^{n},\mathcal{O}\right)  $
with the Fr\'{e}chet algebra sheaf $\mathcal{O}$ of germs of holomorphic
functions on $\mathbb{C}^{n}$ stands for the standard complex analytic
geometry of $\widehat{B}$ (or $B$). In this case, the canonical inclusions
$B\rightarrow\mathcal{O}\left(  U\right)  $ over domains $U\subseteq
\mathbb{C}^{n}$ of holomorphy are all localizations, that is, $\left(
\mathbb{C}^{n},\mathcal{O}\right)  $ has the localization property. Moreover,
the $\mathcal{O}$-functional calculus problem is solved by Taylor's
holomorphic functional calculus and the joint spectral theory of several
commuting operators \cite{EP}, \cite{Tay}, \cite{Tay2}.

In the noncommutative setting, the proposed framework of noncommutative
analytic spaces including the related functional calculus and spectral theory
were implemented in the case of the universal enveloping (base) algebra
$B=\mathcal{U}\left(  \mathfrak{g}\right)  $ of a finite dimensional nilpotent
Lie algebra $\mathfrak{g}$ \cite{DosiIzv09}, \cite{DosiAC}, \cite{DosJOT10}.
The localizations of $\mathcal{U}\left(  \mathfrak{g}\right)  $ for a finite
dimensional Lie algebra $\mathfrak{g}$ were gradually investigated in the last
20 years mainly by O. Yu. Aristov, A. Yu. Pirkovskii and the author in a
series of papers. The first target was the canonical map of $\mathcal{U}%
\left(  \mathfrak{g}\right)  $ into its Arens-Michael envelope $\mathcal{O}%
\left(  \mathfrak{g}\right)  $ that represents the algebra of all
noncommutative entire functions in variables generating $\mathfrak{g}$.
Certainly $\mathcal{U}\left(  \mathfrak{g}\right)  \rightarrow\mathcal{O}%
\left(  \mathfrak{g}\right)  $ is a localization when $\mathfrak{g}$ is an
abelian Lie algebra, and that is not the case if $\mathfrak{g}$ is a
semisimple Lie algebra \cite{Tay2}. In \cite{Pir16}, A. Yu. Pirkovskii proved
that the solvability of $\mathfrak{g}$ is a necessary condition for
$\mathcal{U}\left(  \mathfrak{g}\right)  \rightarrow\mathcal{O}\left(
\mathfrak{g}\right)  $ to be a localization. The case of a nilpotent Lie
algebra $\mathfrak{g}$ was investigated in \cite{Dosi2003}, \cite{DosiMS10}
and \cite{DosiK} mainly motivated by implementation of Taylor`s general
framework in the noncommutative case \cite{Dos}, \cite{DosCS}, \cite{DosJOT10}%
. The detailed analysis of the localization problem $\mathcal{U}\left(
\mathfrak{g}\right)  \rightarrow\mathcal{O}\left(  \mathfrak{g}\right)  $ was
considered in \cite{PirSF}, and completed in \cite{Aris20} by O. Yu. Aristov,
where it was shown that $\mathcal{U}\left(  \mathfrak{g}\right)
\rightarrow\mathcal{O}\left(  \mathfrak{g}\right)  $ is a localization for
every nilpotent Lie algebra $\mathfrak{g}$. The case of a non-nilpotent but
solvable Lie algebra $\mathfrak{g}$ was recently solved in \cite{Aris25}. Thus
$\mathcal{O}\left(  \mathfrak{g}\right)  $ is stably flat over $\mathcal{U}%
\left(  \mathfrak{g}\right)  $ if and only if $\mathfrak{g}$ is a solvable Lie
algebra. The localization property of the noncommutative analytic space
associated with the PI-envelope of $\mathcal{U}\left(  \mathfrak{g}\right)  $
in the case of a nilpotent Lie algebra $\mathfrak{g}$ was solved in
\cite{DosiAC}. The related structure sheaf is the Fr\'{e}chet algebra sheaf of
formally-radical functions in elements of $\mathfrak{g}$ \cite{DosJOT10}.

In the present paper we investigate the localization problem for the
noncommutative analytic spaces of the formal geometry of a contractive
$q$-plane from \cite{Dosi25}. The quantum plane (or just $q$-plane) is the
free associative algebra
\[
\mathfrak{A}_{q}=\mathbb{C}\left\langle x,y\right\rangle /\left(
xy-q^{-1}yx\right)  ,\quad q\in\mathbb{C}\backslash\left\{  0,1\right\}
\]
generated by $x$ and $y$ modulo $xy=q^{-1}yx$ (see \cite{M}, \cite{Goor}). If
$\left\vert q\right\vert \neq1$, then we come up with a contractive quantum
plane. The Arens-Michael envelope of the base algebra $B=\mathfrak{A}_{q}$
denoted by $\mathcal{O}_{q}\left(  \mathbb{C}^{2}\right)  $ represents the
algebra of all entire functions in noncommutative variables $x$, $y$. It turns
out that $\mathcal{O}_{q}\left(  \mathbb{C}^{2}\right)  $ is commutative
modulo its Jacobson radical $\operatorname{Rad}\mathcal{O}_{q}\left(
\mathbb{C}^{2}\right)  $ whenever $\left\vert q\right\vert \neq1$ (see
\cite{Dosi24}), which means that the spectrum $\mathfrak{X}$ of $\mathcal{O}%
_{q}\left(  \mathbb{C}^{2}\right)  $ is reduced to $\operatorname{Spec}\left(
\mathcal{O}_{q}\left(  \mathbb{C}^{2}\right)  \right)  $. In this case, it
consists of two complex lines
\[
\mathfrak{X}=\mathbb{C}_{xy}:=\mathbb{C}_{x}\cup\mathbb{C}_{y}\quad
\text{with\quad}\mathbb{C}_{x}=\mathbb{C\times}\left\{  0\right\}
\subseteq\mathbb{C}^{2},\quad\mathbb{C}_{y}=\left\{  0\right\}  \times
\mathbb{C\subseteq C}^{2},
\]
and their unique intersection point $\left(  0,0\right)  $. One can equip
$\mathbb{C}_{x}$ with the $q$-topology (see below Subsection \ref{subsecQT})
and the other line $\mathbb{C}_{y}$ with the disk topology (see \cite{Dosi24})
or vice-versa. Then $\mathbb{C}_{xy}$ is equipped with the final topology so
that both inclusions of these lines into $\mathbb{C}_{xy}$ are continuous. The
(noncommutative) analytic space $\left(  \mathbb{C}_{xy},\mathcal{O}%
_{q}\right)  $ that stands for $\mathcal{O}_{q}\left(  \mathbb{C}^{2}\right)
$ is given by the topological space $\mathbb{C}_{xy}$ and the structure
Fr\'{e}chet algebra presheaf $\mathcal{O}_{q}$ \cite{Dosi24}. The related
$\mathcal{O}_{q}$-functional calculus problem is solved in \cite{Dosi24}.

The PI envelope $\mathcal{F}_{q}\left(  \mathbb{C}_{xy}\right)  $ of
$\mathfrak{A}_{q}$ admits the standard geometry $\left(  \mathbb{C}%
_{xy},\mathcal{F}_{q}\right)  $ referred as the formal geometry of
$\mathfrak{A}_{q}$, which has the same spectrum $\mathbb{C}_{xy}$ but now
equipped with the $q$-topology over both lines, and the Fr\'{e}chet algebra
sheaf $\mathcal{F}_{q}$ \cite{Dosi25}. The structure sheaf $\mathcal{F}_{q}$
is obtained as the fibered product
\[
\mathcal{F}_{q}=\mathcal{O}\left[  \left[  y\right]  \right]
\underset{\mathbb{C}\left[  \left[  x,y\right]  \right]  }{\times}\left[
\left[  x\right]  \right]  \mathcal{O}%
\]
of the formal power series sheaves $\mathcal{O}\left[  \left[  y\right]
\right]  $ and $\left[  \left[  x\right]  \right]  \mathcal{O}$ over the
constant sheaf $\mathbb{C}\left[  \left[  x,y\right]  \right]  $, where
$\mathcal{O}$ is the sheaf of stalks of holomorphic functions on the complex
$q$-plane. The multiplications in all these sheaves are naturally extended by
one of the original algebra $\mathfrak{A}_{q}$ (or $\mathcal{O}_{q}\left(
\mathbb{C}_{xy}\right)  $). The connecting morphisms $\mathcal{O}\left[
\left[  y\right]  \right]  \rightarrow\mathbb{C}\left[  \left[  x,y\right]
\right]  $ and $\left[  \left[  x\right]  \right]  \mathcal{O}\rightarrow
\mathbb{C}\left[  \left[  x,y\right]  \right]  $ defining the fibered product
are evaluations maps of stalks at the intersection point of two irreducible
subsets $\mathbb{C}_{x}$ and $\mathbb{C}_{y}$ in $\mathbb{C}_{xy}$. Thus, we
come up with the following analytic spaces $\left(  \left\{  \left(
0,0\right)  \right\}  ,\mathbb{C}\left[  \left[  x,y\right]  \right]  \right)
$, $\left(  \mathbb{C}_{x},\mathcal{O}\left[  \left[  y\right]  \right]
\right)  $ and $\left(  \mathbb{C}_{y},\left[  \left[  x\right]  \right]
\mathcal{O}\right)  $ representing certain multinormed envelopes of
$\mathfrak{A}_{q}$, whose spectra are the intersection point $\left(
0,0\right)  $, the $x$-line $\mathbb{C}_{x}$, and the $y$-line $\mathbb{C}%
_{y}$, respectively (see Definitions \ref{def1} and \ref{def2}). Moreover, all
these analytic spaces are linked with the natural morphisms (see Subsection
\ref{subsecNCCP}), and $\left(  \mathbb{C}_{xy},\mathcal{F}_{q}\right)  $ is
the fibered product of the analytic spaces $\left(  \mathbb{C}_{x}%
,\mathcal{O}\left[  \left[  y\right]  \right]  \right)  $ and $\left(
\mathbb{C}_{y},\left[  \left[  x\right]  \right]  \mathcal{O}\right)  $ over
$\left(  \left\{  \left(  0,0\right)  \right\}  ,\mathbb{C}\left[  \left[
x,y\right]  \right]  \right)  $ (see Definition \ref{def3}).

The main result of the present work states that all these noncommutative
analytic spaces of the formal geometry of $\mathfrak{A}_{q}$ possess the
localization property. Namely, if $U$ is a $q$-open subset of the complex
$q$-plane, then the canonical embeddings of $\mathfrak{A}_{q}$ into the
Arens-Michael-Fr\'{e}chet algebras $\mathbb{C}\left[  \left[  x,y\right]
\right]  $, $\mathcal{O}\left(  U\right)  \left[  \left[  y\right]  \right]  $
and $\left[  \left[  x\right]  \right]  \mathcal{O}\left(  U\right)  $ are
localizations. Moreover, $\mathfrak{A}_{q}\rightarrow\mathcal{F}_{q}\left(
U_{i}\right)  $ are localizations over a topology base $\left\{
U_{i}\right\}  $ for $\mathbb{C}_{xy}$, which consists of Runge $q$-open
subsets (see Definition \ref{def6}). In particular, the PI-envelope of
$\mathfrak{A}_{q}$ is a localization. Thus we have the full description of the
homological properties of these algebras. Our approach to the problem is
crucially based on the Decomposition Theorem from \cite{Dosi25}, and the
triangular representations of the differentials by involving some partial
differential and difference operators. The related technical machinery is
developed in Section \ref{sectionLRq}, and some long calculations on the
diagonal cohomology groups are reflected in the Appendix Secton
\ref{sectionDCG}. A proper discussion of localizations we start in Section
\ref{sectionNL}. The main result (Theorem \ref{thCENTER}) of the paper is
proved in Section \ref{sectionFGq} right after a brief discussion of the
formal geometry of $\mathfrak{A}_{q}$.

Finally, in Section \ref{secJS} we analyze the joint spectra obtained from
topological homology constructions. The transversality theorem from
\cite[Theorem 6.5]{Dosi25} combined with our result on the localizations yield
that $\mathcal{F}_{q}\left(  U\right)  \perp X$ (see Subsection \ref{subSth})
holds iff $U\cap\sigma\left(  T,S\right)  =\varnothing$, whenever $U$ is a
Runge $q$-open subset and $X$ is a left Banach $q$-module. We also prove that
the $q$-closure of the Taylor spectrum $\sigma\left(  T,S\right)  $ of a left
Banach $q$-module $X$ given by an operator couple $\left(  T,S\right)  $ on
$X$ with $TS=q^{-1}ST$ is reduced to the Putinar spectrum $\sigma\left(
\mathcal{F}_{q},X\right)  $ of the module $X$ with respect to the sheaf
$\mathcal{F}_{q}$ (see Theorem \ref{corPT}), that is,
\[
\sigma\left(  \mathcal{F}_{q},X\right)  =\sigma\left(  T,S\right)
^{-q}\text{.}%
\]
It turns out that there is a big gap between these spectra caused mainly by
the weakness of the $q$-topology. That is somewhat new that we have not seen
even in the case of Banach modules over a nilpotent (in particular abelian)
Lie algebra \cite[Proposition 5.3]{DosJOT10}. These spectra used to coincide
in the case of a Banach module, but that is not the case in the $q$-theory.
This phenomenon explains the lack of the projection property for the Taylor
spectrum of the Banach $q$-modules \cite{Dosi242}.

\textbf{Acknowledgement. }The author thanks O. Yu. Aristov for the useful
discussions, remarks and his essential help in many details. In particular,
the key idea on the topological injectivity property in the proof of the main
result belongs to him.

\section{Preliminaries\label{sPre}}

All considered vector spaces are assumed to be complex. The algebra of all
continuous linear operators on a polynormed (or locally convex) space $X$ is
denoted by $\mathcal{L}\left(  X\right)  $. If the topology of a complete
polynormed algebra $A$ is defined by means of a family of multiplicative
seminorms, then $A$ is called an \textit{Arens-Michael algebra} \cite[1.2.4]%
{Hel}.

\subsection{Triangular (co)chains\label{subsecTCs}}

Let $X_{n}$, $Y_{n}$, $Z_{n}$, $n\in\mathbb{N}$ be Fr\'{e}chet spaces and let
$X=\prod_{n}X_{n}$, $Y=\prod_{n}Y_{n}$, and $Z=\prod_{n}Z_{n}$ be their
topological direct products. A continuous linear mapping $S:X\rightarrow Y$ is
said to be \textit{a triangular operator} if it is given by a lower (or upper)
triangular operator matrix $S=\left[  S_{mn}\right]  $, $S_{mn}=0$, $m<n$,
where $S_{mn}:X_{n}\rightarrow Y_{m}$ is a continuous linear operator. By
\textit{a triangular cochain} we mean a couple of triangular operators
$X\overset{S}{\longrightarrow}Y\overset{T}{\longrightarrow}Z$ such that
$TS=0$. In this case, there are the diagonal cochains $X_{m}\overset{S_{mm}%
}{\longrightarrow}Y_{m}\overset{T_{mm}}{\longrightarrow}Z_{m}$, $m\in
\mathbb{N}$. In general, if a (co)chain complex of countable direct products
of Fr\'{e}chet spaces has a lower (or upper) triangular differentials, then it
is called \textit{a triangular complex}. The diagonal entries of those
differentials generate the related \textit{diagonal complexes}.

The following assertion was taken from \cite[Proposition 4.2]{DMedJM09} and
\cite{Aris}.

\begin{lemma}
\label{lemTri}(Aristov-Dosi) Let $X\overset{S}{\longrightarrow}%
Y\overset{T}{\longrightarrow}Z$ be a triangular cochain of Fr\'{e}chet spaces.
If all diagonal cochains $X_{m}\overset{S_{mm}}{\longrightarrow}%
Y_{m}\overset{T_{mm}}{\longrightarrow}Z_{m}$, $m\in\mathbb{N}$ are exact then
so is the original cochain. Moreover, if all diagonal operators $X_{m}%
\overset{S_{mm}}{\longrightarrow}Y_{m}$, $m\in\mathbb{N}$ are topologically
injective maps, then so is $X\overset{S}{\longrightarrow}Y$.
\end{lemma}

Let us notify that the reverse statement to one from Lemma \ref{lemTri} is not
true, the exactness of a triangular cochain does not necessarily imply the
exactness of its diagonals. That is the case of the bimodule resolution for
the Fr\'{e}chet algebra $\mathcal{F}_{q}\left(  U\right)  $ considered below
in Subsection \ref{subsecIO}.

\subsection{Topological homology and localizations\label{subSth}}

The main category of the underlying polynormed spaces for the considered
framework of topological homology is the category of Fr\'{e}chet spaces. Let
$A$ be a Fr\'{e}chet algebra, and let $X$ be a left Fr\'{e}chet $A$-module.
The category of all left Fr\'{e}chet $A$-modules is denoted by $A$%
-$\operatorname*{mod}$ whereas $\operatorname{mod}$-$A$ (resp., $A$%
-$\operatorname{mod}$-$A$) denotes the category of all right (resp., bi)
Fr\'{e}chet $A$-modules. Recall that $X$ is said to be a free left $A$-module
if $X=A\widehat{\otimes}E$ is the projective tensor product of $A$ and some
Fr\'{e}chet space $E$ equipped with the natural left $A$-module structure. A
retract in the category $A$-$\operatorname*{mod}$ of a free $A$-module is
called a projective left $A$-module. \textit{A projective resolution }of a
left $A$-module $X$ is a complex $\left(  \mathcal{P},d\right)  $ of left
projective $A$-modules with a morphism $\pi:\mathcal{P}_{0}\rightarrow X$ such
that the augmented complex $0\leftarrow X\overset{\pi}{\longleftarrow
}\mathcal{P}_{0}\overset{d_{0}}{\longleftarrow}\mathcal{P}_{1}\overset{d_{1}%
}{\longleftarrow}\cdots$ is admissible (splits as a complex of the Fr\'{e}chet spaces).

Let $B$ be a polynormed algebra and let $F:A$-$\operatorname*{mod}\rightarrow
B$-$\operatorname*{mod}$ be an additive (co)functor. By $F_{n}$ (respectively,
$F^{n}$) we denote the $n$-th (projective) derived (co)functor of $F$. By its
very definition (see \cite[3.3.1]{Hel}), $F_{n}\left(  X\right)  $
(respectively, $F^{n}\left(  X\right)  $) is just the $n$-th (co)homology of
the (co)chain complex $\left(  F\left(  \mathcal{P}\right)  ,F\left(
d\right)  \right)  $ for a projective resolution $\left(  \mathcal{P}%
,d\right)  $ of the $A$-module $X$. Taking into account that all projective
resolutions of a module are homotopy equivalent \cite[3.2.3]{Hel}, we conclude
that $F_{n}\left(  X\right)  $ (respectively, $F^{n}\left(  X\right)  $) does
not depend upon the particular choice of a projective resolution $\left(
\mathcal{P},d\right)  $ of $X$. If $F$ is the (co)functor
$X\underset{A}{\widehat{\otimes}}\circ$ (respectively, $\operatorname{Hom}%
_{A}\left(  \circ\mathfrak{,}Y\right)  $) then we write $\operatorname{Tor}%
_{n}^{A}\left(  X,\circ\right)  $ (respectively, $\operatorname{Ext}_{A}%
^{n}\left(  \circ,Y\right)  $) instead of the $n$-th derived (co)functor, as
usual. If $\operatorname{Tor}_{n}^{A}\left(  X,Y\right)  =\left\{  0\right\}
$, $n\geq0$, then we write $X\perp Y$, and say that these $A$-modules are
\textit{in the} \textit{transversality relation}. Notice that is a particular
case of the transversality relation considered in \cite[Definition 3.1.14]%
{EP}, where it is just assumed that $\operatorname{Tor}_{n}^{A}\left(
X,Y\right)  =\left\{  0\right\}  $, $n\geq1$ and $\operatorname{Tor}_{0}%
^{A}\left(  X,Y\right)  $ is Hausdorff (or it is identified with
$X\widehat{\otimes}_{A}Y$).

Now assume that $\iota:B\rightarrow A$ is a continuous algebra homomorphism
from a polynormed algebra $B$ into a Fr\'{e}chet algebra $A$. The source
algebra $B$ can be a pure algebra equipped with the finest locally convex
topology. The homomorphism $\iota:B\rightarrow A$ is said to be \textit{a
localization} if it induces the natural isomorphisms $H_{n}\left(  B,M\right)
=H_{n}\left(  A,M\right)  $, $n\geq0$ of all homology groups for every
Fr\'{e}chet $A$-bimodule $M$. In particular, the multiplication mapping
$A\overline{\otimes}_{B}A\rightarrow A$, $a_{1}\otimes a_{2}\mapsto a_{1}%
a_{2}$ (on the completed inductive tensor product) is a topological
isomorphism, and $H_{n}\left(  B,A\widehat{\otimes}A\right)
=\operatorname{Tor}_{n}^{B}\left(  A,A\right)  =\left\{  0\right\}  $ for all
$n>0$, where the algebra $A$ is considered to be a $B$-bimodule via the
homomorphism $\iota$. If $0\leftarrow B\longleftarrow\mathcal{P}$ is an
admissible projective bimodule resolution of $B$, then the application of the
functor $A\overline{\otimes}_{B}\circ\overline{\otimes}_{B}A$ to $\mathcal{P}$
augmented by the multiplication morphism results in the following chain
complex
\begin{equation}
0\leftarrow A\longleftarrow A\overline{\otimes}_{B}\mathcal{P}\overline
{\otimes}_{B}A. \label{TAc}%
\end{equation}
If (\ref{TAc}) is admissible for some (actually for every) $\mathcal{P}$, then
certainly $\iota:B\rightarrow A$ is a localization called \textit{an absolute
localization }(or $A$ is \textit{stably flat over} $B$). In this case,
$H_{n}\left(  B,M\right)  =H_{n}\left(  A,M\right)  $, $n\geq0$ hold for every
$\widehat{\otimes}$-bimodule $M$ (not necessarily a Fr\'{e}chet one). If $A$
and $A\overline{\otimes}_{B}\mathcal{P}\overline{\otimes}_{B}A$ consists of
nuclear Fr\'{e}chet spaces, then $\iota:B\rightarrow A$ is a localization (see
\cite[Proposition 1.6]{Tay2}) whenever (\ref{TAc}) is exact.

\subsection{Arens-Michael envelope\label{subsecAME}}

The \textit{Arens-Michael envelope} \cite[5.2.21]{Hel} of a polynormed algebra
$A$ is given by a couple $\left(  \widetilde{A},\omega\right)  $ of an
Arens-Michael algebra $\widetilde{A}$ and a continuous algebra homomorphism
$\omega:A\rightarrow\widetilde{A}$ with the following universal
projective\ property. Every continuous algebra homomorphism $\pi:A\rightarrow
B$ into an Arens-Michael algebra $B$ admits a unique continuous algebra
homomorphism $\widetilde{\pi}:\widetilde{A}\rightarrow B$ such that
$\widetilde{\pi}\cdot\omega=\pi$. Since an Arens-Michael algebra is an inverse
limit of Banach algebras \cite[5.2.10]{Hel}, the latter universal projective
property can be given in terms of the homomorphisms $A\rightarrow B$ into
Banach algebra $B$. The set of all continuous characters of $A$ is denoted by
$\operatorname{Spec}\left(  A\right)  $. If $\lambda\in\operatorname{Spec}%
\left(  A\right)  $ then the algebra homomorphism $\lambda:A\rightarrow
\mathbb{C}$ defines (Banach) $A$-module structure on $\mathbb{C}$ via pull
back along $\lambda$. This module called \textit{a trivial module} is denoted
by $\mathbb{C}\left(  \lambda\right)  $, which is an $\widetilde{A}$-module
automatically. Thus $\operatorname{Spec}\left(  A\right)  =\operatorname{Spec}%
\left(  \widetilde{A}\right)  $.

\subsection{The $q$-topology of the complex plane\label{subsecQT}}

Let us fix $q\in\mathbb{C}\backslash\left\{  0\right\}  $ with $\left\vert
q\right\vert <1$. A subset $S\subseteq\mathbb{C}$ is called a $q$%
\textit{-spiraling set }if it contains the origin and $qS\subseteq S$. An open
subset $U\subseteq\mathbb{C}$ is said to be \textit{a }$q$\textit{-open set}
if it is an open subset of $\mathbb{C}$ in the standard topology, which is
also a $q$-spiraling set. The whole plane $\mathbb{C}$ is $q$-open, and the
empty set is assumed to be $q$-open set. The family of all $q$-open subsets
defines a new topology $\mathfrak{q}$ in $\mathbb{C}$, which is weaker than
the original standard topology of the complex plane. Every open disk $B\left(
0,r\right)  $ centered at the origin is a $q$-open set.

Notice that $\left\{  0\right\}  $ is a generic point of the topological space
$\left(  \mathbb{C},\mathfrak{q}\right)  $ being dense in the whole plane. One
can easily prove that if $x\in\mathbb{C}\backslash\left\{  0\right\}  $ then
its closure in $\left(  \mathbb{C},\mathfrak{q}\right)  $ is given by
\[
\left\{  x\right\}  ^{-}=\left\{  q^{-k}x:k\in\mathbb{Z}_{+}\right\}  ,
\]
whereas $\left\{  x\right\}  _{q}=\left\{  q^{k}x:k\in\mathbb{Z}_{+}\right\}
$ defines its $q$-hull. Thus the topological space $\left(  \mathbb{C}%
,\mathfrak{q}\right)  $ satisfies the axiom $T_{0}$, and it turns out to be an
irreducible topological space \cite[II, 4.1]{BurComA}. One can easily see that
$\left(  \mathbb{C},\mathfrak{q}\right)  $ is not quasicompact, in particular,
it is not Noetherian.

If $K\subseteq\mathbb{C}$ is a compact subset in the standard topology of
$\mathbb{C}$ then it turns out to be a quasicompact subset of $\left(
\mathbb{C},\mathfrak{q}\right)  $, but not necessarily a $\mathfrak{q}$-closed
subset. In particular, all disks (open or closed) centered at the origin are
quasicompact (nonclosed) subsets of $\left(  \mathbb{C},\mathfrak{q}\right)
$. They are all dense subsets of $\left(  \mathbb{C},\mathfrak{q}\right)  $.
Every closure $\left\{  x\right\}  ^{-}$ of a point $x\in\mathbb{C}$ is not quasicompact.

Let $\mathcal{O}$ be the standard Fr\'{e}chet sheaf of stalks of the
holomorphic functions on $\mathbb{C}$ and let $\operatorname{id}%
:\mathbb{C\rightarrow}\left(  \mathbb{C},\mathfrak{q}\right)  $ be the
identity mapping, which is continuous. The direct image $\operatorname{id}%
_{\ast}\mathcal{O}$ of the sheaf $\mathcal{O}$ along the identity mapping
turns out to be a Fr\'{e}chet algebra sheaf on the topological space $\left(
\mathbb{C},\mathfrak{q}\right)  $. For every $q$-open set $U$ and its
quasicompact subset $K\subseteq U$ we define the related seminorm $\left\Vert
f\right\Vert _{K}=\sup\left\vert f\left(  K\right)  \right\vert $,
$f\in\mathcal{O}\left(  U\right)  $ on the algebra $\mathcal{O}\left(
U\right)  $. The family $\left\{  \left\Vert \cdot\right\Vert _{K}\right\}  $
of seminorms over all $q$-compact subsets $K\subseteq U$ (that is, $K=K_{q}$)
defines the same original Fr\'{e}chet topology of $\mathcal{O}\left(
U\right)  $. The stalk of $\operatorname{id}_{\ast}\mathcal{O}$ at zero is the
same original stalk $\mathcal{O}_{0}$, whereas $\left(  \operatorname{id}%
_{\ast}\mathcal{O}\right)  _{\lambda}=\mathcal{O}\left(  \left\{
\lambda\right\}  _{q}\right)  =\mathcal{O}_{0}+\sum_{n\in\mathbb{Z}_{+}%
}\mathcal{O}_{q^{n}\lambda}$ at every $\lambda\in\mathbb{C}\backslash\left\{
0\right\}  $. The algebra $\left(  \operatorname{id}_{\ast}\mathcal{O}\right)
_{\lambda}$ is not local for $\lambda\in\mathbb{C}\backslash\left\{
0\right\}  $. The ideal of those stalks $\left\langle U,f\right\rangle
\in\left(  \operatorname{id}_{\ast}\mathcal{O}\right)  _{\lambda}$ with
$f\left(  \left\{  \lambda\right\}  _{q}\right)  =\left\{  0\right\}  $ is not
maximal (see \cite{Dosi24}). Everywhere below, we use the same notation
$\mathcal{O}$ instead of $\operatorname{id}_{\ast}\mathcal{O}$.

\subsection{The Fr\'{e}chet space sheaves $\mathcal{O}\left(  d\right)  $ on
$\left(  \mathbb{C},\mathfrak{q}\right)  $\label{subsecFSO}}

If $U\subseteq\left(  \mathbb{C},\mathfrak{q}\right)  $ is a $q$-open subset,
then it contains the origin and we put%
\[
\mathcal{O}\left(  d\right)  \left(  U\right)  =\left\{  f\left(  z\right)
\in\mathcal{O}\left(  U\right)  :z^{-d}f\left(  z\right)  \in\mathcal{O}%
\left(  U\right)  \right\}
\]
to be a closed ideal of $\mathcal{O}\left(  U\right)  $, where $d\in
\mathbb{Z}_{+}$. Notice that $\mathcal{O}\left(  0\right)  \left(  U\right)
=\mathcal{O}\left(  U\right)  $, and $\mathcal{O}\left(  d\right)  \left(
U\right)  $ consists of those $f\left(  z\right)  \in\mathcal{O}\left(
U\right)  $ such that $f\left(  0\right)  =f^{\prime}\left(  0\right)
=\cdots=f^{\left(  d-1\right)  }\left(  0\right)  =0$. The following assertion
was taken from \cite[Lemma 2.2]{Dosi25}. The projection $P_{d}$ from its proof
will be used below.

\begin{lemma}
\label{lemMD}The ideal $\mathcal{O}\left(  d\right)  \left(  U\right)  $ is
the principal ideal of $\mathcal{O}\left(  U\right)  $ generated by $z^{d}$,
that is, $\mathcal{O}\left(  d\right)  \left(  U\right)  =z^{d}\mathcal{O}%
\left(  U\right)  $. The linear mapping $z^{-d}:\mathcal{O}\left(  d\right)
\left(  U\right)  \rightarrow\mathcal{O}\left(  U\right)  $, $f\left(
z\right)  \mapsto z^{-d}f\left(  z\right)  $ implements a topological
isomorphism of the Fr\'{e}chet spaces preserving the multiplication operator
by $z$. Moreover,
\[
\mathcal{O}\left(  U\right)  =\mathcal{O}\left(  d\right)  \left(  U\right)
\oplus\mathbb{C}1\oplus\mathbb{C}z\oplus\cdots\oplus\mathbb{C}z^{d-1}%
\]
is a topological direct sum of the closed subspaces $\mathcal{O}\left(
d\right)  \left(  U\right)  $ and $\mathbb{C}1\oplus\mathbb{C}z\oplus
\cdots\oplus\mathbb{C}z^{d-1}$.
\end{lemma}

\begin{proof}
If $f\in\mathcal{O}\left(  d\right)  \left(  U\right)  $, $d\geq1$, then
$f\left(  z\right)  =z\left(  z^{-1}f\left(  z\right)  \right)  $ with
$z^{-1}f\left(  z\right)  \in\mathcal{O}\left(  d-1\right)  \left(  U\right)
$. Thus $\mathcal{O}\left(  d\right)  \left(  U\right)  \subseteq
z\mathcal{O}\left(  d-1\right)  \left(  U\right)  $. One can easily verify
that
\[
P_{d}\in\mathcal{L}\left(  \mathcal{O}\left(  U\right)  \right)  ,\quad
P_{d}\left(  f\right)  \left(  z\right)  =f\left(  z\right)  -\sum_{i=0}%
^{d-1}\frac{f^{\left(  i\right)  }\left(  0\right)  }{i!}z^{i}%
\]
is a continuous projection onto $\mathcal{O}\left(  d\right)  \left(
U\right)  $ such that $\ker\left(  P_{d}\right)  =\oplus_{i=0}^{d-1}%
\mathbb{C}z^{i}$ is a finite dimensional polynomial subspace. Furthermore,
\begin{align*}
P_{d}\left(  zf\left(  z\right)  \right)   &  =zf\left(  z\right)  -\sum
_{i=0}^{d-1}\frac{\left(  zf\right)  ^{\left(  i\right)  }\left(  0\right)
}{i!}z^{i}=zf\left(  z\right)  -\sum_{i=1}^{d-1}\frac{f^{\left(  i-1\right)
}\left(  0\right)  }{\left(  i-1\right)  !}z^{i}=z\left(  f\left(  z\right)
-\sum_{i=0}^{d-2}\frac{f^{\left(  i\right)  }\left(  0\right)  }{i!}%
z^{i}\right) \\
&  =zP_{d-1}\left(  f\left(  z\right)  \right)
\end{align*}
for all $f\in\mathcal{O}\left(  U\right)  $. It follows that $z\mathcal{O}%
\left(  d-1\right)  \left(  U\right)  \subseteq\mathcal{O}\left(  d\right)
\left(  U\right)  \subseteq z\mathcal{O}\left(  d-1\right)  \left(  U\right)
$, that is, $\mathcal{O}\left(  d\right)  \left(  U\right)  =z\mathcal{O}%
\left(  d-1\right)  \left(  U\right)  =z^{d}\mathcal{O}\left(  U\right)  $ is
the principal ideal generated by the monomial $z^{d}$. Moreover,
$\mathcal{O}\left(  U\right)  =\left(  z^{d}\mathcal{O}\left(  U\right)
\right)  \oplus\left(  \mathbb{C}z^{d-1}\oplus\cdots\oplus\mathbb{C}1\right)
$ is a topological direct sum of the closed subspaces.

Finally, the multiplication operator $z^{d}:\mathcal{O}\left(  U\right)
\rightarrow\mathcal{O}\left(  d\right)  \left(  U\right)  $ is a continuous
bijection of the Fr\'{e}chet spaces. By the Open Mapping Theorem, its inverse
mapping $z^{-d}:\mathcal{O}\left(  d\right)  \left(  U\right)  \rightarrow
\mathcal{O}\left(  U\right)  $, $f\left(  z\right)  \mapsto z^{-d}f\left(
z\right)  $ is continuous, and $z\left(  z^{-d}\left(  f\left(  z\right)
\right)  \right)  =z^{-d}\left(  zf\left(  z\right)  \right)  $,
$f\in\mathcal{O}\left(  d\right)  \left(  U\right)  $.
\end{proof}

Thus $U\mapsto\mathcal{O}\left(  d\right)  \left(  U\right)  $ defines a
Fr\'{e}chet space sheaf on $\mathbb{C}$ denoted by $\mathcal{O}\left(
d\right)  $. It is an isomorphic copy of the original sheaf $\mathcal{O}$
obtained by means of the $d$-shift, that is, $\mathcal{O}\left(  d\right)
=z^{d}\mathcal{O}$ is a (free) coherent $\mathcal{O}$-module \cite[4]{EP}. By
Lemma \ref{lemMD}, we deduce that
\[
\mathcal{O}=\mathcal{O}\left(  d\right)  \oplus\mathbb{C}1\oplus
\mathbb{C}z\oplus\cdots\oplus\mathbb{C}z^{d-1}%
\]
is a direct sum of the Fr\'{e}chet sheaves on $\left(  \mathbb{C}%
,\mathfrak{q}\right)  $, where $\mathbb{C}z^{i}$, $0\leq i<d$ are the constant sheaves.

Along with the projection $P_{d}$ from the proof of Lemma \ref{lemMD}, let us
also introduce the following (translation) operator
\begin{equation}
\Delta_{q}\in\mathcal{L}\left(  \mathcal{O}\left(  U\right)  \right)
,\quad\Delta_{q}\left(  f\right)  \left(  z\right)  =f\left(  qz\right)  .
\label{Delq}%
\end{equation}
Once can easily verify that $\Delta_{q}P_{d}=P_{d}\Delta_{q}$. Finally, for a
holomorphic function $\theta\left(  z,w\right)  $ about the origin in two
complex variables $z$ and $w$, we define
\[
\theta_{z}\left(  z,w\right)  =\frac{\theta\left(  z,w\right)  -\theta\left(
0,w\right)  }{z},\quad\theta_{w}\left(  z,w\right)  =\frac{\theta\left(
z,w\right)  -\theta\left(  z,0\right)  }{w}%
\]
to be the difference operators.

\subsection{Arens-Michael envelope of $\mathfrak{A}_{q}$ and the character
space $\mathbb{C}_{xy}$\label{subsecAMEC}}

As above let $\mathfrak{A}_{q}$ be the quantum plane (or just $q$-plane)
generated by $x$ and $y$ modulo the relation $xy=q^{-1}yx$, where
$q\in\mathbb{C}\backslash\left\{  0,1\right\}  $, $\left\vert q\right\vert
\leq1$. Thus $\mathfrak{A}_{q}$ is the vector space of all ordered polynomials
$f=\sum_{i,k\in\mathbb{Z}_{+}}a_{ik}x^{i}y^{k}$, $a_{ik}\in\mathbb{C}$, which
can be equipped with the following submultiplicative seminorms $\left\Vert
f\right\Vert _{\rho}=\sum_{i,k\in\mathbb{Z}_{+}}\left\vert a_{ik}\right\vert
\rho^{i+k}<\infty$, $\rho>0$. The Banach algebra completion of the seminormed
space $\left(  \mathfrak{A}_{q},\left\Vert \cdot\right\Vert _{\rho}\right)  $
is denoted by $\mathcal{A}\left(  \rho\right)  $. The Arens-Michael envelope
$\mathcal{O}_{q}\left(  \mathbb{C}^{2}\right)  $ of $\mathfrak{A}_{q}$ turns
out to be an inverse limit of the Banach algebras $\mathcal{A}\left(
\rho\right)  $, $\rho>0$ (see \cite{Pir}, \cite{Dosi24}). The closed two-sided
ideal in $\mathcal{O}_{q}\left(  \mathbb{C}^{2}\right)  $ generated by $xy$ is
denoted by $\mathcal{I}_{xy}$. The following assertion was proved in
\cite{Dosi24}.

\begin{proposition}
\label{tO}If $\left\vert q\right\vert \neq1$ then the algebra $\mathcal{O}%
_{q}\left(  \mathbb{C}^{2}\right)  $ is commutative modulo its Jacobson
radical $\operatorname{Rad}\mathcal{O}_{q}\left(  \mathbb{C}^{2}\right)  $.
Moreover,
\[
\operatorname{Rad}\mathcal{O}_{q}\left(  \mathbb{C}^{2}\right)  =\cap\left\{
\ker\left(  \lambda\right)  :\lambda\in\operatorname{Spec}\left(
\mathcal{O}_{q}\left(  \mathbb{C}^{2}\right)  \right)  \right\}
=\mathcal{I}_{xy},
\]%
\[
\operatorname{Spec}\left(  \mathcal{O}_{q}\left(  \mathbb{C}^{2}\right)
\right)  =\operatorname{Spec}\left(  \mathcal{O}_{q}\left(  \mathbb{C}%
^{2}\right)  /\mathcal{I}_{xy}\right)  =\mathbb{C}_{xy},
\]
where $\mathbb{C}_{x}=\mathbb{C\times}\left\{  0\right\}  \subseteq
\mathbb{C}^{2}$, $\mathbb{C}_{y}=\left\{  0\right\}  \times\mathbb{C\subseteq
C}^{2}$ and $\mathbb{C}_{xy}=\mathbb{C}_{x}\cup\mathbb{C}_{y}$. If $\left\vert
q\right\vert =1$ then $\mathcal{I}_{xy}$ contains all quasinilpotents of
$\mathcal{O}_{q}\left(  \mathbb{C}^{2}\right)  $, and none of the ordered
monomials $x^{i}y^{k}$, $i,k\in\mathbb{N}$ is quasinilpotent.
\end{proposition}

Thus the character space $\mathbb{C}_{xy}$\ of the Fr\'{e}chet algebra
$\mathcal{O}_{q}\left(  \mathbb{C}^{2}\right)  $ is the union $\mathbb{C}%
_{x}\cup\mathbb{C}_{y}$ of two copies of the complex plane whose intersection
consists of the trivial character that corresponds to $\left(  0,0\right)  $.
One can swap the role of $x$ and $y$ having $\mathbb{C}_{yx}$ as the character
space of $\mathcal{O}_{q^{-1}}\left(  \mathbb{C}^{2}\right)  $. Based on
Proposition \ref{tO}, we use the notation $\mathcal{O}_{q}\left(
\mathbb{C}_{xy}\right)  $ instead of $\mathcal{O}_{q}\left(  \mathbb{C}%
^{2}\right)  $ whenever $\left\vert q\right\vert \neq1$. The character space
$\mathbb{C}_{xy}$ being the union $\mathbb{C}_{x}\cup\mathbb{C}_{y}$ is
equipped with the final topology so that both embeddings
\[
\left(  \mathbb{C}_{x},\mathfrak{q}\right)  \hookrightarrow\mathbb{C}%
_{xy}\hookleftarrow\left(  \mathbb{C}_{y},\mathfrak{q}\right)
\]
are continuous, which is called the $\mathfrak{q}$\textit{-topology of}
$\mathbb{C}_{xy}$. The topology base in $\mathbb{C}_{xy}$ consists of all open
subsets $U=U_{x}\cup U_{y}$ with $q$-open sets $U_{x}\subseteq\mathbb{C}_{x}$
and $U_{y}\subseteq\mathbb{C}_{y}$. In this case, $\mathbb{C}_{xy}%
=\mathbb{C}_{x}\cup\mathbb{C}_{y}$ is the union of two irreducible subsets,
whose intersection is the unique generic point.

\section{The $x$-diagonal and $y$-diagonal Fr\'{e}chet space
complexes\label{sectionLRq}}

In this section we introduce the $x$-diagonal and $y$-diagonal Fr\'{e}chet
space complexes mainly on the level of the formal power series, and
investigate their exactness. They play the key roles in the localization
problems of the structure sheaves of the noncommutative analytic spaces
$\mathbb{C}_{x}$, $\mathbb{C}_{y}$ and their intersection point $\left\{
\left(  0,0\right)  \right\}  $ within the formal geometry of $\mathfrak{A}%
_{q}$ considered in \cite{Dosi25}.

\subsection{The $y$-diagonal formal $q$-complex\label{subsecLDFC}}

Let $\mathbb{C}\left[  \left[  x_{1},x_{2}\right]  \right]  $ be the
commutative Fr\'{e}chet algebra of all formal power series in the independent
variables $x_{1}$, $x_{2}$, and let $\mathbb{C}\left[  \left[  y\right]
\right]  $ be another Fr\'{e}chet algebra of formal power series in $y$. The
right (or left) multiplication operator $R_{y}$ by $y$ on $\mathbb{C}\left[
\left[  y\right]  \right]  $ is acting as the shift operator. Since
$\mathbb{C}\left[  \left[  x_{1},x_{2}\right]  \right]  \left[  \left[
y\right]  \right]  =\mathbb{C}\left[  \left[  x_{1},x_{2}\right]  \right]
\widehat{\otimes}\mathbb{C}\left[  \left[  y\right]  \right]  $, it follows
that $1\otimes R_{y}$ is also acting as the following shift operator
\[
N_{y}\in\mathcal{L}\left(  \mathbb{C}\left[  \left[  x_{1},x_{2}\right]
\right]  \left[  \left[  y\right]  \right]  \right)  \text{,\quad}N_{y}\left(
f\right)  =\sum_{n\geq1}f_{n-1}\left(  x_{1},x_{2}\right)  y^{n},
\]
where $f=\sum_{n\geq0}f_{n}\left(  x_{1},x_{2}\right)  y^{n}$ represents an
element of $\mathbb{C}\left[  \left[  x_{1},x_{2}\right]  \right]  \left[
\left[  y\right]  \right]  $. The variables $x_{i}$ are acting as the
multiplication operators on $\mathbb{C}\left[  \left[  x_{1},x_{2}\right]
\right]  $, and they can be inflated to the space $\mathbb{C}\left[  \left[
x_{1},x_{2}\right]  \right]  \left[  \left[  y\right]  \right]  $ as
$x_{i}\otimes1$. For brevity we use the same notations $x_{i}$ instead of
$x_{i}\otimes1$, respectively. There is also a well defined continuous linear
operator (the diagonal projection)
\[
\pi_{y}:\mathbb{C}\left[  \left[  x_{1},x_{2}\right]  \right]  \left[  \left[
y\right]  \right]  \longrightarrow\mathbb{C}\left[  \left[  x\right]  \right]
,\quad\pi_{y}\left(  \sum_{n\geq0}h_{n}\left(  x_{1},x_{2}\right)
y^{n}\right)  =h_{0}\left(  x,x\right)  .
\]
As above we fix $q\in\mathbb{C}\backslash\left\{  0,1\right\}  $, and put
\[
D_{y,q}\in\mathcal{L}\left(  \mathbb{C}\left[  \left[  x_{1},x_{2}\right]
\right]  \left[  \left[  y\right]  \right]  \right)  ,\quad D_{y,q}\left(
f\right)  =\sum_{n\geq0}q^{n}f_{n}\left(  x_{1},x_{2}\right)  y^{n}%
\]
to be the $q$-diagonal operator, that is, $D_{y,q}=\prod\limits_{n}q^{n}$. Let
us define the following Fr\'{e}chet space complex $\mathfrak{D}_{q}\left(
y\right)  $:
\[
0\rightarrow\mathbb{C}\left[  \left[  x_{1},x_{2}\right]  \right]  \left[
\left[  y\right]  \right]  \overset{d_{y}^{0}}{\longrightarrow}%
\begin{array}
[c]{c}%
\mathbb{C}\left[  \left[  x_{1},x_{2}\right]  \right]  \left[  \left[
y\right]  \right] \\
\oplus\\
\mathbb{C}\left[  \left[  x_{1},x_{2}\right]  \right]  \left[  \left[
y\right]  \right]
\end{array}
\overset{d_{y}^{1}}{\longrightarrow}\mathbb{C}\left[  \left[  x_{1}%
,x_{2}\right]  \right]  \left[  \left[  y\right]  \right]  \overset{\pi
_{y}}{\longrightarrow}\mathbb{C}\left[  \left[  x\right]  \right]
\rightarrow0,
\]
whose differentials are acting by%
\begin{equation}
d_{y}^{0}=\left[
\begin{array}
[c]{c}%
N_{y}\\
x_{2}-qx_{1}D_{y,q}%
\end{array}
\right]  ,\quad d_{y}^{1}=\left[
\begin{array}
[c]{cc}%
x_{2}-x_{1}D_{y,q} & -N_{y}%
\end{array}
\right]  . \label{Fdif}%
\end{equation}
Notice that $\left(  x_{2}-x_{1}D_{y,q}\right)  f=\sum_{n\geq0}\left(
x_{2}-q^{n}x_{1}\right)  f_{n}\left(  x_{1},x_{2}\right)  y^{n}=\left(
\prod\limits_{n}\left(  x_{2}-q^{n}x_{1}\right)  \right)  f$ for all
$f\in\mathbb{C}\left[  \left[  x_{1},x_{2}\right]  \right]  \left[  \left[
y\right]  \right]  $, that is, $x_{2}-x_{1}D_{y,q}=\prod\limits_{n}\left(
x_{2}-q^{n}x_{1}\right)  $. In a similar way, we have $x_{2}-qx_{1}%
D_{y,q}=\prod\limits_{n}\left(  x_{2}-q^{n+1}x_{1}\right)  $. Then
\begin{align*}
d_{y}^{1}d_{y}^{0}f  &  =\left(  \prod\limits_{n}\left(  x_{2}-q^{n}%
x_{1}\right)  \right)  \sum_{n\geq1}f_{n-1}\left(  x_{1},x_{2}\right)
y^{n}-\left(  1\otimes R_{y}\right)  \sum_{n\geq0}\left(  x_{2}-q^{n+1}%
x_{1}\right)  f_{n}\left(  x_{1},x_{2}\right)  y^{n}\\
&  =\sum_{n\geq1}\left(  x_{2}-q^{n}x_{1}\right)  f_{n-1}\left(  x_{1}%
,x_{2}\right)  y^{n}-\sum_{n\geq0}\left(  x_{2}-q^{n+1}x_{1}\right)
f_{n}\left(  x_{1},x_{2}\right)  y^{n+1}=0
\end{align*}
for every $f\in\mathbb{C}\left[  \left[  x_{1},x_{2}\right]  \right]  \left[
\left[  y\right]  \right]  $. Moreover,
\begin{align*}
\pi_{y}d_{y}^{1}h  &  =\pi_{y}\left(  \sum_{n\geq0}\left(  x_{2}-q^{n}%
x_{1}\right)  f_{n}\left(  x_{1},x_{2}\right)  y^{n}-\sum_{n\geq1}%
g_{n-1}\left(  x_{1},x_{2}\right)  y^{n}\right) \\
&  =\left(  x_{2}-x_{1}\right)  f_{0}\left(  x_{1},x_{2}\right)
|_{x_{1}=x_{2}=x}=0
\end{align*}
for all $h=\left[
\begin{array}
[c]{c}%
f\\
g
\end{array}
\right]  \in\mathbb{C}\left[  \left[  x_{1},x_{2}\right]  \right]  \left[
\left[  y\right]  \right]  ^{\oplus2}$. It follows that $d_{y}^{1}d_{y}^{0}=0$
and $\pi_{y}d_{y}^{1}=0$, which means that $\mathfrak{D}_{q}\left(  y\right)
$ is a cochain complex of the Fr\'{e}chet spaces called \textit{the }%
$y$-\textit{diagonal formal }$q$\textit{-complex}.

\subsection{The $x$-diagonal formal $q$-complex\label{subsecRDC}}

In a similar way, we consider the commutative Fr\'{e}chet algebra
$\mathbb{C}\left[  \left[  y_{1},y_{2}\right]  \right]  $ of all formal power
series in the independent variables $y_{1}$,$y_{2}$, and the Fr\'{e}chet space
$\left[  \left[  x\right]  \right]  \mathbb{C}\left[  \left[  y_{1}%
,y_{2}\right]  \right]  $ (to be $\mathbb{C}\left[  \left[  x\right]  \right]
\widehat{\otimes}\mathbb{C}\left[  \left[  y_{1},y_{2}\right]  \right]  $)
with the related operators
\begin{align*}
N_{x},D_{x,q}  &  \in\mathcal{L}\left(  \left[  \left[  x\right]  \right]
\mathbb{C}\left[  \left[  y_{1},y_{2}\right]  \right]  \right)  ,\quad
N_{x}\left(  f\right)  =\sum_{n\geq1}x^{n}f_{n-1}\left(  y_{1},y_{2}\right)
,\text{ }D_{x,q}\left(  f\right)  =\sum_{n\geq0}x^{n}q^{n}f_{n}\left(
y_{1},y_{2}\right)  ,\\
y_{1}-qy_{2}D_{x,q}  &  =\prod\limits_{n}\left(  y_{1}-q^{n+1}y_{2}\right)
,\quad y_{2}D_{x,q}-y_{1}=\prod\limits_{n}\left(  q^{n}y_{2}-y_{1}\right)  .
\end{align*}
Let us define the following\textit{ }$x$-\textit{diagonal formal }%
$q$\textit{-complex} $\mathfrak{D}_{q}\left(  x\right)  $ of the Fr\'{e}chet
spaces
\[
0\rightarrow\left[  \left[  x\right]  \right]  \mathbb{C}\left[  \left[
y_{1},y_{2}\right]  \right]  \overset{d_{x}^{0}}{\longrightarrow}%
\begin{array}
[c]{c}%
\left[  \left[  x\right]  \right]  \mathbb{C}\left[  \left[  y_{1}%
,y_{2}\right]  \right] \\
\oplus\\
\left[  \left[  x\right]  \right]  \mathbb{C}\left[  \left[  y_{1}%
,y_{2}\right]  \right]
\end{array}
\overset{d_{x}^{1}}{\longrightarrow}\left[  \left[  x\right]  \right]
\mathbb{C}\left[  \left[  y_{1},y_{2}\right]  \right]  \overset{\pi
_{x}}{\longrightarrow}\mathbb{C}\left[  \left[  y\right]  \right]
\rightarrow0,
\]
whose differentials are acting by%
\[
d_{x}^{0}=\left[
\begin{array}
[c]{c}%
y_{1}-qy_{2}D_{x,q}\\
N_{x}%
\end{array}
\right]  ,\quad d_{x}^{1}=\left[
\begin{array}
[c]{cc}%
N_{x} & y_{2}D_{x,q}-y_{1}%
\end{array}
\right]  ,
\]
and the morphism $\pi_{x}$ has the same action $\pi_{x}\left(  \sum_{n}%
x^{n}h_{n}\left(  y_{1},y_{2}\right)  \right)  =h_{0}\left(  y,y\right)  $.
Notice that
\begin{align*}
\pi_{x}d_{x}^{1}\left[
\begin{array}
[c]{c}%
f\\
g
\end{array}
\right]   &  =\pi_{x}\left(  N_{x}f+\left(  y_{2}D_{x,q}-y_{1}\right)
g\right) \\
&  =\pi_{x}\left(  \sum_{n\geq1}x^{n}f_{n-1}\left(  y_{1},y_{2}\right)
+\sum_{n\geq0}x^{n}\left(  y_{2}q^{n}-y_{1}\right)  g_{n}\left(  y_{1}%
,y_{2}\right)  \right) \\
&  =\left(  y_{2}-y_{1}\right)  g_{0}\left(  y_{1},y_{2}\right)
|_{y_{1}=y_{2}=y}=0.
\end{align*}
In the same way, we have
\begin{align*}
d_{x}^{1}d_{x}^{0}h  &  =N_{x}\left(  y_{1}-qy_{2}D_{x,q}\right)  h+\left(
y_{2}D_{x,q}-y_{1}\right)  N_{x}h\\
&  =N_{x}\sum_{n\geq0}x^{n}\left(  y_{1}-q^{n+1}y_{2}\right)  h_{n}\left(
y_{1},y_{2}\right)  -\sum_{n\geq0}x^{n+1}\left(  y_{1}-q^{n+1}y_{2}\right)
h_{n}\left(  y_{1},y_{2}\right) \\
&  =\sum_{n\geq1}x^{n}\left(  y_{1}-q^{n}y_{2}\right)  h_{n-1}\left(
y_{1},y_{2}\right)  -\sum_{n\geq1}x^{n}\left(  y_{1}-q^{n}y_{2}\right)
h_{n-1}\left(  y_{1},y_{2}\right)  =0.
\end{align*}
Thus $\mathfrak{D}_{q}\left(  x\right)  $ is a cochain complex of the
Fr\'{e}chet spaces indeed.

\subsection{The exactness of the diagonal formal complexes}

Actually, both diagonal formal $q$-complexes are exact.

\begin{lemma}
\label{lemLDc1}The diagonal formal $q$-complexes $\mathfrak{D}_{q}\left(
y\right)  $ and $\mathfrak{D}_{q}\left(  x\right)  $ are exact.
\end{lemma}

\begin{proof}
We provide the proof in the case of the $y$-diagonal formal $q$-complex
$\mathfrak{D}_{q}\left(  y\right)  $. The same argument is applicable in the
case of $\mathfrak{D}_{q}\left(  x\right)  $ (see below Remark \ref{remLDc1}).
If $h\in\mathbb{C}\left[  \left[  x_{1},x_{2}\right]  \right]  \left[  \left[
y\right]  \right]  $ with $d_{y}^{0}h=0$, then $N_{y}h=\sum_{n\geq1}%
h_{n-1}\left(  x_{1},x_{2}\right)  y^{n}=0$, which means that $h_{n}\left(
x_{1},x_{2}\right)  =0$ for all $n$. Thus $\ker\left(  d_{y}^{0}\right)
=\left\{  0\right\}  $. Pick $\left[
\begin{array}
[c]{c}%
f\\
g
\end{array}
\right]  \in\ker\left(  d_{y}^{1}\right)  $, that is, $\prod\limits_{n}\left(
x_{2}-q^{n}x_{1}\right)  f=\left(  1\otimes R_{y}\right)  g$. Then $\left(
x_{2}-x_{1}\right)  f_{0}\left(  x_{1},x_{2}\right)  =0$ and $\left(
x_{2}-q^{n}x_{1}\right)  f_{n}\left(  x_{1},x_{2}\right)  =g_{n-1}\left(
x_{1},x_{2}\right)  $, $n\geq1$. Since the commutative algebra $\mathbb{C}%
\left[  \left[  x_{1},x_{2}\right]  \right]  $ has no zero divisors
\cite[IV.5.3, Theorem 1]{BurA2}, it follows that $f_{0}\left(  x_{1}%
,x_{2}\right)  =0$. Put $h=\sum_{n\geq0}f_{n+1}\left(  x_{1},x_{2}\right)
y^{n}\in\mathbb{C}\left[  \left[  x_{1},x_{2}\right]  \right]  \left[  \left[
y\right]  \right]  $. Then
\begin{align*}
d_{y}^{0}h  &  =\left[
\begin{array}
[c]{c}%
N_{y}h\\
\left(  x_{2}-qx_{1}D_{y,q}\right)  h
\end{array}
\right]  =\left[
\begin{array}
[c]{c}%
\sum_{n\geq1}f_{n}\left(  x_{1},x_{2}\right)  y^{n}\\
\sum_{n\geq0}\left(  x_{2}-q^{n+1}x_{1}\right)  f_{n+1}\left(  x_{1}%
,x_{2}\right)  y^{n}%
\end{array}
\right] \\
&  =\left[
\begin{array}
[c]{c}%
f\\
\sum_{n\geq0}g_{n}\left(  x_{1},x_{2}\right)  y^{n}%
\end{array}
\right]  =\left[
\begin{array}
[c]{c}%
f\\
g
\end{array}
\right]  ,
\end{align*}
that is, $\ker\left(  d_{y}^{1}\right)  =\operatorname{im}\left(  d_{y}%
^{0}\right)  $. Further, take $h\in\ker\left(  \pi_{y}\right)  $, that is,
$h_{0}\left(  x,x\right)  =0$. Based on Taylor's formula for the power series
\cite[IV.5.8, Proposition 9]{BurA2}, we deduce that
\[
h_{0}\left(  x_{1},x_{1}+t\right)  -h_{0}\left(  x_{1},x_{1}\right)
=\sum_{m\geq1}\frac{1}{m!}\frac{\partial^{m}h_{0}}{\partial x_{2}^{m}}\left(
x_{1},x_{1}\right)  t^{m}=t\left(  \sum_{m\geq1}\frac{1}{m!}\frac{\partial
^{m}h_{0}}{\partial x_{2}^{m}}\left(  x_{1},x_{1}\right)  t^{m-1}\right)  .
\]
It follows that
\[
h_{0}\left(  x_{1},x_{2}\right)  =h_{0}\left(  x_{1},x_{1}+\left(  x_{2}%
-x_{1}\right)  \right)  -h_{0}\left(  x_{1},x_{1}\right)  =\left(  x_{2}%
-x_{1}\right)  f_{0}\left(  x_{1},x_{2}\right)
\]
for some $f_{0}\left(  x_{1},x_{2}\right)  \in\mathbb{C}\left[  \left[
x_{1},x_{2}\right]  \right]  $. Put $g=-\sum_{n\geq0}h_{n+1}\left(
x_{1},x_{2}\right)  y_{1}^{n}$ and $f=f_{0}\left(  x_{1},x_{2}\right)  $.
Then
\begin{align*}
d_{y}^{1}\left[
\begin{array}
[c]{c}%
f\\
g
\end{array}
\right]   &  =\left(  x_{2}-x_{1}\right)  f_{0}\left(  x_{1},x_{2}\right)
-\left(  -\sum_{n\geq0}h_{n+1}\left(  x_{1},x_{2}\right)  y_{1}^{n+1}\right)
=h_{0}\left(  x_{1},x_{2}\right)  +\sum_{n\geq1}h_{n}\left(  x_{1}%
,x_{2}\right)  y_{1}^{n}\\
&  =h,
\end{align*}
that is, $\ker\left(  \pi_{y}\right)  =\operatorname{im}\left(  d_{y}%
^{1}\right)  $. Hence $\mathfrak{D}_{q}\left(  y\right)  $ is an exact complex.
\end{proof}

\begin{remark}
\label{remLDc1}In the case of the $x$-diagonal formal $q$-complex
$\mathfrak{D}_{q}\left(  x\right)  $, one needs to take $h=\sum_{n\geq0}%
x^{n}g_{n+1}\left(  y_{1},y_{2}\right)  \in\left[  \left[  x\right]  \right]
\mathbb{C}\left[  \left[  y_{1},y_{2}\right]  \right]  $ whenever $\left[
\begin{array}
[c]{c}%
f\\
g
\end{array}
\right]  \in\ker\left(  d_{x}^{1}\right)  $. If $h\in\ker\left(  \pi
_{x}\right)  $ then we put $f=\sum_{n\geq0}x^{n}h_{n+1}\left(  y_{1}%
,y_{2}\right)  $, and $g=g_{0}\left(  y_{1},y_{2}\right)  $ whenever
$h_{0}\left(  y_{1},y_{2}\right)  =\left(  y_{2}-y_{1}\right)  g_{0}\left(
y_{1},y_{2}\right)  $.
\end{remark}

\subsection{The $y$-diagonal holomorphic $q$-complex\label{subsecDHC}}

Now assume that $\left\vert q\right\vert <1$, $U\subseteq\mathbb{C}_{xy}$ a
$q$-open subset, and let $U_{x_{1}}$, $U_{x_{2}}$ (resp., $U_{y_{1}}$,
$U_{y_{2}}$) be two copies of the $q$-open subset $U_{x}\subseteq\mathbb{C}$
(resp., $U_{y}\subseteq\mathbb{C}$). We replace $\mathbb{C}\left[  \left[
x_{1},x_{2}\right]  \right]  $ by the Fr\'{e}chet algebra $\mathcal{O}\left(
U_{x_{1}}\times U_{x_{2}}\right)  $ of holomorphic functions on the polydomain
$U_{x_{1}}\times U_{x_{2}}\subseteq\mathbb{C}^{2}$. In the same way,
$\mathbb{C}\left[  \left[  y_{1},y_{2}\right]  \right]  $ is replaced by
$\mathcal{O}\left(  U_{y_{1}}\times U_{y_{2}}\right)  $. As above there are
well defined continuous linear operators $N_{y},D_{y,q}\in\mathcal{L}\left(
\mathcal{O}\left(  U_{x_{1}}\times U_{x_{2}}\right)  \left[  \left[  y\right]
\right]  \right)  $, and $N_{x},D_{x,q}\in\mathcal{L}\left(  \left[  \left[
x\right]  \right]  \mathcal{O}\left(  U_{y_{1}}\times U_{y_{2}}\right)
\right)  $, which in turn define the diagonal holomorphic $q$-complexes
$\mathcal{O}_{q}\left(  y\right)  $ and $\mathcal{O}_{q}\left(  x\right)  $
instead of $\mathfrak{D}_{q}\left(  y\right)  $ and $\mathfrak{D}_{q}\left(
x\right)  $, respectively. Namely,\textit{ the }$y$-\textit{diagonal
holomorphic }$q$\textit{-complex} $\mathcal{O}_{q}\left(  y\right)  $ is the
following Fr\'{e}chet space complex
\[
0\rightarrow\mathcal{O}\left(  U_{x_{1}}\times U_{x_{2}}\right)  \left[
\left[  y\right]  \right]  \overset{d_{y}^{0}}{\longrightarrow}%
\begin{array}
[c]{c}%
\mathcal{O}\left(  U_{x_{1}}\times U_{x_{2}}\right)  \left[  \left[  y\right]
\right] \\
\oplus\\
\mathcal{O}\left(  U_{x_{1}}\times U_{x_{2}}\right)  \left[  \left[  y\right]
\right]
\end{array}
\overset{d_{y}^{1}}{\longrightarrow}\mathcal{O}\left(  U_{x_{1}}\times
U_{x_{2}}\right)  \left[  \left[  y\right]  \right]  \overset{\pi
_{y}}{\longrightarrow}\mathcal{O}\left(  U_{x}\right)  \rightarrow0,
\]
where $d_{y}^{i}$ are given by the same matrices (\ref{Fdif}), and $\pi
_{y}\left(  \sum_{n\geq0}h_{n}\left(  x_{1},x_{2}\right)  y^{n}\right)
=h_{0}\left(  x,x\right)  $. In the same way, \textit{the }$x$%
\textit{-diagonal holomorphic }$q$\textit{-complex }$\mathcal{O}_{q}\left(
x\right)  $ is defined.

\begin{proposition}
\label{propLDc1}The diagonal holomorphic $q$-complexes $\mathcal{O}_{q}\left(
x\right)  $ and $\mathcal{O}_{q}\left(  y\right)  $ are exact.
\end{proposition}

\begin{proof}
Let us consider the case of the $y$-diagonal complex $\mathcal{O}_{q}\left(
y\right)  $. By mimicking the proof of Lemma \ref{lemLDc1}, we have to prove
first that the equality $\left(  x_{2}-x_{1}\right)  f_{0}\left(  x_{1}%
,x_{2}\right)  =0$ for some $f_{0}\in\mathcal{O}\left(  U_{x_{1}}\times
U_{x_{2}}\right)  $ implies that $f_{0}=0$. Notice that $x_{2}=x_{1}$ defines
a thin subset $\Xi\subseteq\mathbb{C}^{2}$ to be zero set of the holomorphic
function $x_{2}-x_{1}$. The function $f_{0}$ is vanishing out of $\Xi
\cap\left(  U_{x_{1}}\times U_{x_{2}}\right)  $. By Removable Singularity
Theorem \cite[I.C.3]{GRo}, we deduce that $f_{0}$ is vanishing over all
$U_{x_{1}}\times U_{x_{2}}$. Further, if $h_{0}\in\mathcal{O}\left(  U_{x_{1}%
}\times U_{x_{2}}\right)  $, then the function%
\[
f_{0}\left(  x_{1},x_{2}\right)  =\frac{h_{0}\left(  x_{1},x_{2}\right)
-h_{0}\left(  x_{1},x_{1}\right)  }{x_{2}-x_{1}}%
\]
defined on $\left(  \mathbb{C}^{2}\backslash\Xi\right)  \cap\left(  U_{x_{1}%
}\times U_{x_{2}}\right)  $ has a unique holomorphic extension to $U_{x_{1}%
}\times U_{x_{2}}$ again by virtue of Removable Singularity Theorem.

Finally, in the case of $\mathcal{O}_{q}\left(  x\right)  $ we use the same
argument by appealing to Remark \ref{remLDc1}. The rest is clear.
\end{proof}

\section{The localizations of the quantum plane\label{sectionNL}}

In this section we switch to the quantum plane $\mathfrak{A}_{q}$ and its
bimodule resolution, which is lifted to the bimodule resolutions of the
structure sheaf algebras of some noncommutative analytic spaces of
$\mathfrak{A}_{q}$.

\subsection{The bimodule resolution of the quantum plane\label{subsecBRQP}}

The quantum plane $\mathfrak{A}_{q}$ possesses (see \cite{Tah}, \cite{WM}) the
following free $\mathfrak{A}_{q}$-bimodule resolution
\begin{equation}
0\leftarrow\mathfrak{A}_{q}\overset{\pi}{\longleftarrow}\mathfrak{A}%
_{q}\otimes\mathfrak{A}_{q}\overset{\partial_{0}}{\longleftarrow}%
\mathfrak{A}_{q}\otimes\mathbb{C}^{2}\otimes\mathfrak{A}_{q}\overset{\partial
_{1}}{\longleftarrow}\mathfrak{A}_{q}\otimes\wedge^{2}\mathbb{C}^{2}%
\otimes\mathfrak{A}_{q}\leftarrow0, \label{Ures}%
\end{equation}
with the mapping $\pi\left(  a\otimes b\right)  =ab$ and the differentials
\begin{align*}
\partial_{0}\left(  a\otimes e_{1}\otimes b\right)   &  =a\otimes xb-ax\otimes
b,\quad\partial_{0}\left(  a\otimes e_{2}\otimes b\right)  =a\otimes
yb-ay\otimes b,\\
\partial_{1}\left(  a\otimes e_{1}\wedge e_{2}\otimes b\right)   &  =a\otimes
e_{2}\otimes xb-qax\otimes e_{2}\otimes b-qa\otimes e_{1}\otimes yb+ay\otimes
e_{1}\otimes b,
\end{align*}
where $a,b\in\mathfrak{A}_{q}$ and $\left(  e_{1},e_{2}\right)  $ is the
standard basis of $\mathbb{C}^{2}$. One of the key results of \cite{Pir}
asserts that the canonical embedding $\mathfrak{A}_{q}\rightarrow
\mathcal{O}_{q}\left(  \mathbb{C}_{xy}\right)  $ is an absolute localization
(see Subsection \ref{subSth}). Thus, by applying the functor $\mathcal{O}%
_{q}\left(  \mathbb{C}_{xy}\right)  \underset{\mathfrak{A}_{q}%
}{\widehat{\otimes}}\circ\underset{\mathfrak{A}_{q}}{\widehat{\otimes}%
}\mathcal{O}_{q}\left(  \mathbb{C}_{xy}\right)  $ to the resolution
(\ref{Ures}), we derive that the following complex%
\[
0\leftarrow\mathcal{O}_{q}\left(  \mathbb{C}_{xy}\right)  \overset{\pi
}{\longleftarrow}\mathcal{O}_{q}\left(  \mathbb{C}_{xy}\right)
^{\widehat{\otimes}2}\overset{d_{0}}{\longleftarrow}\mathcal{O}_{q}\left(
\mathbb{C}_{xy}\right)  \widehat{\otimes}\mathbb{C}^{2}\widehat{\otimes
}\mathcal{O}_{q}\left(  \mathbb{C}_{xy}\right)  \overset{d_{1}}{\longleftarrow
}\mathcal{O}_{q}\left(  \mathbb{C}_{xy}\right)  \widehat{\otimes}\wedge
^{2}\mathbb{C}^{2}\widehat{\otimes}\mathcal{O}_{q}\left(  \mathbb{C}%
_{xy}\right)  \leftarrow0
\]
similar to (\ref{Ures}) is admissible. It provides a free $\mathcal{O}%
_{q}\left(  \mathbb{C}_{xy}\right)  $-bimodule resolution of $\mathcal{O}%
_{q}\left(  \mathbb{C}_{xy}\right)  $ (see \cite[3.2]{HelHom}). We use a bit
modified cochain version of this resolution
\[
\mathcal{R}\left(  \mathcal{O}_{q}\left(  \mathbb{C}_{xy}\right)
^{\widehat{\otimes}2}\right)  :0\rightarrow\mathcal{O}_{q}\left(
\mathbb{C}_{xy}\right)  ^{\widehat{\otimes}2}\overset{d^{0}}{\longrightarrow
}\mathcal{O}_{q}\left(  \mathbb{C}_{xy}\right)  ^{\widehat{\otimes}2}%
\oplus\mathcal{O}_{q}\left(  \mathbb{C}_{xy}\right)  ^{\widehat{\otimes}%
2}\overset{d^{1}}{\longrightarrow}\mathcal{O}_{q}\left(  \mathbb{C}%
_{xy}\right)  ^{\widehat{\otimes}2}\rightarrow0.
\]
Since $\mathcal{R}\left(  \mathcal{O}_{q}\left(  \mathbb{C}_{xy}\right)
^{\widehat{\otimes}2}\right)  \rightarrow\mathcal{O}_{q}\left(  \mathbb{C}%
_{xy}\right)  \rightarrow0$ is an admissible complex of the free
$\mathcal{O}_{q}\left(  \mathbb{C}_{xy}\right)  $-bimodules, we deduce (see
\cite[3.1.18]{HelHom}) that it splits in the category of Fr\'{e}chet
$\mathcal{O}_{q}\left(  \mathbb{C}_{xy}\right)  $-bimodules too.

Now let $\mathcal{A}$ be a Fr\'{e}chet $\mathcal{O}_{q}\left(  \mathbb{C}%
_{xy}\right)  $-algebra with its continuous homomorphism $\mathcal{O}%
_{q}\left(  \mathbb{C}_{xy}\right)  \rightarrow\mathcal{A}$. By applying the
functor $\mathcal{A}\underset{\mathcal{O}_{q}\left(  \mathbb{C}_{xy}\right)
}{\widehat{\otimes}}\circ\underset{\mathcal{O}_{q}\left(  \mathbb{C}%
_{xy}\right)  }{\widehat{\otimes}}\mathcal{A}$ to the resolution
$\mathcal{R}\left(  \mathcal{O}_{q}\left(  \mathbb{C}_{xy}\right)
^{\widehat{\otimes}2}\right)  $, we obtain that
\[
\mathcal{A}\underset{\mathcal{O}_{q}\left(  \mathbb{C}_{xy}\right)
}{\widehat{\otimes}}\mathcal{R}\left(  \mathcal{O}_{q}\left(  \mathbb{C}%
_{xy}\right)  ^{\widehat{\otimes}2}\right)  \underset{\mathcal{O}_{q}\left(
\mathbb{C}_{xy}\right)  }{\widehat{\otimes}}\mathcal{A=R}\left(
\mathcal{A}^{\widehat{\otimes}2}\right)  \text{,}%
\]
which is a free $\mathcal{A}$-bimodule complex. It can be augmented by the
multiplication morphism $\pi:\mathcal{A}^{\widehat{\otimes}2}\rightarrow
\mathcal{A}$, that is, $\mathcal{R}\left(  \mathcal{A}^{\widehat{\otimes}%
2}\right)  \overset{\pi}{\longrightarrow}\mathcal{A}\rightarrow0$ is a complex
of the free $\mathcal{A}$-bimodules. Thus, we come up with the following
cochain complex
\begin{equation}
0\rightarrow\mathcal{A}^{\widehat{\otimes}2}\overset{d^{0}}{\longrightarrow}%
\begin{array}
[c]{c}%
\mathcal{A}^{\widehat{\otimes}2}\\
\oplus\\
\mathcal{A}^{\widehat{\otimes}2}%
\end{array}
\overset{d^{1}}{\longrightarrow}\mathcal{A}^{\widehat{\otimes}2}%
\overset{\pi}{\longrightarrow}\mathcal{A}\rightarrow0, \label{Rsp}%
\end{equation}
whose differentials can be given by the following operator matrices
\begin{equation}
d^{0}=\left[
\begin{array}
[c]{c}%
R_{y}\otimes1-q\otimes L_{y}\\
1\otimes L_{x}-R_{x}\otimes q
\end{array}
\right]  ,\text{ \ }d^{1}=\left[
\begin{array}
[c]{cc}%
1\otimes L_{x}-R_{x}\otimes1 & 1\otimes L_{y}-R_{y}\otimes1
\end{array}
\right]  , \label{d}%
\end{equation}
where $L$ and $R$ indicate to the left and right regular (anti)
representations of the algebra $\mathcal{A}$.

Based on the proper definition of localizations from Subsection \ref{subSth},
we deduce that $\mathcal{O}_{q}\left(  \mathbb{C}_{xy}\right)  \rightarrow
\mathcal{A}$ is an absolute localization if (\ref{Rsp}) is admissible (in
particular, exact). In this case, (\ref{Rsp}) provides a free $\mathcal{A}%
$-bimodule resolution of the algebra $\mathcal{A}$. But if $\mathcal{A}$ is a
nuclear Fr\'{e}chet space, and (\ref{Rsp}) is exact, then $\mathcal{O}%
_{q}\left(  \mathbb{C}_{xy}\right)  \rightarrow\mathcal{A}$ (or $\mathfrak{A}%
_{q}\rightarrow\mathcal{A}$) is a localization, that is, all homology functors
$H_{p}\left(  \mathcal{A},\circ\right)  $, $p\geq0$ are reduced to
$H_{p}\left(  \mathfrak{A}_{q},\circ\right)  $, $p\geq0$ (see Subsection
\ref{subSth}). Notice that $\mathcal{O}_{q}\left(  \mathbb{C}_{xy}\right)
\rightarrow\mathcal{A}$ is a localization if and only so is $\mathfrak{A}%
_{q}\rightarrow\mathcal{A}$ \cite[Proposition 1.8]{Tay2}. The potential
candidates for the localizing algebras $\mathcal{A}$ are the Fr\'{e}chet
$\mathcal{O}_{q}\left(  \mathbb{C}_{xy}\right)  $-algebras $\mathcal{O}\left(
U_{x}\right)  \left[  \left[  y\right]  \right]  $, $\left[  \left[  x\right]
\right]  \mathcal{O}\left(  U_{y}\right)  $, $\mathcal{F}_{q}\left(  U\right)
$ and $\mathbb{C}_{q}\left[  \left[  x,y\right]  \right]  $ considered below.
They are all nuclear Fr\'{e}chet spaces, and we prove that the related
complexes (\ref{Rsp}) are exact indeed.

\begin{proposition}
\label{propKD}Let $\mathcal{A}$ be a Fr\'{e}chet $\mathcal{O}_{q}\left(
\mathbb{C}_{xy}\right)  $-algebra with its continuous homomorphism
$\mathcal{O}_{q}\left(  \mathbb{C}_{xy}\right)  \rightarrow\mathcal{A}$. If
$\operatorname{im}\left(  d^{1}\right)  $ is closed then $H^{2}\left(
\mathcal{R}\left(  \mathcal{A}^{\widehat{\otimes}2}\right)  \right)
=\mathcal{A}\widehat{\otimes}_{\mathcal{O}_{q}\left(  \mathbb{C}_{xy}\right)
}\mathcal{A}$. If additionally $\mathcal{O}_{q}\left(  \mathbb{C}_{xy}\right)
\rightarrow\mathcal{A}$ has the dense range, then $\operatorname{im}\left(
d^{1}\right)  =\ker\left(  \pi\right)  $.
\end{proposition}

\begin{proof}
First notice that $ax\otimes b-a\otimes xb$ and $ay\otimes b-a\otimes yb$
belong to $\operatorname{im}\left(  d^{1}\right)  $ for all $a,b\in
\mathcal{A}$. In particular, $ap\otimes b-a\otimes pb\in\operatorname{im}%
\left(  d^{1}\right)  $ for every polynomial $p\in\mathfrak{A}_{q}$. Using the
continuity argument, we deduce that $af\otimes b-a\otimes fb\in
\operatorname{im}\left(  d^{1}\right)  $ for all $f\in\mathcal{O}_{q}\left(
\mathbb{C}_{xy}\right)  $ whenever $\operatorname{im}\left(  d^{1}\right)  $
is closed. It follows that $\operatorname{im}\left(  d^{1}\right)  $ is the
closed linear span of all elements $af\otimes b-a\otimes fb$, $a,b\in
\mathcal{A}$, $f\in\mathcal{O}_{q}\left(  \mathbb{C}_{xy}\right)  $. Hence
\[
H^{2}\left(  \mathcal{R}\left(  \mathcal{A}^{\widehat{\otimes}2}\right)
\right)  =\mathcal{A}^{\widehat{\otimes}2}/\operatorname{im}\left(
d^{1}\right)  =\mathcal{A}\widehat{\otimes}_{\mathcal{O}_{q}\left(
\mathbb{C}_{xy}\right)  }\mathcal{A}.
\]
Finally, if $\mathcal{O}_{q}\left(  \mathbb{C}_{xy}\right)  \rightarrow
\mathcal{A}$ has the dense range, then $H^{2}\left(  \mathcal{R}\left(
\mathcal{A}^{\widehat{\otimes}2}\right)  \right)  =\mathcal{A}\widehat{\otimes
}_{\mathcal{A}}\mathcal{A=A}$. But $\mathcal{A}^{\widehat{\otimes}2}%
/\ker\left(  \pi\right)  =\mathcal{A}$ holds too, and $\operatorname{im}%
\left(  d^{1}\right)  \subseteq\ker\left(  \pi\right)  $. Therefore
$\operatorname{im}\left(  d^{1}\right)  =\ker\left(  \pi\right)  $.
\end{proof}

\subsection{The noncommutative analytic space $\left\{  \left(  0,0\right)
\right\}  $\label{subsecFQM}}

The space $\mathbb{C}\left[  \left[  x,y\right]  \right]  $ of all formal
power series in variables $x$ and $y$ is a Fr\'{e}chet space equipped with the
direct product topology of $\prod\limits_{i,k}\mathbb{C}x^{i}y^{k}$. If $x$
and $y$ are the generators of the quantum plane $\mathfrak{A}_{q}$ then
$\mathbb{C}\left[  \left[  x,y\right]  \right]  $ can be equipped with a
formal $q$-multiplication so that the canonical embedding $\mathfrak{A}%
_{q}\rightarrow\mathbb{C}\left[  \left[  x,y\right]  \right]  $ is an algebra
homomorphism. Namely, if $f=\sum_{i,k}a_{ik}x^{i}y^{k}$ and $g=\sum
_{i,k}b_{ik}x^{i}y^{k}$, then we define their formal $q$-multiplication to be
the following formal power series%
\begin{equation}
f\cdot g=\sum_{m,n}\left(  \sum_{s+t=m,i+j=n}a_{si}q^{it}b_{tj}\right)
x^{m}y^{n}. \label{mulFr}%
\end{equation}
One can easily verify that this is an associative multiplication, it defines
an Arens-Michael-Fr\'{e}chet algebra structure on $\mathbb{C}\left[  \left[
x,y\right]  \right]  $, and $\mathfrak{A}_{q}\rightarrow\mathbb{C}\left[
\left[  x,y\right]  \right]  $ is a (unital) algebra homomorphism. In
particular, $\mathbb{C}\left[  \left[  x,y\right]  \right]  $ is a Fr\'{e}chet
$\mathcal{O}_{q}\left(  \mathbb{C}_{xy}\right)  $-algebra denoted by
$\mathbb{C}_{q}\left[  \left[  x,y\right]  \right]  $. Using the canonical
representations $\mathbb{C}_{q}\left[  \left[  x,y\right]  \right]
=\mathbb{C}\left[  \left[  x\right]  \right]  \left[  \left[  y\right]
\right]  =\left[  \left[  x\right]  \right]  \mathbb{C}\left[  \left[
y\right]  \right]  $, the formal $q$-multiplication in $\mathbb{C}_{q}\left[
\left[  x,y\right]  \right]  $ can be defined in two different formulas
\begin{equation}
f\cdot g=\sum_{n}\left(  \sum_{i+j=n}f_{i}\left(  x\right)  g_{j}\left(
q^{i}x\right)  \right)  y^{n}\text{,\quad\ }f,g\in\mathbb{C}\left[  \left[
x\right]  \right]  \left[  \left[  y\right]  \right]  \text{,} \label{5p2}%
\end{equation}%
\begin{equation}
f\cdot g=\sum_{n}x^{n}\left(  \sum_{i+j=n}f_{i}\left(  q^{j}y\right)
g_{j}\left(  y\right)  \right)  \text{, \quad}f,g\in\left[  \left[  x\right]
\right]  \mathbb{C}\left[  \left[  y\right]  \right]  \label{pp4}%
\end{equation}
by extending the original multiplication of the $q$-plane $\mathfrak{A}_{q}$
(see \cite{Dosi25}). In particular,
\[
y^{n}g\left(  x\right)  =g\left(  q^{n}x\right)  y^{n}\quad\text{and\quad
}f\left(  y\right)  x^{n}=x^{n}f\left(  q^{n}y\right)
\]
for all $g\in\mathbb{C}\left[  \left[  x\right]  \right]  $, $f\in
\mathbb{C}\left[  \left[  y\right]  \right]  $ and $n\geq0$. The algebra
$\mathbb{C}_{q}\left[  \left[  x,y\right]  \right]  $ has the continuous
trivial character $\left(  0,0\right)  :\mathbb{C}_{q}\left[  \left[
x,y\right]  \right]  \rightarrow\mathbb{C}$ annihilating both variables $x$
and $y$. That is the only continuous character of the
Arens-Michael-Fr\'{e}chet algebra $\mathbb{C}_{q}\left[  \left[  x,y\right]
\right]  $ denoted by $\left(  0,0\right)  $ representing the point of
$\mathbb{C}_{xy}$. Actually, the algebra $\mathbb{C}_{q}\left[  \left[
x,y\right]  \right]  $ is local with $\operatorname{Rad}\mathbb{C}_{q}\left[
\left[  x,y\right]  \right]  =\ker\left(  0,0\right)  $ to be the closed two
sided ideal generated by $x$ and $y$. Thus $\operatorname{Spec}\left(
\mathbb{C}_{q}\left[  \left[  x,y\right]  \right]  \right)  =\left\{  \left(
0,0\right)  \right\}  $ (see \cite[Lemma 5.1]{Dosi25}).

\begin{definition}
\label{def1}The noncommutative analytic space representing the intersection
point $\left(  0,0\right)  $ of the complex lines $\mathbb{C}_{x}$ and
$\mathbb{C}_{y}$ in $\mathbb{C}_{xy}$ is called the ringed space $\left(
\left\{  \left(  0,0\right)  \right\}  ,\mathbb{C}_{q}\left[  \left[
x,y\right]  \right]  \right)  $, which consists of the topological space
$\left\{  \left(  0,0\right)  \right\}  $ and the constant structure
Fr\'{e}chet $\mathcal{O}_{q}\left(  \mathbb{C}_{xy}\right)  $-algebra sheaf
$\mathbb{C}_{q}\left[  \left[  x,y\right]  \right]  $.
\end{definition}

Roughly speaking, the ringed space $\left(  \left\{  \left(  0,0\right)
\right\}  ,\mathbb{C}_{q}\left[  \left[  x,y\right]  \right]  \right)  $ is a
geometric model of $\mathfrak{A}_{q}$ (actually of its certain multinormed
envelope) representing just a single point, and it has a unique stalk
$\mathbb{C}_{q}\left[  \left[  x,y\right]  \right]  $ to be a local algebra.
If $X$ is a left Fr\'{e}chet $\mathcal{O}_{q}\left(  \mathbb{C}_{xy}\right)
$-module given by a couple $\left(  T,S\right)  $ of continuous linear
operators on $X$ with the relation $TS=q^{-1}ST$, then the transversality
relation (see below Section \ref{secJS}) of $X$ with respect to the sheaf
$\mathbb{C}_{q}\left[  \left[  x,y\right]  \right]  $ is solved in terms of
the Taylor spectrum of the module $X$ \cite[Proposition 6.4]{Dosi25}. Below we
investigate the localization property of the structure sheaf $\mathbb{C}%
_{q}\left[  \left[  x,y\right]  \right]  $ of this noncommutative space.

Put $\mathcal{A}=\mathbb{C}_{q}\left[  \left[  x,y\right]  \right]  $ and take
its two copies $\mathbb{C}_{q}\left[  \left[  x_{i},y_{i}\right]  \right]  $,
$i=1,2$. Since $\mathcal{A}=\mathbb{C}_{q}\left[  \left[  x_{i}\right]
\right]  \left[  \left[  y_{i}\right]  \right]  $, we conclude that
\begin{equation}
\mathcal{A}^{\widehat{\otimes}2}=\mathbb{C}_{q}\left[  \left[  x_{1}\right]
\right]  \left[  \left[  y_{1}\right]  \right]  \widehat{\otimes}%
\mathbb{C}_{q}\left[  \left[  x_{2}\right]  \right]  \left[  \left[
y_{2}\right]  \right]  =\mathbb{C}\left[  \left[  x_{1},x_{2}\right]  \right]
\left[  \left[  y_{1}\right]  \right]  \left[  \left[  y_{2}\right]  \right]
\label{AlI}%
\end{equation}
up to the canonical identification of the Fr\'{e}chet spaces. For a while we
fix this identification and modify the complex (\ref{Rsp}). First, the right
multiplication (or shift) operators $R_{y_{i}}\in\mathcal{L}\left(
\mathbb{C}_{q}\left[  \left[  x_{i}\right]  \right]  \left[  \left[
y_{i}\right]  \right]  \right)  $ define their inflations $N_{y_{1}}=R_{y_{1}%
}\otimes1$ and $N_{y_{2}}=1\otimes R_{y_{2}}$ on $\mathcal{A}%
^{\widehat{\otimes}2}$ such that
\[
N_{y_{1}}\left(  h\right)  =\sum_{n,m}h_{nm}\left(  x_{1},x_{2}\right)
y_{1}^{n+1}y_{2}^{m}\text{\quad and\quad}N_{y_{2}}\left(  h\right)
=\sum_{n,m}h_{nm}\left(  x_{1},x_{2}\right)  y_{1}^{n}y_{2}^{m+1}%
\]
for every $h=\sum_{n,m}h_{nm}\left(  x_{1},x_{2}\right)  y_{1}^{n}y_{2}^{m}%
\in\mathcal{A}^{\widehat{\otimes}2}$ up to the identification from
(\ref{AlI}). Notice also that $L_{x_{1}}\otimes1$ and $1\otimes L_{x_{2}}$ are
acting as the diagonal operators $x_{1}\otimes1$ and $1\otimes x_{2}$ obtained
by the multiplication operators $x_{1}$ and $x_{2}$ on $\mathbb{C}\left[
\left[  x_{1},x_{2}\right]  \right]  $. We use the same notations $x_{1}$ and
$x_{2}$ instead of $x_{1}\otimes1$ and $1\otimes x_{2}$. There are also the
diagonal operators $D_{q},\Delta_{q}:\mathcal{A\rightarrow A}$ acting by
$D_{q}\left(  f\right)  =\sum_{n}q^{n}f_{n}\left(  x\right)  y^{n}$ and
$\Delta_{q}\left(  f\right)  =\sum_{n}f_{n}\left(  qx\right)  y^{n}$ for every
$f=\sum_{n}f_{n}\left(  x\right)  y^{n}$. Let us pick up all necessary
formulas in the following assertion.

\begin{lemma}
\label{lemCLR}The following operator equalities%
\[
N_{y_{1}}=R_{y_{1}}\otimes1,\quad N_{y_{2}}=1\otimes R_{y_{2}},\quad L_{x_{1}%
}\otimes1=x_{1}\text{,\quad}1\otimes L_{x_{2}}=x_{2},
\]%
\[
R_{x_{1}}\otimes1=x_{1}\left(  D_{q}\otimes1\right)  \text{,\quad}1\otimes
L_{y_{2}}=\left(  1\otimes\Delta_{q}\right)  N_{y_{2}}%
\]
hold on $\mathcal{A}^{\widehat{\otimes}2}$ up to the identification (\ref{AlI}).
\end{lemma}

\begin{proof}
Take $h=f\otimes g\in\mathcal{A}^{\widehat{\otimes}2}$ with $f=\sum_{n}%
f_{n}\left(  x_{1}\right)  y_{1}^{n}$ and $g=\sum_{m}g_{m}\left(
x_{2}\right)  y_{2}^{m}$. Based on (\ref{AlI}), $h$ can be identified with
$\sum_{n,m}f_{n}\left(  x_{1}\right)  g_{m}\left(  x_{2}\right)  y_{1}%
^{n}y_{2}^{m}$ to be an element of $\mathbb{C}\left[  \left[  x_{1}%
,x_{2}\right]  \right]  \left[  \left[  y_{1}\right]  \right]  \left[  \left[
y_{2}\right]  \right]  $. Then
\[
\left(  R_{x_{1}}\otimes1\right)  \left(  h\right)  =\sum_{n}f_{n}\left(
x_{1}\right)  y_{1}^{n}x_{1}\otimes g=\sum_{n}f_{n}\left(  x_{1}\right)
q^{n}x_{1}y_{1}^{n}\otimes g=x_{1}\left(  D_{q}\otimes1\right)  \left(
h\right)  ,
\]
that is, $R_{x_{1}}\otimes1=x_{1}\left(  D_{q}\otimes1\right)  $ (see
(\ref{5p2})). Further, we have
\begin{align*}
\left(  1\otimes L_{y_{2}}\right)  \left(  h\right)   &  =f\otimes\sum
_{m}y_{2}g_{m}\left(  x_{2}\right)  y_{2}^{m}=f\otimes\sum_{m}g_{m}\left(
qx_{2}\right)  y_{2}^{m+1}=N_{y_{2}}\left(  f\otimes\sum_{m}g_{m}\left(
qx_{2}\right)  y_{2}^{m}\right) \\
&  =N_{y_{2}}\left(  f\otimes\Delta_{q}\left(  g\right)  \right)  =N_{y_{2}%
}\left(  1\otimes\Delta_{q}\right)  \left(  h\right)  .
\end{align*}
Finally, notice that the operators $N_{y_{2}}$ and $1\otimes\Delta_{q}$ commute.
\end{proof}

\begin{corollary}
\label{corCLR1}If $\mathcal{A}^{\widehat{\otimes}2}$ is identified with
$\left[  \left[  x_{1}\right]  \right]  \left[  \left[  x_{2}\right]  \right]
\mathbb{C}\left[  \left[  y_{1},y_{2}\right]  \right]  $, then
\[
N_{x_{1}}=L_{x_{1}}\otimes1,\quad N_{x_{2}}=1\otimes L_{x_{2}},\quad R_{y_{1}%
}\otimes1=y_{1}\text{,\quad}1\otimes R_{y_{2}}=y_{2},
\]%
\[
R_{x_{1}}\otimes1=N_{x_{1}}\left(  \Delta_{q}\otimes1\right)  \text{,\quad
}1\otimes L_{y_{2}}=\left(  1\otimes D_{q}\right)  y_{2}%
\]
hold.
\end{corollary}

\begin{proof}
One needs to use the same arguments from the proof of Lemma \ref{lemCLR}.
\end{proof}

\subsection{The localization property of $\mathbb{C}_{q}\left[  \left[
x,y\right]  \right]  $\label{subsecCQxy}}

Let us analyze the complex (\ref{Rsp}) based on the identification
(\ref{AlI}). Taking into account that
\[
\mathcal{A}^{\widehat{\otimes}2}\oplus\mathcal{A}^{\widehat{\otimes}2}=\left(
\mathbb{C}\left[  \left[  x_{1},x_{2}\right]  \right]  \left[  \left[
y_{1}\right]  \right]  \oplus\mathbb{C}\left[  \left[  x_{1},x_{2}\right]
\right]  \left[  \left[  y_{1}\right]  \right]  \right)  \left[  \left[
y_{2}\right]  \right]  ,
\]
one can transform (\ref{Rsp}) into the following complex
\[
0\rightarrow\mathbb{C}\left[  \left[  x_{1},x_{2}\right]  \right]  \left[
\left[  y_{1}\right]  \right]  \left[  \left[  y_{2}\right]  \right]
\overset{d^{0}}{\longrightarrow}\left(
\begin{array}
[c]{c}%
\mathbb{C}\left[  \left[  x_{1},x_{2}\right]  \right]  \left[  \left[
y_{1}\right]  \right] \\
\oplus\\
\mathbb{C}\left[  \left[  x_{1},x_{2}\right]  \right]  \left[  \left[
y_{1}\right]  \right]
\end{array}
\right)  \left[  \left[  y_{2}\right]  \right]  \overset{d^{1}%
}{\longrightarrow}\mathbb{C}\left[  \left[  x_{1},x_{2}\right]  \right]
\left[  \left[  y_{1}\right]  \right]  \left[  \left[  y_{2}\right]  \right]
\overset{\pi}{\longrightarrow}\mathcal{A}\rightarrow0.
\]
The related differentials can be calculated based on Lemma \ref{lemCLR}.
Namely, (see (\ref{d}))%
\[
d^{0}=\left[
\begin{array}
[c]{c}%
R_{y_{1}}\otimes1-q\otimes L_{y_{2}}\\
1\otimes L_{x_{2}}-R_{x_{1}}\otimes q
\end{array}
\right]  =\left[
\begin{array}
[c]{c}%
N_{y_{1}}\\
x_{2}-x_{1}\left(  D_{q}\otimes q\right)
\end{array}
\right]  +\left[
\begin{array}
[c]{c}%
-\left(  q\otimes\Delta_{q}\right)  N_{y_{2}}\\
0
\end{array}
\right]  .
\]
In a similar way, we have%
\begin{align*}
d^{1}  &  =\left[
\begin{array}
[c]{cc}%
1\otimes L_{x_{2}}-R_{x_{1}}\otimes1 & 1\otimes L_{y_{2}}-R_{y_{1}}\otimes1
\end{array}
\right] \\
&  =\left[
\begin{array}
[c]{cc}%
x_{2}-x_{1}\left(  D_{q}\otimes1\right)  & -N_{y_{1}}%
\end{array}
\right]  +\left[
\begin{array}
[c]{cc}%
0 & \left(  1\otimes\Delta_{q}\right)  N_{y_{2}}%
\end{array}
\right]  .
\end{align*}
Notice that for every $h=\sum_{n}h_{n}\left(  x_{1},x_{2}\right)  y_{1}^{n}%
\in\mathbb{C}\left[  \left[  x_{1},x_{2}\right]  \right]  \left[  \left[
y_{1}\right]  \right]  $, we have
\begin{align*}
\left(  x_{2}-x_{1}\left(  D_{q}\otimes q\right)  \right)  \left(  hy_{2}%
^{m}\right)   &  =\sum_{n}\left(  x_{2}-q^{n+1}x_{1}\right)  h_{n}\left(
x_{1},x_{2}\right)  y_{1}^{n}y_{2}^{m}\\
&  =\prod\limits_{n}\left(  x_{2}-q^{n+1}x_{1}\right)  \left(  h\right)
y_{2}^{m},
\end{align*}
that is, $\left(  x_{2}-x_{1}\left(  D_{q}\otimes q\right)  \right)
|\mathbb{C}\left[  \left[  x_{1},x_{2}\right]  \right]  \left[  \left[
y_{1}\right]  \right]  y_{2}^{m}=x_{2}-qx_{1}D_{y_{1},q}$ (see (\ref{Fdif})).
In a similar way, we have $x_{2}-x_{1}\left(  D_{q}\otimes1\right)
|\mathbb{C}\left[  \left[  x_{1},x_{2}\right]  \right]  \left[  \left[
y_{1}\right]  \right]  y_{2}^{m}=x_{2}-x_{1}D_{y_{1},q}$. Thus $d^{0}%
|\mathbb{C}\left[  \left[  x_{1},x_{2}\right]  \right]  \left[  \left[
y_{1}\right]  \right]  y_{2}^{m}=d_{y_{1}}^{0}+N_{y_{2}}^{0}$ with%
\[
d_{y_{1}}^{0}=\left[
\begin{array}
[c]{c}%
N_{y_{1}}\\
x_{2}-qx_{1}D_{y_{1},q}%
\end{array}
\right]  :\mathbb{C}\left[  \left[  x_{1},x_{2}\right]  \right]  \left[
\left[  y_{1}\right]  \right]  y_{2}^{m}\rightarrow\mathbb{C}\left[  \left[
x_{1},x_{2}\right]  \right]  \left[  \left[  y_{1}\right]  \right]  ^{\oplus
2}y_{2}^{m},
\]%
\[
N_{y_{2}}^{0}=\left[
\begin{array}
[c]{c}%
-\left(  q\otimes\Delta_{q}\right)  N_{y_{2}}\\
0
\end{array}
\right]  :\mathbb{C}\left[  \left[  x_{1},x_{2}\right]  \right]  \left[
\left[  y_{1}\right]  \right]  y_{2}^{m}\rightarrow\mathbb{C}\left[  \left[
x_{1},x_{2}\right]  \right]  \left[  \left[  y_{1}\right]  \right]  ^{\oplus
2}y_{2}^{m+1}.
\]
In a similar way, we have $d^{1}|\mathbb{C}\left[  \left[  x_{1},x_{2}\right]
\right]  \left[  \left[  y_{1}\right]  \right]  ^{\oplus2}y_{2}^{m}=d_{y_{1}%
}^{1}+N_{y_{2}}^{1}$ with
\[
d_{y_{1}}^{1}=\left[
\begin{array}
[c]{cc}%
x_{2}-x_{1}D_{y_{1},q} & -N_{y_{1}}%
\end{array}
\right]  :\mathbb{C}\left[  \left[  x_{1},x_{2}\right]  \right]  \left[
\left[  y_{1}\right]  \right]  ^{\oplus2}y_{2}^{m}\rightarrow\mathbb{C}\left[
\left[  x_{1},x_{2}\right]  \right]  \left[  \left[  y_{1}\right]  \right]
y_{2}^{m},
\]%
\[
N_{y_{2}}^{1}=\left[
\begin{array}
[c]{cc}%
0 & \left(  1\otimes\Delta_{q}\right)  N_{y_{2}}%
\end{array}
\right]  :\mathbb{C}\left[  \left[  x_{1},x_{2}\right]  \right]  \left[
\left[  y_{1}\right]  \right]  ^{\oplus2}y_{2}^{m}\rightarrow\mathbb{C}\left[
\left[  x_{1},x_{2}\right]  \right]  \left[  \left[  y_{1}\right]  \right]
y_{2}^{m+1}.
\]
It follows that the differentials $d^{i}$, $i=0,1$ have lower triangular
operator matrices%
\[
d^{0}=\left[
\begin{array}
[c]{ccc}%
\ddots &  & 0\\
N_{y_{2}}^{0} & d_{y_{1}}^{0} & \\
0 & N_{y_{2}}^{0} & \ddots
\end{array}
\right]  ,\quad d^{1}=\left[
\begin{array}
[c]{ccc}%
\ddots &  & 0\\
N_{y_{2}}^{1} & d_{y_{1}}^{1} & \\
0 & N_{y_{2}}^{1} & \ddots
\end{array}
\right]  .
\]
Moreover, $d_{y_{1}}^{i}$, $i=0,1$ are differentials of the $y_{1}$-diagonal
formal $q$-complex $\mathfrak{D}_{q}\left(  y_{1}\right)  $ considered above
in Subsection \ref{subsecLDFC}. Further, take $h=f\otimes g\left(
x_{2}\right)  y_{2}^{m}\in\mathbb{C}\left[  \left[  x_{1},x_{2}\right]
\right]  \left[  \left[  y_{1}\right]  \right]  y_{2}^{m}$ with $f\in
\mathcal{A}$. Since $\pi:\mathbb{C}\left[  \left[  x_{1},x_{2}\right]
\right]  \left[  \left[  y_{1}\right]  \right]  \left[  \left[  y_{2}\right]
\right]  \longrightarrow\mathbb{C}_{q}\left[  \left[  x\right]  \right]
\left[  \left[  y\right]  \right]  \mathcal{\ }$is the multiplication
morphism, it follows that
\begin{align*}
\pi\left(  h\right)   &  =\pi\left(  \sum_{k}f_{k}\left(  x_{1}\right)
y_{1}^{k}\otimes g\left(  x_{2}\right)  y_{2}^{m}\right)  =\sum_{k}%
f_{k}\left(  x\right)  y^{k}g\left(  x\right)  y^{m}\\
&  =\sum_{k}f_{k}\left(  x\right)  g\left(  q^{k}x\right)  y^{k+m}%
=f_{0}\left(  x\right)  g\left(  x\right)  y^{m}+f_{1}\left(  x\right)
g\left(  qx\right)  y^{m+1}+\cdots.
\end{align*}
In particular, we have
\[
\pi\left(  h\right)  =\pi\left(  \sum_{n}h_{n}\left(  x_{1},x_{2}\right)
y_{1}^{n}y_{2}^{m}\right)  =h_{0}\left(  x,x\right)  y^{m}+h_{1}\left(
x,qx\right)  y^{m+1}+\cdots
\]
for all $h=\sum_{n}h_{n}\left(  x_{1},x_{2}\right)  y_{1}^{n}y_{2}^{m}%
\in\mathbb{C}\left[  \left[  x_{1},x_{2}\right]  \right]  \left[  \left[
y_{1}\right]  \right]  y_{2}^{m}$. Thus $\pi$ has a lower triangular operator
matrix
\[
\pi=\left[
\begin{array}
[c]{ccc}%
\ddots &  & 0\\
& \pi_{y_{1}} & \\
\ast &  & \ddots
\end{array}
\right]  ,
\]
where $\pi_{y_{1}}:\mathbb{C}\left[  \left[  x_{1},x_{2}\right]  \right]
\left[  \left[  y_{1}\right]  \right]  y_{2}^{m}\rightarrow\mathbb{C}\left[
\left[  x\right]  \right]  y^{m}$, $\pi_{y_{1}}\left(  h\right)  =h_{0}\left(
x,x\right)  y^{m}$ is the morphism of the $y_{1}$-diagonal formal $q$-complex
$\mathfrak{D}_{q}\left(  y_{1}\right)  $.

\begin{proposition}
\label{propLoc1}The noncommutative analytic space $\left(  \left\{  \left(
0,0\right)  \right\}  ,\mathbb{C}_{q}\left[  \left[  x,y\right]  \right]
\right)  $ has the localization property, that is, the canonical homomorphism
$\mathfrak{A}_{q}\rightarrow\mathbb{C}_{q}\left[  \left[  x,y\right]  \right]
$ is a localization, where $q\in\mathbb{C}\backslash\left\{  0,1\right\}  $.
\end{proposition}

\begin{proof}
Since we come up with the nuclear algebra $\mathcal{A}=\mathbb{C}_{q}\left[
\left[  x,y\right]  \right]  $, it suffices to prove that the complex
(\ref{Rsp}) is exact. But $\mathcal{A}=\mathbb{C}_{q}\left[  \left[  x\right]
\right]  \left[  \left[  y\right]  \right]  $, and as we have seen above
(\ref{Rsp}) can be transformed into a complex of the formal power series over
$y_{2}$ whose differentials have a lower triangular operator matrices. The
related diagonal complexes are in turn reduced to the same $y_{1}$-diagonal
formal $q$-complex $\mathfrak{D}_{q}\left(  y_{1}\right)  $. By Lemma
\ref{lemLDc1}, the complex $\mathfrak{D}_{q}\left(  y_{1}\right)  $ is exact,
that is, all diagonal complexes of (\ref{Rsp}) are exact. Using Lemma
\ref{lemTri}, we deduce that (\ref{Rsp}) remains exact.
\end{proof}

The same argument is applicable to the representation of $\mathbb{C}%
_{q}\left[  \left[  x,y\right]  \right]  $ as $\left[  \left[  x\right]
\right]  \mathbb{C}_{q}\left[  \left[  y\right]  \right]  $ with the
$q$-multiplication (\ref{pp4}) too. In this case, (\ref{Rsp}) is transformed
into the following complex
\[
0\rightarrow\left[  \left[  x_{1}\right]  \right]  \left[  \left[
x_{2}\right]  \right]  \mathbb{C}\left[  \left[  y_{1},y_{2}\right]  \right]
\overset{d^{0}}{\longrightarrow}\left[  \left[  x_{1}\right]  \right]  \left(
\begin{array}
[c]{c}%
\left[  \left[  x_{2}\right]  \right]  \mathbb{C}\left[  \left[  y_{1}%
,y_{2}\right]  \right] \\
\oplus\\
\left[  \left[  x_{2}\right]  \right]  \mathbb{C}\left[  \left[  y_{1}%
,y_{2}\right]  \right]
\end{array}
\right)  \overset{d^{1}}{\longrightarrow}\left[  \left[  x_{1}\right]
\right]  \left[  \left[  x_{2}\right]  \right]  \mathbb{C}\left[  \left[
y_{1},y_{2}\right]  \right]  \overset{\pi}{\longrightarrow}\mathcal{A}%
\rightarrow0.
\]
The related differentials can be calculated based on Corollary \ref{corCLR1}.
Namely, as above we have%
\begin{align*}
d^{0}  &  =\left[
\begin{array}
[c]{c}%
R_{y_{1}}\otimes1-q\otimes L_{y_{2}}\\
1\otimes L_{x_{2}}-R_{x_{1}}\otimes q
\end{array}
\right]  =\left[
\begin{array}
[c]{c}%
y_{1}-\left(  q\otimes D_{q}\right)  y_{2}\\
N_{x_{2}}%
\end{array}
\right]  +\left[
\begin{array}
[c]{c}%
0\\
-N_{x_{1}}\left(  \Delta_{q}\otimes q\right)
\end{array}
\right] \\
&  =d_{x_{2}}^{0}+N_{x_{1}}^{0},\\
d^{1}  &  =\left[
\begin{array}
[c]{cc}%
1\otimes L_{x_{2}}-R_{x_{1}}\otimes1 & 1\otimes L_{y_{2}}-R_{y_{1}}\otimes1
\end{array}
\right] \\
&  =\left[
\begin{array}
[c]{cc}%
N_{x_{2}} & \left(  1\otimes D_{q}\right)  y_{2}-y_{1}%
\end{array}
\right]  +\left[
\begin{array}
[c]{cc}%
-N_{x_{1}}\left(  \Delta_{q}\otimes1\right)  & 0
\end{array}
\right]  =d_{x_{2}}^{1}+N_{x_{1}}^{1}%
\end{align*}
(see Subsection \ref{subsecRDC}). The differentials $d^{i}$, $i=0,1$ have
lower operator matrices%
\[
d^{0}=\left[
\begin{array}
[c]{ccc}%
\ddots &  & 0\\
N_{x_{1}}^{0} & d_{x_{2}}^{0} & \\
0 & N_{x_{1}}^{0} & \ddots
\end{array}
\right]  ,\quad d^{1}=\left[
\begin{array}
[c]{ccc}%
\ddots &  & 0\\
N_{x_{1}}^{1} & d_{x_{2}}^{1} & \\
0 & N_{x_{1}}^{1} & \ddots
\end{array}
\right]
\]
whose diagonal are differentials $d_{x_{2}}^{i}$ of the $x_{2}$-diagonal
formal $q$-complex $\mathfrak{D}_{q}\left(  x_{2}\right)  $. Finally, take
$h=x_{1}^{m}f\left(  y_{1}\right)  \otimes g\in x_{1}^{m}\left[  \left[
x_{2}\right]  \right]  \mathbb{C}\left[  \left[  y_{1},y_{2}\right]  \right]
$ with $g\in\mathcal{A}$. Then $\pi\left(  h\right)  \in\left[  \left[
x\right]  \right]  \mathbb{C}_{q}\left[  \left[  y\right]  \right]  $ such
that
\begin{align*}
\pi\left(  h\right)   &  =\pi\left(  \sum_{n}x_{1}^{m}f\left(  y_{1}\right)
\otimes x_{2}^{n}g_{n}\left(  y_{2}\right)  \right)  =\sum_{n}x^{m}f\left(
y\right)  x^{n}g_{n}\left(  y\right)  =\sum_{n}x^{m+n}f\left(  q^{n}y\right)
g_{n}\left(  y\right) \\
&  =x^{m}f\left(  y\right)  g_{0}\left(  y\right)  +x^{m+1}f\left(  qy\right)
g_{1}\left(  y\right)  +\cdots.
\end{align*}
Hence $\pi\left(  h\right)  =x^{m}h_{0}\left(  y,y\right)  +x^{m+1}%
h_{1}\left(  qy,y\right)  +\cdots$ holds for every $h=\sum_{n}x_{1}^{m}%
x_{2}^{n}h_{n}\left(  y_{1},y_{2}\right)  \in x_{1}^{m}\left[  \left[
x_{2}\right]  \right]  \mathbb{C}\left[  \left[  y_{1},y_{2}\right]  \right]
$, which means that $\pi$ has the following lower triangular operator matrix
\[
\pi=\left[
\begin{array}
[c]{ccc}%
\ddots &  & 0\\
& \pi_{x_{2}} & \\
\ast &  & \ddots
\end{array}
\right]  .
\]
Thus the differentials of the complex (\ref{Rsp}) have lower triangular
operator matrices, and the related diagonal complexes are reduced to the same
$\mathfrak{D}_{q}\left(  x_{2}\right)  $.

\subsection{Noncommutative analytic spaces $\mathbb{C}_{x}$ and $\mathbb{C}%
_{y}$\label{subsecHC}}

Now assume that $\left\vert q\right\vert <1$, and let $U_{x}\subseteq
\mathbb{C}_{x}$ be a $q$-open subset. The Fr\'{e}chet space $\mathcal{O}%
\left(  U_{x}\right)  \left[  \left[  y\right]  \right]  $ of all formal power
series in $y$ over the (commutative) Fr\'{e}chet algebra $\mathcal{O}\left(
U_{x}\right)  $ can be equipped with the $q$-multiplication from (\ref{5p2}).
It turns out that $\mathcal{O}\left(  U_{x}\right)  \left[  \left[  y\right]
\right]  $ is an Arens-Michael-Fr\'{e}chet algebra, and the canonical (left
representation) mapping
\[
l:\mathcal{O}_{q}\left(  \mathbb{C}_{xy}\right)  \rightarrow\mathcal{O}\left(
U_{x}\right)  \left[  \left[  y\right]  \right]  ,\quad f=\sum_{i,k}%
a_{ik}x^{i}y^{k}\mapsto l\left(  f\right)  =\sum_{n}\left(  \sum_{i}%
a_{in}z^{i}\right)  y^{n}%
\]
is a continuous algebra homomorphism. Thus we come up with a Fr\'{e}chet
$\mathcal{O}_{q}\left(  \mathbb{C}_{xy}\right)  $-algebra sheaf $\mathcal{O}%
\left[  \left[  y\right]  \right]  $ on the line $\mathbb{C}_{x}$, and
$\operatorname{Spec}\left(  \mathcal{O}\left(  U_{x}\right)  \left[  \left[
y\right]  \right]  \right)  =U_{x}$ (see \cite{Dosi25} for the details).
Moreover, $\mathcal{O}\left(  \mathbb{C}\right)  \left[  \left[  y\right]
\right]  $ is the algebra of all global sections of $\mathcal{O}\left[
\left[  y\right]  \right]  $, that is a multinormed envelope of $\mathfrak{A}%
_{q}$. Namely, $\mathcal{O}\left(  \mathbb{C}\right)  \left[  \left[
y\right]  \right]  $ is the universal object in the class of all Banach
algebra actions of $\mathfrak{A}_{q}$ with nilpotent images of $y$. We skip
the details, the multinormed envelopes are not discussed in the present work.

In a similar way, if $U_{y}\subseteq\mathbb{C}_{y}$ is a $q$-open subset, then
$\left[  \left[  x\right]  \right]  \mathcal{O}\left(  U_{y}\right)  $ is an
Arens-Michael-Fr\'{e}chet algebra equipped with the $q$-multiplication
(\ref{pp4}), and the canonical (right representation) mapping
\[
r:\mathcal{O}_{q}\left(  \mathbb{C}_{xy}\right)  \rightarrow\left[  \left[
x\right]  \right]  \mathcal{O}\left(  U_{y}\right)  ,\quad f=\sum_{i,k}%
a_{ik}x^{i}y^{k}\mapsto r\left(  f\right)  =\sum_{n}x^{n}\left(  \sum
_{k}a_{nk}w^{k}\right)
\]
is a continuous algebra homomorphism. As above, $\left[  \left[  x\right]
\right]  \mathcal{O}$ is a Fr\'{e}chet $\mathcal{O}_{q}\left(  \mathbb{C}%
_{xy}\right)  $-algebra sheaf on the complex line $\mathbb{C}_{y}$ with
$\operatorname{Spec}\left(  \left[  \left[  x\right]  \right]  \mathcal{O}%
\left(  U_{y}\right)  \right)  =U_{y}$ for every $q$-open subset
$U_{y}\subseteq\mathbb{C}_{y}$, and $\left[  \left[  x\right]  \right]
\mathcal{O}\left(  \mathbb{C}\right)  $ is a certain multinormed envelope of
$\mathfrak{A}_{q}$.

\begin{definition}
\label{def2}The ringed spaces $\left(  \mathbb{C}_{x},\mathcal{O}\left[
\left[  y\right]  \right]  \right)  $ and $\left(  \mathbb{C}_{y},\left[
\left[  x\right]  \right]  \mathcal{O}\right)  $ are called noncommutative
analytic spaces of the multinormed envelopes $\mathcal{O}\left(
\mathbb{C}\right)  \left[  \left[  y\right]  \right]  $ and $\left[  \left[
x\right]  \right]  \mathcal{O}\left(  \mathbb{C}\right)  $ of the contractive
quantum plane $\mathfrak{A}_{q}$, respectively.
\end{definition}

Thus $\left(  \mathbb{C}_{x},\mathcal{O}\left[  \left[  y\right]  \right]
\right)  $ and $\left(  \mathbb{C}_{y},\left[  \left[  x\right]  \right]
\mathcal{O}\right)  $ are the geometric models of $\mathcal{O}\left(
\mathbb{C}\right)  \left[  \left[  y\right]  \right]  $ and $\left[  \left[
x\right]  \right]  \mathcal{O}\left(  \mathbb{C}\right)  $ (or $\mathfrak{A}%
_{q}$) representing the related complex lines $\mathbb{C}_{x}$ and
$\mathbb{C}_{y}$ within $\mathbb{C}_{xy}$, respectively. The canonical
inclusion of their intersection point $\left\{  \left(  0,0\right)  \right\}
$ (see Definition \ref{def1}) into these lines define the morphisms of the
related ringed spaces considered below in Section \ref{sectionFGq}. The
transversality relation of a left Banach $\mathcal{O}_{q}\left(
\mathbb{C}_{xy}\right)  $-module $X$ with respect to the sheaves
$\mathcal{O}\left[  \left[  y\right]  \right]  $ and $\left[  \left[
x\right]  \right]  \mathcal{O}$ are solved in terms of the Taylor spectrum too
\cite[Proposition 6.4]{Dosi25}.

The localization property of these structure sheaves is solved in the same
pattern as it was for $\mathbb{C}_{q}\left[  \left[  x,y\right]  \right]  $.
Namely, the identification (\ref{AlI}) for $\mathcal{A}=\mathcal{O}\left(
U_{x}\right)  \left[  \left[  y\right]  \right]  $ is reduced to
\[
\mathcal{A}^{\widehat{\otimes}2}=\mathcal{O}\left(  U_{x_{1}}\right)  \left[
\left[  y_{1}\right]  \right]  \widehat{\otimes}\mathcal{O}\left(  U_{x_{2}%
}\right)  \left[  \left[  y_{2}\right]  \right]  =\mathcal{O}\left(  U_{x_{1}%
}\times U_{x_{2}}\right)  \left[  \left[  y_{1}\right]  \right]  \left[
\left[  y_{2}\right]  \right]
\]
and the structure of the considered operators are the same. The complex
(\ref{Rsp}) is reduced to the following
\[
0\rightarrow\mathcal{O}\left(  U_{x_{1}}\times U_{x_{2}}\right)  \left[
\left[  y_{1}\right]  \right]  \left[  \left[  y_{2}\right]  \right]
\overset{d^{0}}{\longrightarrow}\left(
\begin{array}
[c]{c}%
\mathcal{O}\left(  U_{x_{1}}\times U_{x_{2}}\right)  \left[  \left[
y_{1}\right]  \right] \\
\oplus\\
\mathcal{O}\left(  U_{x_{1}}\times U_{x_{2}}\right)  \left[  \left[
y_{1}\right]  \right]
\end{array}
\right)  \left[  \left[  y_{2}\right]  \right]  \overset{d^{1}%
}{\longrightarrow}%
\]%
\[
\overset{d^{1}}{\longrightarrow}\mathcal{O}\left(  U_{x_{1}}\times U_{x_{2}%
}\right)  \left[  \left[  y_{1}\right]  \right]  \left[  \left[  y_{2}\right]
\right]  \overset{\pi}{\longrightarrow}\mathcal{A}\rightarrow0
\]
with the same type of differentials. The complex (\ref{Rsp}) can also be
transformed for the algebra $\mathcal{A}=\left[  \left[  x\right]  \right]
\mathcal{O}\left(  U_{y}\right)  $.

\begin{proposition}
\label{propLoc2}The canonical homomorphisms $\mathfrak{A}_{q}\rightarrow
\mathcal{O}\left(  U_{x}\right)  \left[  \left[  y\right]  \right]  $ and
$\mathfrak{A}_{q}\rightarrow\left[  \left[  x\right]  \right]  \mathcal{O}%
\left(  U_{y}\right)  $ are localizations for all $q$-open subsets
$U_{x}\subseteq\mathbb{C}_{x}$ and $U_{y}\subseteq\mathbb{C}_{y}$.
\end{proposition}

\begin{proof}
As in the proof of Proposition \ref{propLoc1}, one needs to prove that the
complex (\ref{Rsp}) for $\mathcal{A}=\mathcal{O}\left(  U_{x}\right)  \left[
\left[  y\right]  \right]  $ (resp., $\mathcal{A}=\left[  \left[  x\right]
\right]  \mathcal{O}\left(  U_{y}\right)  $) is exact. Using the same
triangular representations for the differentials of the complex (\ref{Rsp}),
we come up with the diagonal holomorphic $q$-complexes $\mathcal{O}_{q}\left(
y\right)  $ (resp., $\mathcal{O}_{q}\left(  x\right)  $) considered above in
Subsection \ref{subsecDHC}. By Proposition \ref{propLDc1}, the diagonal
holomorphic $q$-complexes $\mathcal{O}_{q}\left(  y\right)  $ and
$\mathcal{O}_{q}\left(  x\right)  $ are exact. Using again Lemma \ref{lemTri},
we deduce that (\ref{Rsp}) remains exact.
\end{proof}

Thus both noncommutative analytic spaces $\left(  \mathbb{C}_{x}%
,\mathcal{O}\left[  \left[  y\right]  \right]  \right)  $ and $\left(
\mathbb{C}_{y},\left[  \left[  x\right]  \right]  \mathcal{O}\right)  $
possess the localization property.

\section{The formal geometry of the contractive quantum plane}

In this section we briefly review the formal geometry of $\mathfrak{A}_{q}$
given by the noncommutative analytic space $\left(  \mathbb{C}_{xy}%
,\mathcal{F}_{q}\right)  $ \cite{Dosi25}, and consider the Fr\'{e}chet space
sheaves $\mathcal{O}\left(  d\right)  $, $d\in\mathbb{Z}_{+}$ on the
commutative analytic space $\left(  \mathbb{C}_{xy},\mathcal{O}\right)  $, and
some key operators acting over them. One needs to refer to these sheaves as
the commutative building blocks of the structure sheaf $\mathcal{F}_{q}$.

\subsection{Noncommutative analytic space $\mathbb{C}_{xy}$\label{subsecNCCP}}

The canonical inclusions of the intersection point $\left(  0,0\right)  $ into
the complex lines $\mathbb{C}_{x}$ and $\mathbb{C}_{y}$ define the morphisms%
\[
\left(  \left\{  \left(  0,0\right)  \right\}  ,\mathbb{C}_{q}\left[  \left[
x,y\right]  \right]  \right)  \rightarrow\left(  \mathbb{C}_{x},\mathcal{O}%
\left[  \left[  y\right]  \right]  \right)  \text{\quad and\quad}\left(
\left\{  \left(  0,0\right)  \right\}  ,\mathbb{C}_{q}\left[  \left[
x,y\right]  \right]  \right)  \rightarrow\left(  \mathbb{C}_{y},\left[
\left[  x\right]  \right]  \mathcal{O}\right)
\]
of the ringed spaces, respectively. They are morphisms of the noncommutative
analytic spaces introduced in Definitions \ref{def1} and \ref{def2}. The
related morphisms of the sheaves are given by the power series expansions of
the related sections about the intersection point. Namely, if $U_{x}%
\subseteq\mathbb{C}_{x}$ and $U_{y}\subseteq\mathbb{C}_{y}$ are $q$-open
subsets, then the continuous algebra homomorphisms
\begin{align*}
s_{x}  &  :\mathcal{O}\left(  U_{x}\right)  \left[  \left[  y\right]  \right]
\rightarrow\mathbb{C}_{q}\left[  \left[  x,y\right]  \right]  ,\quad
s_{x}\left(  \sum_{n}f_{n}\left(  z\right)  y^{n}\right)  =\sum_{s,i}%
\dfrac{f_{i}^{\left(  s\right)  }\left(  0\right)  }{s!}x^{s}y^{i},\\
s_{y}  &  :\left[  \left[  x\right]  \right]  \mathcal{O}\left(  U_{y}\right)
\rightarrow\mathbb{C}_{q}\left[  \left[  x,y\right]  \right]  ,\quad
s_{y}\left(  \sum_{n}x^{n}g_{n}\left(  w\right)  \right)  =\sum_{t,j}%
\dfrac{g_{t}^{\left(  j\right)  }\left(  0\right)  }{j!}x^{t}y^{j}%
\end{align*}
define the the related sheaf morphisms.

Now let us consider the whole space $\mathbb{C}_{xy}$ equipped withe the
$\mathfrak{q}$-topology (see Subsection \ref{subsecAMEC}). If $U\subseteq
\mathbb{C}_{xy}$ is a nonempty $q$-open subset, then it has two (nonempty)
$q$-open components $U_{x}\subseteq\mathbb{C}_{x}$ and $U_{y}\subseteq
\mathbb{C}_{y}$. We define the Arens-Michael-Fr\'{e}chet algebra
$\mathcal{F}_{q}\left(  U\right)  $ to be the fibered product $\mathcal{O}%
\left(  U_{x}\right)  \left[  \left[  y\right]  \right]  \underset{\mathbb{C}%
\left[  \left[  x,y\right]  \right]  }{\times}\left[  \left[  x\right]
\right]  \mathcal{O}\left(  U_{y}\right)  $ defined in terms of the following
commutative diagram%
\[%
\begin{array}
[c]{ccccc}
&  & \mathcal{F}_{q}\left(  U\right)  &  & \\
& ^{\pi_{x}}\swarrow &  & \searrow^{\pi_{y}} & \\
\mathcal{O}\left(  U_{x}\right)  \left[  \left[  y\right]  \right]  &  &  &  &
\left[  \left[  x\right]  \right]  \mathcal{O}\left(  U_{y}\right) \\
& _{s_{x}}\searrow &  & \swarrow_{s_{y}} & \\
&  & \mathbb{C}_{q}\left[  \left[  x,y\right]  \right]  &  &
\end{array}
\]
with the canonical projections $\pi_{x}$, $\pi_{y}$, and the continuous
algebra homomorphisms $s_{x}$ and $s_{y}$ introduced above. It turns out that
\begin{equation}
\mathcal{F}_{q}\left(  U\right)  =\left\{  \left(  f,g\right)  \in
\mathcal{O}\left(  U_{x}\right)  \left[  \left[  y\right]  \right]
\oplus\left[  \left[  x\right]  \right]  \mathcal{O}\left(  U_{y}\right)
:\dfrac{f_{k}^{\left(  i\right)  }\left(  0\right)  }{i!}=\dfrac
{g_{i}^{\left(  k\right)  }\left(  0\right)  }{k!},i,k\in\mathbb{Z}%
_{+}\right\}  . \label{FqF}%
\end{equation}
Thus $\mathcal{F}_{q}=\mathcal{O}\left[  \left[  y\right]  \right]
\underset{\mathbb{C}_{q}\left[  \left[  x,y\right]  \right]  }{\times}\left[
\left[  x\right]  \right]  \mathcal{O}$ is a Fr\'{e}chet algebra sheaf on
$\mathbb{C}_{xy}$ being the fibered product of the Fr\'{e}chet algebra sheaves
$\mathcal{O}\left[  \left[  y\right]  \right]  $ and $\left[  \left[
x\right]  \right]  \mathcal{O}$ over the constant sheaf $\mathbb{C}_{q}\left[
\left[  x,y\right]  \right]  $. It can be proven that $\mathcal{F}_{q}\left(
\mathbb{C}_{xy}\right)  $ is reduced to the PI-envelope of $\mathfrak{A}_{q}$
(we skip the details).

\begin{definition}
\label{def3}The ringed space $\left(  \mathbb{C}_{xy},\mathcal{F}_{q}\right)
$ is called a noncommutative analytic space of the PI-envelope $\mathcal{F}%
_{q}\left(  \mathbb{C}_{xy}\right)  $ of the contractive quantum plane
$\mathfrak{A}_{q}$.
\end{definition}

Thus $\left(  \mathbb{C}_{xy},\mathcal{F}_{q}\right)  $ is a standard
geometric model of the PI-envelope of $\mathfrak{A}_{q}$ called also the
formal geometry of $\mathfrak{A}_{q}$. The canonical inclusions of the complex
lines $\mathbb{C}_{x}$ and $\mathbb{C}_{y}$ into the plane $\mathbb{C}_{xy}$
followed by the inclusions of the intersection point $\left(  0,0\right)  $
into these complex lines $\mathbb{C}_{x}$ and $\mathbb{C}_{y}$ define the
related morphisms of noncommutative analytic spaces%
\[%
\begin{array}
[c]{ccccc}
&  & \mathbb{C}_{xy} &  & \\
& \nearrow &  & \nwarrow & \\
\mathbb{C}_{x} &  &  &  & \mathbb{C}_{y}\\
& \nwarrow &  & \nearrow & \\
&  & \left\{  \left(  0,0\right)  \right\}  &  &
\end{array}
,\quad%
\begin{array}
[c]{ccccc}
&  & \mathcal{F}_{q} &  & \\
& \swarrow &  & \searrow & \\
\mathcal{O}\left[  \left[  y\right]  \right]  &  &  &  & \left[  \left[
x\right]  \right]  \mathcal{O}\\
& \searrow &  & \swarrow & \\
&  & \mathbb{C}_{q}\left[  \left[  x,y\right]  \right]  &  &
\end{array}
\]
Moreover, the morphisms $l$ and $r$ (see Subsection \ref{subsecHC}) define the
Fr\'{e}chet algebra sheaf morphism $\mathcal{O}_{q}\left(  \mathbb{C}%
_{xy}\right)  \rightarrow\mathcal{F}_{q}$, that turns $\mathcal{F}_{q}$ into a
Fr\'{e}chet algebra sheaf over $\mathcal{O}_{q}\left(  \mathbb{C}_{xy}\right)
$ (see \cite{Dosi25} for the details).

\subsection{The Fr\'{e}chet space sheaves $\mathcal{O}\left(  d\right)  $ on
$\mathbb{C}_{xy}$\label{subsecOQ}}

Now we consider the commutative analytic space $\left(  \mathbb{C}%
_{xy},\mathcal{O}\right)  $ of the plane $\mathbb{C}_{xy}$ equipped with the
$\mathfrak{q}$-topology. The standard commutative Fr\'{e}chet algebra sheaf
$\mathcal{O}$ of germs of holomorphic functions on $\mathbb{C}_{xy}$ can be
treated as the fibered product $\mathcal{O}_{x}\underset{\mathbb{C}}{\times
}\mathcal{O}_{y}$, where $\mathcal{O}_{x}$ and $\mathcal{O}_{y}$ are copies
(on $\mathbb{C}_{x}$ and $\mathbb{C}_{y}$, respectively) of the standard sheaf
$\mathcal{O}$ on $\mathbb{C}$. Namely, for every $q$-open subset
$U\subseteq\mathbb{C}_{xy}$ we put
\[
\mathcal{O}\left(  U\right)  =\mathcal{O}\left(  U_{x}\right)
\underset{\mathbb{C}}{\times}\mathcal{O}\left(  U_{y}\right)  =\left\{
\left(  f\left(  z\right)  ,g\left(  w\right)  \right)  \in\mathcal{O}\left(
U_{x}\right)  \oplus\mathcal{O}\left(  U_{y}\right)  :f\left(  0\right)
=g\left(  0\right)  \right\}
\]
to be a closed subalgebra of the Fr\'{e}chet sum $\mathcal{O}\left(
U_{x}\right)  \oplus\mathcal{O}\left(  U_{y}\right)  $. In this case, the
equality $\operatorname{Spec}\left(  \mathcal{O}\left(  U\right)  \right)  =U$
holds up to the canonical identification (see \cite{Dosi25}).

Let $\left\{  \mathcal{O}_{x}\left(  d\right)  \right\}  $ be a filtration of
$\mathcal{O}_{x}$ given by the subsheaves of the $z^{d}$-principal ideals (see
Subsection \ref{subsecFSO}). In a similar way, there is a subsheaf filtration
$\left\{  \mathcal{O}_{y}\left(  d\right)  \right\}  $ of $w^{d}$-principal
ideals of $\mathcal{O}_{y}$ on $\mathbb{C}_{y}$. The continuous (evaluation)
linear maps
\begin{align*}
s_{d}  &  :\mathcal{O}_{x}\left(  d\right)  \rightarrow\mathbb{C},\quad
f\left(  z\right)  \mapsto\left(  d!\right)  ^{-1}f^{\left(  d\right)
}\left(  0\right)  ,\\
t_{d}  &  :\mathcal{O}_{y}\left(  d\right)  \rightarrow\mathbb{C},\quad
g\left(  w\right)  \mapsto\left(  d!\right)  ^{-1}g^{\left(  d\right)
}\left(  0\right)
\end{align*}
define the fibered product
\[
\mathcal{O}\left(  d\right)  =\mathcal{O}_{x}\left(  d\right)
\underset{\mathbb{C}}{\times}\mathcal{O}_{y}\left(  d\right)
\]
to be a Fr\'{e}chet space sheaf on $\mathbb{C}_{xy}$. It is uniquely given by
the following commutative diagram
\[%
\begin{array}
[c]{ccccc}
&  & \mathcal{O}\left(  d\right)  &  & \\
& ^{p_{x}}\swarrow &  & \searrow^{p_{y}} & \\
\mathcal{O}_{x}\left(  d\right)  &  &  &  & \mathcal{O}_{y}\left(  d\right) \\
& _{s_{d}}\searrow &  & \swarrow_{t_{d}} & \\
&  & \mathbb{C} &  &
\end{array}
\]
with the canonical projections $p_{x}$ and $p_{y}$. Notice that $\mathcal{O}%
\left(  0\right)  \left(  U\right)  =\mathcal{O}\left(  U_{x}\right)
\underset{\mathbb{C}}{\times}\mathcal{O}\left(  U_{y}\right)  =\mathcal{O}%
\left(  U\right)  $ is the algebra of all holomorphic functions on a $q$-open
subset $U\subseteq\mathbb{C}_{xy}$. But for every $f\in\mathcal{O}_{x}\left(
d\right)  \left(  U_{x}\right)  $, we have
\[
s_{0}\left(  z^{-d}f\left(  z\right)  \right)  =\left(  z^{-d}f\left(
z\right)  \right)  |_{z=0}=\left(  d!\right)  ^{-1}f^{\left(  d\right)
}\left(  0\right)  =s_{d}\left(  f\left(  z\right)  \right)  .
\]
In a similar way, we have $t_{0}\left(  w^{-d}g\left(  w\right)  \right)
=t_{d}\left(  g\left(  w\right)  \right)  $ for all $g\in\mathcal{O}%
_{y}\left(  d\right)  \left(  U_{y}\right)  $. Thus the sheaf isomorphisms
$z^{-d}:\mathcal{O}_{x}\left(  d\right)  \rightarrow\mathcal{O}_{x}$ and
$w^{-d}:\mathcal{O}_{y}\left(  d\right)  \rightarrow\mathcal{O}_{y}$ from
Lemma \ref{lemMD}, and the identity map $1:\mathbb{C}\rightarrow\mathbb{C}$ of
the constant sheaves are compatible with the evaluations maps. They in turn
define the following topological isomorphism
\begin{equation}
z^{-d}\times_{1}w^{-d}:\mathcal{O}\left(  d\right)  \rightarrow\mathcal{O}%
,\quad\left(  f\left(  z\right)  ,g\left(  w\right)  \right)  \mapsto\left(
z^{-d}f\left(  z\right)  ,w^{-d}g\left(  w\right)  \right)  \label{imoo}%
\end{equation}
of the Fr\'{e}chet sheaves called \textit{the }$d$\textit{-isomorphism}, whose
inverse is given by $z^{d}\times_{1}w^{d}$. Thus $\left\{  \mathcal{O}\left(
d\right)  \right\}  $ are isomorphic copies of the Fr\'{e}chet sheaf
$\mathcal{O}$ on $\mathbb{C}_{xy}$. Everywhere below we also use the notation
$\mathcal{O}_{d}$ instead of $\mathcal{O}$ whenever we wish to indicate that
it is the $d$-isomorphic copy of $\mathcal{O}\left(  d\right)  $ through
(\ref{imoo}).

\subsection{The operators over the filtration $\left\{  \mathcal{O}\left(
d\right)  \right\}  $\label{subsecOF}}

The multiplication operator $z$ on $\mathcal{O}_{x}\left(  d\right)  $ (resp.,
$w$ on $\mathcal{O}_{y}\left(  d\right)  $) is converted into the
multiplication operator by $\left(  z,0\right)  $ (resp., $\left(  0,w\right)
$) on $\mathcal{O}\left(  d\right)  $. Namely,%
\[
s_{d}\left(  zf\left(  z\right)  \right)  =\left(  d!\right)  ^{-1}\left(
zf\left(  z\right)  \right)  ^{\left(  d\right)  }\left(  0\right)  =\left(
\left(  d-1\right)  !\right)  ^{-1}f^{\left(  d-1\right)  }\left(  0\right)
=0
\]
for all $f\in\mathcal{O}_{x}\left(  d\right)  \left(  U_{x}\right)  $, which
means that $z\times_{0}0:\mathcal{O}\left(  d\right)  \left(  U\right)
\rightarrow\mathcal{O}\left(  d\right)  \left(  U\right)  $ is a well defined
continuous linear operator such that $p_{x}\left(  z\times_{0}0\right)
=zp_{x}$ and $p_{y}\left(  z\times_{0}0\right)  =0$. In a similar way, we have
the operator $0\times_{0}w$. Notice that $z\times_{0}0=\left(  z,0\right)  $
and $0\times_{0}w=\left(  0,w\right)  $ are the multiplication operators. For
brevity we use the notation $z$ (resp., $w$) instead of $\left(  z,0\right)  $
(resp., $\left(  0,w\right)  $) too. Thus%
\begin{equation}
z,w:\mathcal{O}\left(  d\right)  \rightarrow\mathcal{O}\left(  d\right)
,\quad z\left(  f\left(  z\right)  ,g\left(  w\right)  \right)  =\left(
zf\left(  z\right)  ,0\right)  ,\quad w\left(  f\left(  z\right)  ,g\left(
w\right)  \right)  =\left(  0,wg\left(  w\right)  \right)  \label{RzRw1}%
\end{equation}
are the multiplication operators. The isomorphism (\ref{imoo}) preserves both
operators $z$ and $w$ (see Lemma \ref{lemMD}). Thus $z,w:\mathcal{O}%
_{d}\rightarrow\mathcal{O}_{d}$ are acting in the same way as (\ref{RzRw1}).

Now consider the following continuous linear operators
\begin{equation}
N_{x,d},\text{ }N_{y,d}:\mathcal{O}\left(  d\right)  \rightarrow
\mathcal{O}\left(  d+1\right)  ,\quad M_{x,l},\text{ }M_{y,l}:\mathcal{O}%
\left(  l\right)  \rightarrow\mathcal{O}\left(  l+1\right)  \label{Nd}%
\end{equation}%
\begin{align*}
N_{x,d}\left(  \zeta\right)   &  =\left(  P_{d+1}\left(  f\left(  z\right)
\right)  ,\dfrac{f^{d+1}\left(  0\right)  }{\left(  d+1\right)  !}%
w^{d+1}\right)  ,\quad N_{y,d}\left(  \zeta\right)  =\left(  q^{d+1}%
\dfrac{g^{d+1}\left(  0\right)  }{\left(  d+1\right)  !}z^{d+1},P_{d+1}%
\Delta_{q}\left(  g\left(  w\right)  \right)  \right)  ,\\
M_{x,l}\left(  \zeta\right)   &  =\left(  \dfrac{g^{l+1}\left(  0\right)
}{\left(  l+1\right)  !}z^{l+1},P_{l+1}\left(  g\right)  \left(  w\right)
\right)  ,\quad M_{y,l}\left(  \zeta\right)  =\left(  P_{l+1}\Delta_{q}\left(
f\left(  z\right)  \right)  ,q^{l+1}\dfrac{f^{l+1}\left(  0\right)  }{\left(
l+1\right)  !}w^{l+1}\right)
\end{align*}
for all $\zeta=\left(  f\left(  z\right)  ,g\left(  w\right)  \right)  $ from
$\mathcal{O}\left(  d\right)  \left(  U\right)  $ or $\mathcal{O}\left(
l\right)  \left(  U\right)  $, $d,l\in\mathbb{Z}_{+}$, where $P_{d+1}$ (resp.,
$P_{l+1}$) is the projection onto $\mathcal{O}_{x}\left(  d+1\right)  \left(
U_{x}\right)  $ or $\mathcal{O}_{y}\left(  d+1\right)  \left(  U_{y}\right)  $
(resp., $\mathcal{O}_{x}\left(  l+1\right)  \left(  U_{x}\right)  $ or
$\mathcal{O}_{y}\left(  l+1\right)  \left(  U_{y}\right)  $) from the proof of
Lemma \ref{lemMD}, and $\Delta_{q}$ was defined in (\ref{Delq}). Notice that
\[
P_{d+1}\left(  f\left(  z\right)  \right)  =f\left(  z\right)  -\dfrac
{f^{d}\left(  0\right)  }{d!}z^{d}\in\mathcal{O}_{x}\left(  d+1\right)
\left(  U_{x}\right)  ,\quad\dfrac{f^{d+1}\left(  0\right)  }{\left(
d+1\right)  !}w^{d+1}\in\mathcal{O}_{y}\left(  d+1\right)  \left(
U_{y}\right)  ,
\]
and
\[
s_{d+1}\left(  P_{d+1}\left(  f\left(  z\right)  \right)  \right)
=\dfrac{f^{d+1}\left(  0\right)  }{\left(  d+1\right)  !}=t_{d+1}\left(
\dfrac{f^{d+1}\left(  0\right)  }{\left(  d+1\right)  !}w^{d+1}\right)  ,
\]
which means that $N_{x,d}$ is a well defined continuous linear operator. In a
similar way, we have
\[
q^{d+1}\dfrac{g^{d+1}\left(  0\right)  }{\left(  d+1\right)  !}z^{d+1}%
\in\mathcal{O}_{x}\left(  d+1\right)  \left(  U_{x}\right)  ,\quad
P_{d+1}\Delta_{q}\left(  g\left(  w\right)  \right)  =g\left(  qw\right)
-q^{d}\dfrac{g^{d}\left(  0\right)  }{d!}w^{d}\in\mathcal{O}_{y}\left(
d+1\right)  \left(  U_{y}\right)  ,
\]
and
\[
s_{d+1}\left(  q^{d+1}\dfrac{g^{d+1}\left(  0\right)  }{\left(  d+1\right)
!}z^{d+1}\right)  =q^{d+1}\dfrac{g^{d+1}\left(  0\right)  }{\left(
d+1\right)  !}=t_{d+1}\left(  P_{d+1}\Delta_{q}\left(  g\left(  w\right)
\right)  \right)  ,
\]
which means that $N_{y,d}$ is well defined. In the same way, we can verify
that $M_{x,l}$, $M_{y,l}$ are well defined operators too.

\begin{lemma}
\label{lemTech1}The operators from (\ref{Nd}) are identified with the
operators $N_{x,d},$ $N_{y,d}:\mathcal{O}_{d}\rightarrow\mathcal{O}_{d+1}$,
$M_{x,l},$ $M_{y,l}:\mathcal{O}_{l}\rightarrow\mathcal{O}_{l+1}$ up the
topological isomorphisms (\ref{imoo}), and they act in the following ways
\begin{align*}
N_{x,d}\left(  \zeta\right)   &  =\left(  \frac{f\left(  z\right)  -f\left(
0\right)  }{z},f^{\prime}\left(  0\right)  \right)  ,\quad N_{y,d}\left(
\zeta\right)  =q^{d+1}\left(  g^{\prime}\left(  0\right)  ,\frac{g\left(
qw\right)  -g\left(  0\right)  }{qw}\right)  ,\\
M_{x,l}\left(  \zeta\right)   &  =\left(  g^{\prime}\left(  0\right)
,\frac{g\left(  w\right)  -g\left(  0\right)  }{w}\right)  ,\quad
M_{y,l}\left(  \zeta\right)  =q^{l+1}\left(  \frac{f\left(  qz\right)
-f\left(  0\right)  }{qz},f^{\prime}\left(  0\right)  \right)
\end{align*}
for every $\zeta=\left(  f,g\right)  $.
\end{lemma}

\begin{proof}
Take $\zeta=\left(  f,g\right)  \in\mathcal{O}\left(  U\right)  $, which is
identified with $\left(  z^{-d}\times_{1}w^{-d}\right)  \left(  F,G\right)  $
for some $\left(  F,G\right)  \in\mathcal{O}\left(  d\right)  \left(
U\right)  $. Then $F\left(  z\right)  =z^{d}f\left(  z\right)  $, $G\left(
w\right)  =w^{d}g\left(  w\right)  $ and $N_{x,d}\left(  \zeta\right)  $ is
identified with the element $\left(  z^{-d-1}\times_{1}w^{-d-1}\right)
\left(  N_{x,d}\left(  F,G\right)  \right)  $. Namely,
\begin{align*}
N_{x,d}\left(  \zeta\right)   &  =\left(  z^{-d-1}\times_{1}w^{-d-1}\right)
\left(  P_{d+1}\left(  F\left(  z\right)  \right)  ,\dfrac{F^{d+1}\left(
0\right)  }{\left(  d+1\right)  !}w^{d+1}\right) \\
&  =\left(  z^{-d-1}\left(  F\left(  z\right)  -\dfrac{F^{d}\left(  0\right)
}{d!}z^{d}\right)  ,\dfrac{F^{d+1}\left(  0\right)  }{\left(  d+1\right)
!}\right)  =\left(  \left(  f\left(  z\right)  -f\left(  0\right)  \right)
z^{-1},f^{\prime}\left(  0\right)  \right)  .
\end{align*}
In a similar way, we have
\begin{align*}
N_{y,d}\left(  \zeta\right)   &  =\left(  z^{-d-1}\times_{1}w^{-d-1}\right)
\left(  q^{d+1}\dfrac{G^{d+1}\left(  0\right)  }{\left(  d+1\right)  !}%
z^{d+1},P_{d+1}\Delta_{q}\left(  G\left(  w\right)  \right)  \right) \\
&  =\left(  q^{d+1}g^{\prime}\left(  0\right)  ,w^{-d-1}G\left(  qw\right)
-q^{d}\dfrac{G^{d}\left(  0\right)  }{d!}w^{-1}\right)  =\left(
q^{d+1}g^{\prime}\left(  0\right)  ,q^{d}w^{-1}g\left(  qw\right)
-q^{d}g\left(  0\right)  w^{-1}\right) \\
&  =q^{d+1}\left(  g^{\prime}\left(  0\right)  ,\frac{g\left(  qw\right)
-g\left(  0\right)  }{qw}\right)  .
\end{align*}
The case of operators $M_{x,l}$ and $M_{y,l}$ are similar to $N_{y,d}$ and
$N_{x,d}$, respectively.
\end{proof}

\subsection{The decomposition theorem}

Now consider again the noncommutative analytic space $\left(  \mathbb{C}%
_{xy},\mathcal{F}_{q}\right)  $ (see Definition \ref{def3}). We briefly review
the main decomposition theorem of the sheaf $\mathcal{F}_{q}$ from
\cite{Dosi25}. For every integer $d\in\mathbb{Z}_{+}$ we define the following
closed subspaces
\begin{align*}
\mathcal{O}\left(  U_{x}\right)  \left[  \left[  y\right]  \right]  ^{d}  &
=\mathcal{O}\left(  U_{x}\right)  y^{d}\times\prod\limits_{k>d}\left(
\mathbb{C}z^{d}\right)  y^{k}\subseteq\mathcal{O}\left(  U_{x}\right)  \left[
\left[  y\right]  \right]  \text{,}\\
\left[  \left[  x\right]  \right]  ^{d}\mathcal{O}\left(  U_{y}\right)   &
=x^{d}\mathcal{O}\left(  U_{y}\right)  \times\prod\limits_{i>d}x^{i}\left(
\mathbb{C}w^{d}\right)  \subseteq\left[  \left[  x\right]  \right]
\mathcal{O}\left(  U_{y}\right)  ,
\end{align*}
where $\mathbb{C}z^{d}\subseteq\mathcal{O}\left(  U_{x}\right)  $ (resp.,
$\mathbb{C}w^{d}\subseteq\mathcal{O}\left(  U_{y}\right)  $) is the subspace
generated by the monomial. They consist of certain elements of the order $d$
(the least nonzero term occurs at $d$). We also put
\[
\mathcal{F}_{q}\left(  U\right)  ^{d}=\mathcal{O}\left(  U_{x}\right)  \left[
\left[  y\right]  \right]  ^{d}\underset{\mathbb{C}\left[  \left[  x,y\right]
\right]  }{\times}\left[  \left[  x\right]  \right]  ^{d}\mathcal{O}\left(
U_{y}\right)
\]
to be a closed subspace of $\mathcal{F}_{q}\left(  U\right)  $. For every
$d\in\mathbb{Z}_{+}$ the following continuous linear operator
\begin{align*}
p_{d}\left(  U\right)   &  :\mathcal{F}_{q}\left(  U\right)  \rightarrow
\mathcal{F}_{q}\left(  U\right)  ,\quad p_{d}\left(  U\right)  \left(
f,g\right)  =\left(  \overline{f},\overline{g}\right)  ,\\
\overline{f}  &  =\left(  f_{d}\left(  z\right)  -\sum_{i=0}^{d-1}\dfrac
{g_{i}^{\left(  d\right)  }\left(  0\right)  }{d!}z^{i}\right)  y^{d}%
+\sum_{k>d}\left(  \dfrac{g_{d}^{\left(  k\right)  }\left(  0\right)  }%
{k!}z^{d}\right)  y^{k},\\
\overline{g}  &  =x^{d}\left(  g_{d}\left(  w\right)  -\sum_{i=0}^{d-1}%
\dfrac{f_{i}^{\left(  d\right)  }\left(  0\right)  }{d!}w^{i}\right)
+\sum_{i>d}x^{i}\left(  \dfrac{f_{d}^{\left(  i\right)  }\left(  0\right)
}{i!}w^{d}\right)  .
\end{align*}
defines a continuous projection onto $\mathcal{F}_{q}\left(  U\right)  ^{d}$
such that
\[
p\left(  U\right)  =\prod_{d\in\mathbb{Z}_{+}}p_{d}\left(  U\right)
:\mathcal{F}_{q}\left(  U\right)  \rightarrow\prod_{d\in\mathbb{Z}_{+}%
}\mathcal{F}_{q}\left(  U\right)  ^{d},\quad p\left(  U\right)  \left(
\zeta\right)  =\left(  p_{d}\left(  U\right)  \left(  \zeta\right)  \right)
_{d}%
\]
implements a topological isomorphism of $\mathcal{F}_{q}\left(  U\right)  $
onto the direct product of the Fr\'{e}chet spaces. Moreover, the linear
mapping $\alpha_{d}:\mathcal{O}\left(  d\right)  \left(  U\right)
\rightarrow\mathcal{F}_{q}\left(  U\right)  ^{d},$
\[
\alpha_{d}\left(  f,g\right)  =\left(  f\left(  z\right)  y^{d}+\sum
_{k>d}\left(  \dfrac{g^{\left(  k\right)  }\left(  0\right)  }{k!}%
z^{d}\right)  y^{k},x^{d}g\left(  w\right)  +\sum_{i>d}x^{i}\left(
\dfrac{f^{\left(  i\right)  }\left(  0\right)  }{i!}w^{d}\right)  \right)
\]
implements a topological isomorphism of the related Fr\'{e}chet\ spaces.

\begin{theorem}
\label{thmainDecom}Let $U\subseteq\mathbb{C}_{xy}$ be a $q$-open subset. The
mapping $\Lambda\left(  U\right)  :\mathcal{F}_{q}\left(  U\right)
\rightarrow\mathcal{O}\left(  U\right)  $, $\Lambda\left(  U\right)  \left(
f,g\right)  =\left(  f_{0},g_{0}\right)  $ is a retraction of the Fr\'{e}chet
spaces with its right inverse $\alpha_{0}\left(  U\right)  :\mathcal{O}\left(
U\right)  \rightarrow\mathcal{F}_{q}\left(  U\right)  ^{0}$, which allows us
to identify $\mathcal{O}\left(  U\right)  $ with the complemented subspace
$\mathcal{F}_{q}\left(  U\right)  ^{0}$ of $\mathcal{F}_{q}\left(  U\right)
$. In this case,
\[
\mathcal{F}_{q}\left(  U\right)  =\mathcal{O}\left(  U\right)  \oplus
\operatorname{Rad}\mathcal{F}_{q}\left(  U\right)
\]
is a topological direct sum of the closed subspaces,
\[
\operatorname{Rad}\mathcal{F}_{q}\left(  U\right)  =\prod\limits_{d\in
\mathbb{N}}\mathcal{F}_{q}\left(  U\right)  ^{d}%
\]
up to a topological isomorphism of the Fr\'{e}chet spaces, and
$\operatorname{Spec}\left(  \mathcal{F}_{q}\left(  U\right)  \right)  =U$.
\end{theorem}

That is the Decomposition Theorem from \cite{Dosi25}, which states that the
noncommutative analytic space $\left(  \mathbb{C}_{xy},\mathcal{F}_{q}\right)
$ can be obtained by means of the quantization of the commutative analytic
space $\left(  \mathbb{C}_{xy},\mathcal{O}\right)  $ in the following way
\[
\mathcal{F}_{q}=\prod\limits_{d\in\mathbb{Z}_{+}}\mathcal{O}\left(  d\right)
\]
up to a Fr\'{e}chet space sheaf isomorphism. Moreover, $\mathcal{O}$ can be
restored as the quotient $\mathcal{F}_{q}/\operatorname{Rad}\mathcal{F}%
_{q}=\mathcal{O}$, where $\operatorname{Rad}\mathcal{F}_{q}$ is the sheaf
associated to the presheaf $U\mapsto\operatorname{Rad}\mathcal{F}_{q}\left(
U\right)  $.

\subsection{Runge $q$-open subsets}

Recall that an open subset $V\subseteq\mathbb{C}$ in the standard topology is
said to be a Runge open subset if it has no holes or every component of $V$ is
a simply connected domain. A Runge open subset $V$ can also be defined in
terms of the density of the polynomial algebra $\mathbb{C}\left[  z\right]  $
in the Fr\'{e}chet algebra $\mathcal{O}\left(  V\right)  $.

\begin{definition}
\label{def6}A $q$-open and Runge open subset of the complex plane $\mathbb{C}$
is called \textit{a Runge }$q$\textit{-open set}.
\end{definition}

The basic $q$-open subsets considered in \cite{Dosi24} are all Runge $q$-open
subsets. Namely, take $\lambda\in\mathbb{C}\backslash\left\{  0\right\}  $ and
its $\mathfrak{q}$-neighborhood $U$, which is a $q$-open subset containing
$\lambda$. The $q$-hull $\left\{  \lambda\right\}  _{q}$ of $\lambda$ (see
Subsection \ref{subsecQT}) is contained in $U$. Notice that $B\left(
0,\varepsilon\right)  \subseteq U$ for a small $\varepsilon>0$, and there are
just finitely many of $q^{m}\lambda$ outside of the disk $B\left(
0,\varepsilon\right)  $, say $\left\{  \lambda,q\lambda,\ldots q^{n}%
\lambda\right\}  $. Pick up small disks $B\left(  q^{m}\lambda,\left\vert
q\right\vert ^{m}\delta\right)  \subseteq U$, $0\leq m\leq n$ for $\delta>0$.
Then $U_{\varepsilon,\delta}=B\left(  0,\varepsilon\right)  \cup
\bigcup\limits_{m=0}^{n}B\left(  q^{m}\lambda,\left\vert q\right\vert
^{m}\delta\right)  $ is a Runge $q$-open neighborhood of $\left\{
\lambda\right\}  _{q}$ contained in $U$. A $q$-open subset $U\subseteq
\mathbb{C}_{xy}$ is said to be \textit{a Runge }$q$\textit{-open subset }if
both parts $U_{x}$ and $U_{y}$ are Runge $q$-open subsets of the complex plane
in the sense of Definition \ref{def6}. Thus Runge $q$-open subsets in the
$q$-plane $\mathbb{C}_{xy}$ form a topology basis.

\begin{proposition}
\label{propRunge}Let $U\subseteq\mathbb{C}_{xy}$ be a $q$-open subset. Then
$U$ is a Runge $q$-open subset if and only if $\mathfrak{A}_{q}$ is dense in
$\mathcal{F}_{q}\left(  U\right)  $.
\end{proposition}

\begin{proof}
First assume that $U$ is a Runge $q$-open subset. Take $h=\left(  f,g\right)
\in\mathcal{O}\left(  d\right)  \left(  U\right)  $ to be identified with
\[
\alpha_{d}\left(  h\right)  =\left(  f\left(  z\right)  y^{d}+\sum
_{k>d}\left(  a_{k}z^{d}\right)  y^{k},x^{d}g\left(  w\right)  +\sum
_{i>d}x^{i}\left(  b_{i}w^{d}\right)  \right)  \in\mathcal{F}_{q}\left(
U\right)  ^{d},
\]
where $a_{k}=\dfrac{g^{\left(  k\right)  }\left(  0\right)  }{k!}$ and
$b_{i}=\dfrac{f^{\left(  i\right)  }\left(  0\right)  }{i!}$ for all $k,i>d$.
There are polynomials $\left\{  u_{n}\right\}  \subseteq\mathbb{C}\left[
z\right]  $ and $\left\{  v_{n}\right\}  \subseteq\mathbb{C}\left[  w\right]
$ such that $\lim_{n}u_{n}=f$ in $\mathcal{O}\left(  U_{x}\right)  $ and
$\lim_{n}v_{n}=g$ in $\mathcal{O}\left(  U_{y}\right)  $. Based on the
continuity of the differentiation operator on the Fr\'{e}chet spaces
$\mathcal{O}\left(  U_{x}\right)  $ and $\mathcal{O}\left(  U_{y}\right)  $,
we derive that $a_{k}=\lim_{n}\dfrac{v_{n}^{\left(  k\right)  }\left(
0\right)  }{k!}$ and $b_{i}=\lim_{n}\dfrac{u_{n}^{\left(  i\right)  }\left(
0\right)  }{i!}$. Since $f^{\left(  i\right)  }\left(  0\right)  =0$,
$g^{\left(  k\right)  }\left(  0\right)  =0$, $i,k<d$, we can also assume that
$u_{n}^{\left(  i\right)  }\left(  0\right)  =0$, $v_{n}^{\left(  k\right)
}\left(  0\right)  =0$, $i,k<d$ for all $n$. One needs to replace $u_{n}$,
$v_{n}$ by $P_{d}\left(  u_{n}\right)  $, $P_{d}\left(  v_{n}\right)  $ using
the projection $P_{d}$ from the proof of Lemma \ref{lemMD}. Moreover,
replacing $u_{n}\left(  z\right)  $ by $u_{n}\left(  z\right)  +\lambda
_{n}z^{d}$, $\lambda_{n}=\dfrac{v_{n}^{\left(  d\right)  }\left(  0\right)
}{d!}-\dfrac{u_{n}^{\left(  d\right)  }\left(  0\right)  }{d!}$, we can assume
that $h_{n}=\left(  u_{n},v_{n}\right)  \in\mathcal{O}\left(  d\right)
\left(  U\right)  $ for all $n$. Since $\lim_{n}\lambda_{n}=0$, it follows
that $\lim_{n}h_{n}=h$ in $\mathcal{O}\left(  d\right)  \left(  U\right)  $.
Further, $h_{n}$ is identified with
\[
\alpha_{d}\left(  h_{n}\right)  =\left(  u_{n}\left(  z\right)  y^{d}%
+\sum_{k>d}\left(  a_{k}^{\left(  n\right)  }z^{d}\right)  y^{k},x^{d}%
v_{n}\left(  w\right)  +\sum_{i>d}x^{i}\left(  b_{i}^{\left(  n\right)  }%
w^{d}\right)  \right)  \in\mathcal{F}_{q}\left(  U\right)  ^{d},
\]
where $a_{k}^{\left(  n\right)  }=\dfrac{v_{n}^{\left(  k\right)  }\left(
0\right)  }{k!}$ and $b_{i}^{\left(  n\right)  }=\dfrac{u_{n}^{\left(
i\right)  }\left(  0\right)  }{i!}$ for all $k,i>d$. Notice that $\left\{
\alpha_{d}\left(  h_{n}\right)  \right\}  \subseteq\mathfrak{A}_{q}$ and
$\alpha_{d}\left(  h\right)  =\lim_{n}\alpha_{d}\left(  h_{n}\right)  $. Thus
$\mathfrak{A}_{q}$ is dense in $\mathcal{O}\left(  d\right)  $, that is,
$\mathcal{O}\left(  d\right)  \subseteq\mathfrak{A}_{q}^{-}$. It follows that
$\oplus_{d}\mathcal{O}\left(  d\right)  \subseteq\mathfrak{A}_{q}^{-}$ holds
too. But $\oplus_{d}\mathcal{O}\left(  d\right)  $ is dense in $\mathcal{F}%
_{q}\left(  U\right)  $ by the Decomposition Theorem \ref{thmainDecom}. Hence
$\mathfrak{A}_{q}^{-}=\mathcal{F}_{q}\left(  U\right)  $.

Conversely, assume that $\mathfrak{A}_{q}$ is dense $\mathcal{F}_{q}\left(
U\right)  $. Let us consider the quotient map $\mathcal{O}\left(
U_{x}\right)  \left[  \left[  y\right]  \right]  \rightarrow\mathcal{O}\left(
U_{x}\right)  $ modulo the Jacobson radical $\operatorname{Rad}\mathcal{O}%
\left(  U_{x}\right)  \left[  \left[  y\right]  \right]  $ followed by the
canonical projection $\pi_{x}:\mathcal{F}_{q}\left(  U\right)  \rightarrow
\mathcal{O}\left(  U_{x}\right)  \left[  \left[  y\right]  \right]  $. Thus we
have the surjective continuous homomorphism $\mathcal{F}_{q}\left(  U\right)
\rightarrow\mathcal{O}\left(  U_{x}\right)  $ annihilating $y$. Since the
image of $\mathfrak{A}_{q}$ under that map is reduced to the $\mathbb{C}%
\left[  z\right]  $, we obtain that $\mathbb{C}\left[  z\right]  $ is dense in
$\mathcal{O}\left(  U_{x}\right)  $, that is, $U_{x}$ is a Runge open subset.
In a similar way, by using the projection $\pi_{y}$, we conclude that $U_{y}$
is a Runge open subset. Whence $U$ is a Runge $q$-open subset.
\end{proof}

\subsection{The multiplication operators on $\mathcal{F}_{q}$}

Now we calculate the multiplication operators on $\mathcal{F}_{q}$ based on
the Decomposition Theorem \ref{thmainDecom}.

\begin{lemma}
\label{lemMO1}Let $L_{x,d}$, $L_{y,d}$, $R_{x,d}$ and $R_{y,d}$ be the
multiplication operators by $x$ and $y$ on $\mathcal{F}_{q}$ restricted to the
complemented subspace $\mathcal{O}\left(  d\right)  $. Then
\[
L_{x,d}=z+M_{x,d},\quad L_{y,d}=q^{d}w+M_{y,d},\quad R_{x,d}=q^{d}%
z+N_{y,d}\text{,\quad}R_{y,d}=w+N_{x,d},
\]
where $M_{y,d}$, $N_{x,d}$ are the operators from (\ref{Nd}).
\end{lemma}

\begin{proof}
Take $\zeta=\left(  f,g\right)  \in\mathcal{O}\left(  d\right)  $ to be
identified with
\[
\alpha_{d}\left(  f,g\right)  =\left(  f\left(  z\right)  y^{d}+\sum
_{k>d}\left(  a_{k}z^{d}\right)  y^{k},x^{d}g\left(  w\right)  +\sum
_{i>d}x^{i}\left(  b_{i}w^{d}\right)  \right)  ,
\]
where $a_{k}=\dfrac{g^{\left(  k\right)  }\left(  0\right)  }{k!}$ and
$b_{i}=\dfrac{f^{\left(  i\right)  }\left(  0\right)  }{i!}$ for all $k,i>d$.
Thus $\zeta\in\mathcal{F}_{q}\left(  U\right)  ^{d}$ and we have (see
(\ref{RzRw1}))
\begin{align*}
L_{x}\zeta &  =x\alpha_{d}\left(  \zeta\right)  =\left(  zf\left(  z\right)
y^{d}+\sum_{k>d}\left(  a_{k}z^{d+1}\right)  y^{k},x^{d+1}g\left(  w\right)
+\sum_{i>d+1}x^{i}\left(  b_{i-1}w^{d}\right)  \right) \\
&  =\left(  zf\left(  z\right)  y^{d},\sum_{i>d}x^{i}\left(  \dfrac{f^{\left(
i-1\right)  }\left(  0\right)  }{\left(  i-1\right)  !}w^{d}\right)  \right)
+\left(  \sum_{k>d}\left(  a_{k}z^{d+1}\right)  y^{k},x^{d+1}\left(  g\left(
w\right)  -\dfrac{f^{\left(  d\right)  }\left(  0\right)  }{d!}w^{d}\right)
\right) \\
&  =\alpha_{d}\left(  z\zeta\right)  +\left(  \left(  a_{d+1}z^{d+1}\right)
y^{d+1}+\sum_{k>d+1}\left(  a_{k}z^{d+1}\right)  y^{k},x^{d+1}\left(  g\left(
w\right)  -\dfrac{g^{\left(  d\right)  }\left(  0\right)  }{d!}w^{d}\right)
\right) \\
&  =\alpha_{d}\left(  z\zeta\right)  +\alpha_{d+1}\left(  \dfrac{g^{\left(
d+1\right)  }\left(  0\right)  }{\left(  d+1\right)  !}z^{d+1},P_{d+1}\left(
g\right)  \left(  w\right)  \right)  =\alpha_{d}\left(  z\zeta\right)
+\alpha_{d+1}\left(  M_{x,d}\zeta\right)  ,
\end{align*}
where $M_{x,d}:\mathcal{O}\left(  d\right)  \rightarrow\mathcal{O}\left(
d+1\right)  $ is the operator introduced above in (\ref{Nd}). Thus the
equality $L_{x}\zeta=z\zeta+M_{x,d}\zeta$ holds up to the topological
isomorphisms. In a similar way using (\ref{5p2}), we derive that
\begin{align*}
L_{y}\zeta &  =y\alpha_{d}\left(  \zeta\right)  =\left(  f\left(  qz\right)
y^{d+1}+\sum_{k>d}\left(  a_{k}q^{d}z^{d}\right)  y^{k+1},q^{d}x^{d}wg\left(
w\right)  +\sum_{i>d}q^{i}x^{i}\left(  b_{i}w^{d+1}\right)  \right) \\
&  =\left(  \sum_{k>d}\left(  a_{k-1}q^{d}z^{d}\right)  y^{k},q^{d}%
x^{d}wg\left(  w\right)  \right)  +\left(  \left(  f\left(  qz\right)
-\dfrac{g^{\left(  d\right)  }\left(  0\right)  }{d!}q^{d}z^{d}\right)
y^{d+1},\sum_{i>d}x^{i}\left(  q^{i}b_{i}w^{d+1}\right)  \right) \\
&  =q^{d}\alpha_{d}\left(  w\zeta\right)  +\left(  \left(  f\left(  qz\right)
-\dfrac{f^{\left(  d\right)  }\left(  0\right)  }{d!}q^{d}z^{d}\right)
y^{d+1},\sum_{i>d}x^{i}\left(  q^{i}\dfrac{f^{\left(  i\right)  }\left(
0\right)  }{i!}w^{d+1}\right)  \right) \\
&  =q^{d}\alpha_{d}\left(  w\zeta\right)  +\alpha_{d+1}\left(  P_{d+1}%
\Delta_{q}\left(  f\right)  ,q^{d+1}\dfrac{f^{\left(  d+1\right)  }\left(
0\right)  }{\left(  d+1\right)  !}w^{d+1}\right)  ,
\end{align*}
that is, $L_{y}\zeta=w\zeta+M_{y,d}\zeta$. The rest formulas were proved in
\cite{Dosi25}.
\end{proof}

\section{The localization property of the sheaf $\mathcal{F}_{q}%
$\label{sectionFGq}}

In this section we solve the localization problem for the algebras of the
structure sheaf $\mathcal{F}_{q}$ over Runge $q$-open subsets.

\subsection{The free $\mathcal{F}_{q}\left(  U\right)  $-bimodule
$\mathcal{F}_{q}\left(  U\right)  ^{\protect\widehat{\otimes}2}$ and the index
order\label{subsecIO}}

Let us fix two copies $U_{1}=U_{z_{1}}\cup U_{w_{1}}$ and $U_{2}=U_{z_{2}}\cup
U_{w_{2}}$ of the same $q$-open subset $U\subseteq\mathbb{C}_{xy}$. The
Fr\'{e}chet space $\mathcal{O}\left(  d\right)  \left(  U_{1}\right)
\widehat{\otimes}\mathcal{O}\left(  l\right)  \left(  U_{2}\right)  $ (or the
related presheaf) is an isomorphic copy of $\mathcal{O}\left(  U_{1}\right)
\widehat{\otimes}\mathcal{O}\left(  U_{2}\right)  $ through the topological
isomorphism
\[
\left(  z^{-d}\times_{1}w^{-d}\right)  \otimes\left(  z^{-l}\times_{1}%
w^{-l}\right)  :\mathcal{O}\left(  d\right)  \left(  U_{1}\right)
\widehat{\otimes}\mathcal{O}\left(  l\right)  \left(  U_{2}\right)
\rightarrow\mathcal{O}\left(  U_{1}\right)  \widehat{\otimes}\mathcal{O}%
\left(  U_{2}\right)
\]
called \textit{the }$\left(  d,l\right)  $\textit{-isomorphism, }which is
combined from $d$ and $l$ isomorphisms (see (\ref{imoo})) of the Fr\'{e}chet
space sheaves over $\mathbb{C}_{xy}$. For brevity, we skip $U$ in the
notations below, and put $\mathcal{O}\left(  d,l\right)  =\mathcal{O}\left(
d\right)  \left(  U_{1}\right)  \widehat{\otimes}\mathcal{O}\left(  l\right)
\left(  U_{2}\right)  $. We also use the notation $\mathcal{O}_{d,l}$ instead
of $\mathcal{O}\left(  U_{1}\right)  \widehat{\otimes}\mathcal{O}\left(
U_{2}\right)  $ to specify the $\left(  d,l\right)  $-isomorphic copy of
$\mathcal{O}\left(  d,l\right)  $. In particular, $\mathcal{O}_{0,0}%
=\mathcal{O}\left(  0,0\right)  =\mathcal{O}\left(  U_{1}\right)
\widehat{\otimes}\mathcal{O}\left(  U_{2}\right)  $.

Recall that the sheaf $\mathcal{F}_{q}=\mathcal{O}\left[  \left[  y\right]
\right]  \underset{\mathbb{C}_{q}\left[  \left[  x,y\right]  \right]
}{\times}\left[  \left[  x\right]  \right]  \mathcal{O}$ was introduced as the
structure sheaf of the noncommutative analytic space of the PI-envelope
$\mathcal{F}_{q}\left(  \mathbb{C}_{xy}\right)  $ of the contractive quantum
plane $\mathfrak{A}_{q}$ (see Definition \ref{def3}). The Decomposition
Theorem \ref{thmainDecom} states that $\mathcal{F}_{q}\left(  U\right)
=\prod\limits_{d\in\mathbb{Z}_{+}}\mathcal{O}\left(  d\right)  \left(
U\right)  $ as a Fr\'{e}chet space. Since the projective tensor product is
compatible with the direct products \cite[II.5.19]{HelHom}, we deduce that
\[
\mathcal{F}_{q}\left(  U\right)  ^{\widehat{\otimes}2}=\prod\limits_{d,l\in
\mathbb{Z}_{+}}\mathcal{O}\left(  d,l\right)  =\prod\limits_{\left(
d,l\right)  \in\mathbb{Z}_{+}^{2}}\mathcal{O}_{d,l}%
\]
up to the topological $\left(  d,l\right)  $-isomorphisms (see Theorem
\ref{thmainDecom}). We can convert all operators occurred in the complex
(\ref{Rsp}) (for the algebra $\mathcal{A=F}_{q}\left(  U\right)  $) into their
$\left(  d,l\right)  $-isomorphic (and $d$-isomorphic) copies. Some long
calculations are placed into the Appendix Section \ref{sectionDCG}.

The index set $\mathbb{Z}_{+}^{2}$ of the decomposition can be (linearly)
ordered in the following natural way. Put
\begin{equation}
\left(  i,j\right)  <\left(  s,t\right)  \text{ if }i+j<s+t\text{, or
}i<s\text{ if }i+j=s+t\text{.}\label{oI}%
\end{equation}
Thus $\mathbb{Z}_{+}^{2}\left(  0\right)  =\left\{  \left(  0,0\right)
\right\}  $ is the least index and the homogeneous components
\[
\mathbb{Z}_{+}^{2}\left(  n\right)  =\left\{  \left(  i,j\right)
:i+j=n\right\}  ,\quad n\geq0
\]
are strongly increasing, whereas in every $\mathbb{Z}_{+}^{2}\left(  n\right)
$ we have the following left-increasing order
\[
\left(  0,n\right)  <\left(  1,n-1\right)  <\cdots<\left(  n,0\right)  .
\]
Notice that every $\mathbb{Z}_{+}^{2}\left(  n\right)  $ can also be ordered
in the right-increasing way. Both cases of the linear order can be used below.

Fix a couple $\left(  d,l\right)  \in\mathbb{Z}_{+}^{2}$. Based on Lemma
\ref{lemMO1} and the notations from Subsections \ref{subsecOF} and
\ref{subsecQRP}, we obtain that
\begin{align*}
\left(  R_{x}\otimes1\right)  |\mathcal{O}_{d,l}  &  =R_{x,d}\otimes
1_{l}=z\otimes q^{d}+N_{y,d}\otimes1=q^{d}z_{1}+N_{y_{1},d},\\
\left(  R_{y}\otimes1\right)  |\mathcal{O}_{d,l}  &  =R_{y,d}\otimes
1_{l}=w\otimes1+N_{x,d}\otimes1=w_{1}+N_{x_{1},d},\\
\left(  1\otimes L_{x}\right)  |\mathcal{O}_{d,l}  &  =1_{d}\otimes
L_{x,l}=1\otimes z+1\otimes M_{x,l}=z_{2}+M_{x_{2},l},\\
\left(  1\otimes L_{y}\right)  |\mathcal{O}_{d,l}  &  =1_{d}\otimes
L_{y,l}=q^{l}\otimes w+1\otimes M_{y,l}=q^{l}w_{2}+M_{y_{2},l},
\end{align*}
where $1_{l}$ and $1_{d}$ are the identity maps. As above consider the complex
(\ref{Rsp}) for the algebra $\mathcal{A=F}_{q}\left(  U\right)  $. It tuns out
to be the following complex
\begin{equation}
0\rightarrow\prod\limits_{\left(  d,l\right)  \in\mathbb{Z}_{+}^{2}%
}\mathcal{O}_{d,l}\overset{d^{0}}{\longrightarrow}\prod\limits_{\left(
d,l\right)  \in\mathbb{Z}_{+}^{2}}\left(
\begin{array}
[c]{c}%
\mathcal{O}_{d,l}\\
\oplus\\
\mathcal{O}_{d,l}%
\end{array}
\right)  \overset{d^{1}}{\longrightarrow}\prod\limits_{\left(  d,l\right)
\in\mathbb{Z}_{+}^{2}}\mathcal{O}_{d,l}\overset{\pi}{\longrightarrow}%
\prod\limits_{n\in\mathbb{Z}_{+}}\mathcal{O}_{n}\rightarrow0 \label{Oc}%
\end{equation}
up to the family of $\left(  d,l\right)  $-isomorphisms and $d$-isomorphisms
(\ref{imoo}). The differentials are acting (see (\ref{d})) in the following
ways%
\begin{align*}
d^{0}|\mathcal{O}_{d,l}  &  =\left[
\begin{array}
[c]{c}%
R_{y,d}\otimes1_{l}-q1_{d}\otimes L_{y,l}\\
1_{d}\otimes L_{x,l}-R_{x,d}\otimes q1_{l}%
\end{array}
\right]  =\left[
\begin{array}
[c]{c}%
w_{1}+N_{x_{1},d}-q^{l+1}w_{2}-qM_{y_{2},l}\\
z_{2}+M_{x_{2},l}-q^{d+1}z_{1}-qN_{y_{1},d}%
\end{array}
\right] \\
&  =\partial_{d,l}^{0}+M_{d,l}^{0}+N_{d,l}^{0},
\end{align*}
where $\partial_{d,l}^{0}=\left[
\begin{array}
[c]{c}%
w_{1}-q^{l+1}w_{2}\\
z_{2}-q^{d+1}z_{1}%
\end{array}
\right]  :\mathcal{O}_{d,l}\rightarrow\mathcal{O}_{d,l}^{\oplus2}$ whereas
\[
M_{d,l}^{0}=\left[
\begin{array}
[c]{c}%
-qM_{y_{2},l}\\
M_{x_{2},l}%
\end{array}
\right]  :\mathcal{O}_{d,l}\rightarrow\mathcal{O}_{d,l+1}^{\oplus2},\quad
N_{d,l}^{0}=\left[
\begin{array}
[c]{c}%
N_{x_{1},d}\\
-qN_{y_{1},d}%
\end{array}
\right]  :\mathcal{O}_{d,l}\rightarrow\mathcal{O}_{d+1,l}^{\oplus2}%
\]
(see Subsection \ref{subsecDFC1}). In the same way, we obtain that
\begin{align*}
d^{1}|\left(  \mathcal{O}_{d,l}\right)  ^{\oplus2}  &  =\left[
\begin{array}
[c]{cc}%
1_{d}\otimes L_{x,l}-R_{x,d}\otimes1_{l} & 1_{d}\otimes L_{y,l}-R_{y,d}%
\otimes1_{l}%
\end{array}
\right] \\
&  =\left[
\begin{array}
[c]{cc}%
z_{2}+M_{x_{2},l}-q^{d}z_{1}-N_{y_{1},d} & q^{l}w_{2}+M_{y_{2},l}%
-w_{1}-N_{x_{1},d}%
\end{array}
\right] \\
&  =\partial_{d,l}^{1}+M_{d,l}-N_{d,l},
\end{align*}
where $\partial_{d,l}^{1}=\left[
\begin{array}
[c]{cc}%
z_{2}-q^{d}z_{1} & q^{l}w_{2}-w_{1}%
\end{array}
\right]  :\mathcal{O}_{d,l}^{\oplus2}\rightarrow\mathcal{O}_{d,l}$ whereas
\[
M_{d,l}=\left[
\begin{array}
[c]{cc}%
M_{x_{2},l} & M_{y_{2},l}%
\end{array}
\right]  :\mathcal{O}_{d,l}^{\oplus2}\rightarrow\mathcal{O}_{d,l+1}%
\text{,\quad}N_{d,l}=\left[
\begin{array}
[c]{cc}%
N_{y_{1},d} & N_{x_{1},d}%
\end{array}
\right]  :\mathcal{O}_{d,l}^{\oplus2}\rightarrow\mathcal{O}_{d+1,l}%
\]
are the partial differential and difference operators considered in Subsection
\ref{subsecPDE}. Thus
\begin{equation}
d^{0}|\mathcal{O}_{d,l}=\partial_{d,l}^{0}+M_{d,l}^{0}+N_{d,l}^{0}%
\quad\text{and\quad}d^{1}|\left(  \mathcal{O}_{d,l}\right)  ^{\oplus
2}=\partial_{d,l}^{1}+M_{d,l}-N_{d,l} \label{dd01}%
\end{equation}
hold, which means that%
\begin{equation}
d^{0}=\left[
\begin{array}
[c]{ccc}%
\ddots &  & 0\\
& \partial_{d,l}^{0} & \\
\ast &  & \ddots
\end{array}
\right]  ,\quad d^{1}=\left[
\begin{array}
[c]{ccc}%
\ddots &  & 0\\
& \partial_{d,l}^{1} & \\
\ast &  & \ddots
\end{array}
\right]  \label{d01}%
\end{equation}
have lower triangular operator matrix shapes with respect to the linear order
introduced above.

\begin{lemma}
\label{dmn}the following operator identities
\begin{align*}
\partial_{i,j}^{1}M_{i,j-1}^{0}  &  =-M_{i,j-1}\partial_{i,j-1}^{0}%
,\quad\partial_{i,j}^{1}N_{i-1,j}^{0}=N_{i-1,j}\partial_{i-1,j}^{0},\quad
M_{i,j-1}M_{i,j-2}^{0}=0,\\
M_{i,j-1}N_{i-1,j-1}^{0}  &  =N_{i-1,j}M_{i-1,j-1}^{0},\quad N_{i-1,j}%
N_{i-2,j}^{0}=0
\end{align*}
for all $i,j$.
\end{lemma}

\begin{proof}
Take $\zeta\in\prod\limits_{\left(  d,l\right)  \in\mathbb{Z}_{+}^{2}%
}\mathcal{O}_{d,l}$. The identity $d^{1}d^{0}\zeta=0$ implies that
\begin{align*}
0  &  =\partial_{i,j}^{1}\left(  d^{0}\zeta\right)  _{i,j}+M_{i,j-1}\left(
d^{0}\zeta\right)  _{i,j-1}-N_{i-1,j}\left(  d^{0}\zeta\right)  _{i-1,j}%
=\partial_{i,j}^{1}M_{i,j-1}^{0}\zeta_{i,j-1}+\partial_{i,j}^{1}N_{i-1,j}%
^{0}\zeta_{i-1,j}\\
&  +M_{i,j-1}\partial_{i,j-1}^{0}\zeta_{i,j-1}+M_{i,j-1}M_{i,j-2}^{0}%
\zeta_{i,j-2}+M_{i,j-1}N_{i-1,j-1}^{0}\zeta_{i-1,j-1}\\
&  -N_{i-1,j}\partial_{i-1,j}^{0}\zeta_{i-1,j}-N_{i-1,j}M_{i-1,j-1}^{0}%
\zeta_{i-1,j-1}-N_{i-1,j}N_{i-2,j}^{0}\zeta_{i-2,j}\\
&  =\left(  \partial_{i,j}^{1}M_{i,j-1}^{0}+M_{i,j-1}\partial_{i,j-1}%
^{0}\right)  \zeta_{i,j-1}+\left(  \partial_{i,j}^{1}N_{i-1,j}^{0}%
-N_{i-1,j}\partial_{i-1,j}^{0}\right)  \zeta_{i-1,j}+M_{i,j-1}M_{i,j-2}%
^{0}\zeta_{i,j-2}\\
&  +\left(  M_{i,j-1}N_{i-1,j-1}^{0}-N_{i-1,j}M_{i-1,j-1}^{0}\right)
\zeta_{i-1,j-1}-N_{i-1,j}N_{i-2,j}^{0}\zeta_{i-2,j}.
\end{align*}
Thus $d^{1}d^{0}=0$ is equivalent to the operator identities indicated above.
\end{proof}

\subsection{The multiplication morphism}

Let $U\subseteq\mathbb{C}_{xy}$ be a $q$-open subset. Let us analyze the
multiplication morphism $\pi:\mathcal{F}_{q}\left(  U\right)
^{\widehat{\otimes}2}\rightarrow\mathcal{F}_{q}\left(  U\right)  $ by reducing
it to the subspace $\mathcal{O}_{d,l}$. Let us consider the following morphism
$\pi_{d,l}:\mathcal{O}\left(  d,l\right)  \rightarrow\mathcal{O}\left(
d+l\right)  $, $\pi_{d,l}\left(  \zeta\otimes\eta\right)  =\left(  f\left(
z\right)  u\left(  q^{d}z\right)  ,g\left(  q^{l}w\right)  v\left(  w\right)
\right)  $, $\zeta=\left(  f,g\right)  \in\mathcal{O}\left(  d\right)  $ and
$\eta=\left(  u,v\right)  \in\mathcal{O}\left(  l\right)  $ from Subsection
\ref{subsecQRP}.

\begin{lemma}
\label{lemODL}The multiplication morphism $\pi:\prod\limits_{\left(
d,l\right)  \in\mathbb{Z}_{+}^{2}}\mathcal{O}\left(  d,l\right)
\longrightarrow\prod\limits_{n\in\mathbb{Z}_{+}}\mathcal{O}\left(  n\right)  $
restricted to $\mathcal{O}_{d,l}$ has the following decomposition (up to the
$\left(  d,l\right)  $-isomorphism)
\[
\pi|\mathcal{O}_{d,l}:\mathcal{O}_{d,l}\rightarrow\prod\limits_{m\geq
d+l}\mathcal{O}_{m},\quad\pi|\mathcal{O}_{d,l}=\pi_{dl}+\sum_{m>d+l}\Gamma
_{m}^{d,l},
\]
for some uniquely defined continuous linear (differential) operators
$\Gamma_{m}^{d,l}:\mathcal{O}_{d,l}\rightarrow\mathcal{O}_{m}$, $m>d+l$. In
particular, for every $\psi\in\prod\limits_{\left(  i,j\right)  \in
\mathbb{Z}_{+}^{2}}\mathcal{O}_{i,j}^{\oplus2}$ the identity
\[
\sum_{d+l=n}\pi_{d,l}\left(  M_{d,l-1}\psi_{d,l-1}-N_{d-1,l}\psi
_{d-1,l}\right)  +\sum_{\left(  i,j\right)  ,i+j<n}\Gamma_{n}^{i,j}\left(
\left(  d^{1}\psi\right)  _{i,j}\right)  =0
\]
holds in $\mathcal{O}_{n}$ for every $n$.
\end{lemma}

\begin{proof}
Since $\mathcal{O}\left(  d,l\right)  =\mathcal{O}\left(  d\right)
\widehat{\otimes}\mathcal{O}\left(  l\right)  $, it suffices to prove the
statement for an elementary tensor $h=\zeta\otimes\eta$ with $\zeta=\left(
f,g\right)  \in\mathcal{O}\left(  d\right)  $ and $\eta=\left(  u,v\right)
\in\mathcal{O}\left(  l\right)  $. By Theorem \ref{thmainDecom}, $\zeta$ is
identified with $\alpha_{d}\left(  f,g\right)  $ (as an elements of the
algebra $\mathcal{F}_{q}$) whereas $\eta$ with $\alpha_{l}\left(  u,v\right)
$. Thus
\begin{align*}
f  &  =f\left(  z\right)  y^{d}+\sum_{k>d}\left(  a_{k}z^{d}\right)
y^{k}\text{, }a_{k}=\dfrac{g^{\left(  k\right)  }\left(  0\right)  }%
{k!}\text{,\quad}g=x^{d}g\left(  w\right)  +\sum_{i>d}x^{i}\left(  b_{i}%
w^{d}\right)  \text{, }b_{i}=\dfrac{f^{\left(  i\right)  }\left(  0\right)
}{i!}\text{, }k,i>d,\\
u  &  =u\left(  z\right)  y^{l}+\sum_{k>l}\left(  \alpha_{k}z^{l}\right)
y^{k}\text{, }\alpha_{k}=\dfrac{v^{\left(  k\right)  }\left(  0\right)  }%
{k!}\text{,\quad}v=x^{l}v\left(  w\right)  +\sum_{i>l}x^{i}\left(  \beta
_{i}w^{l}\right)  \text{, }\beta_{i}=\dfrac{u^{\left(  i\right)  }\left(
0\right)  }{i!},\text{ }k,i>l
\end{align*}
to be the elements of the fibered product $\mathcal{F}_{q}=\mathcal{O}\left[
\left[  y\right]  \right]  \underset{\mathbb{C}\left[  \left[  x,y\right]
\right]  }{\times}\left[  \left[  x\right]  \right]  \mathcal{O}$. It follows
that $\zeta\eta=\left(  fu,gv\right)  \in\mathcal{F}_{q}$ such that (see
(\ref{5p2}))
\begin{align*}
fu  &  =f\left(  z\right)  u\left(  q^{d}z\right)  y^{d+l}+\sum_{k>d+l}\left(
\sum_{i+j=k}a_{i}\alpha_{j}q^{il}\right)  z^{d+l}y^{k}+\varphi,\\
gv  &  =x^{d+l}g\left(  q^{l}w\right)  v\left(  w\right)  +\sum_{i>d+l}%
x^{i}\left(  \sum_{s+t=i}b_{s}\beta_{t}q^{td}\right)  w^{d+l}+\psi,
\end{align*}
where
\begin{align*}
\varphi &  =\sum_{k>l}q^{ld}\alpha_{k}f\left(  z\right)  z^{l}y^{d+k}%
+\sum_{k>d}a_{k}u\left(  q^{k}z\right)  z^{d}y^{k+l}=\sum_{k=1}^{\infty
}\left(  q^{ld}\alpha_{l+k}f\left(  z\right)  z^{l}+a_{d+k}u\left(
q^{d+k}z\right)  z^{d}\right)  y^{d+l+k},\\
\psi &  =\sum_{i>l}x^{d+i}\beta_{i}g\left(  q^{i}w\right)  w^{l}+\sum
_{i>d}q^{ld}x^{l+i}b_{i}v\left(  w\right)  w^{d}=\sum_{i=1}^{\infty}%
x^{d+l+i}\left(  \beta_{l+i}g\left(  q^{l+i}w\right)  w^{l}+q^{ld}%
b_{d+i}v\left(  w\right)  w^{d}\right)  .
\end{align*}
Put $A_{k}=\sum_{i+j=k}a_{i}\alpha_{j}q^{il}$ and $B_{i}=\sum_{s+t=i}%
b_{s}\beta_{t}q^{td}$. Notice that%
\begin{align*}
A_{k}  &  =\sum_{i+j=k}\dfrac{g^{\left(  i\right)  }\left(  0\right)  }%
{i!}\dfrac{v^{\left(  j\right)  }\left(  0\right)  }{j!}q^{il}=\dfrac{1}%
{k!}\sum_{i}\dbinom{k}{i}g^{\left(  i\right)  }\left(  q^{l}w\right)
|_{w=0}v^{\left(  k-i\right)  }\left(  0\right)  =\dfrac{1}{k!}\frac
{d^{k}g\left(  q^{l}w\right)  v\left(  w\right)  }{dw^{k}}\left(  0\right)
,\\
B_{i}  &  =\dfrac{1}{i!}\frac{d^{i}f\left(  z\right)  u\left(  q^{d}z\right)
}{dz^{i}}\left(  0\right)  .
\end{align*}
The element $\pi_{d,l}\left(  h\right)  =\left(  f\left(  z\right)  u\left(
q^{d}z\right)  ,g\left(  q^{l}w\right)  v\left(  w\right)  \right)
\in\mathcal{O}\left(  d+l\right)  $ is identified with $\alpha_{d+l}\left(
\pi_{d,l}\left(  h\right)  \right)  $. Then
\[
\pi_{d,l}\left(  h\right)  =\left(  f\left(  z\right)  u\left(  q^{d}z\right)
y^{d+l}+\sum_{k>d+l}A_{k}z^{d+l}y^{k},x^{d+l}g\left(  q^{l}w\right)  v\left(
w\right)  +\sum_{i>d+l}x^{i}B_{i}w^{d+l}\right)  .
\]
Put $\xi=\zeta\eta-\pi_{d,l}\left(  h\right)  $ to be the element $\xi=\left(
\varphi,\psi\right)  \in\mathcal{F}_{q}$. Since
\begin{align*}
\varphi &  =\sum_{k=1}^{\infty}\varphi_{d+l+k}y^{d+l+k}\in\mathcal{O}\left(
U_{x}\right)  \left[  \left[  y\right]  \right]  \text{,\quad}\varphi
_{d+l+k}=q^{ld}\alpha_{l+k}f\left(  z\right)  z^{l}+a_{d+k}u\left(
q^{d+k}z\right)  z^{d},\\
\psi &  =\sum_{i=1}^{\infty}x^{d+l+i}\psi_{l+d+i}\in\left[  \left[  x\right]
\right]  \mathcal{O}\left(  U_{y}\right)  \text{,\quad}\psi_{l+d+i}%
=\beta_{l+i}g\left(  q^{l+i}w\right)  w^{l}+q^{ld}b_{d+i}v\left(  w\right)
w^{d},
\end{align*}
if follows that $\varphi_{n}=\psi_{n}=0$ for all $n\leq d+l$. Note that
$f^{\left(  i\right)  }\left(  0\right)  =0$, $i<d$, $u^{\left(  j\right)
}\left(  0\right)  =0$, $j<l$, and $v^{\left(  k\right)  }\left(  0\right)
=0$, $k<l$, $g^{\left(  j\right)  }\left(  0\right)  =0$, $j<d$. Since
\begin{align*}
\varphi_{d+l+k}  &  =q^{ld}\alpha_{l+k}f\left(  z\right)  z^{l}+a_{d+k}%
u\left(  q^{d+k}z\right)  z^{d}=q^{ld}\dfrac{v^{\left(  l+k\right)  }\left(
0\right)  }{\left(  l+k\right)  !}f\left(  z\right)  z^{l}+\dfrac{g^{\left(
d+k\right)  }\left(  0\right)  }{\left(  d+k\right)  !}u\left(  q^{d+k}%
z\right)  z^{d},\\
\psi_{l+d+i}  &  =\beta_{l+i}g\left(  q^{l+i}w\right)  w^{l}+q^{ld}%
b_{d+i}v\left(  w\right)  w^{d}=q^{ld}\dfrac{f^{\left(  d+i\right)  }\left(
0\right)  }{\left(  d+i\right)  !}v\left(  w\right)  w^{d}+\dfrac{u^{\left(
l+i\right)  }\left(  0\right)  }{\left(  l+i\right)  !}g\left(  q^{l+i}%
w\right)  w^{l},
\end{align*}
it follows that $\varphi_{d+l+k}^{\left(  i\right)  }\left(  0\right)  =0$,
$i<d+l$, $\psi_{l+d+i}^{\left(  k\right)  }\left(  0\right)  =0$, $k<l+d$,
and
\[
\dfrac{\varphi_{d+l+k}^{\left(  d+l+i\right)  }\left(  0\right)  }{\left(
d+l+i\right)  !}=q^{ld}\dfrac{v^{\left(  l+k\right)  }\left(  0\right)
}{\left(  l+k\right)  !}\dfrac{f^{\left(  d+i\right)  }\left(  0\right)
}{\left(  d+i\right)  !}+q^{\left(  d+k\right)  \left(  l+i\right)  }%
\dfrac{g^{\left(  d+k\right)  }\left(  0\right)  }{\left(  d+k\right)
!}\dfrac{u^{\left(  l+i\right)  }\left(  0\right)  }{\left(  l+i\right)
!}=\dfrac{\psi_{l+d+i}^{\left(  d+l+k\right)  }\left(  0\right)  }{\left(
d+l+k\right)  !}%
\]
for all $i,k\geq0$ (see (\ref{FqF})). Hence $\xi\in\mathcal{F}_{q}$ and
\begin{align*}
p_{n}\left(  U\right)  \xi &  =\overline{\xi}=\left(  \overline{\varphi
},\overline{\psi}\right)  ,\\
\overline{\varphi}  &  =\left(  \varphi_{n}\left(  z\right)  -\sum_{i=0}%
^{n-1}\dfrac{\psi_{i}^{\left(  n\right)  }\left(  0\right)  }{n!}z^{i}\right)
y^{n}+\sum_{k>n}\left(  \dfrac{\psi_{n}^{\left(  k\right)  }\left(  0\right)
}{k!}z^{n}\right)  y^{k}=0,\\
\overline{\psi}  &  =x^{n}\left(  \psi_{n}\left(  w\right)  -\sum_{i=0}%
^{n-1}\dfrac{\varphi_{i}^{\left(  n\right)  }\left(  0\right)  }{n!}%
w^{i}\right)  +\sum_{i>n}x^{i}\left(  \dfrac{\varphi_{n}^{\left(  i\right)
}\left(  0\right)  }{i!}w^{d}\right)  =0
\end{align*}
for all $n\leq d+l$, where $p_{n}\left(  U\right)  $ is the continuous
projection onto $\mathcal{O}\left(  n\right)  $. But $\sum_{n=0}^{\infty}%
p_{n}\left(  U\right)  \xi=\xi$ thanks to the Decomposition Theorem
\ref{thmainDecom}, therefore $\xi\in\prod\limits_{t\geq1}\mathcal{O}\left(
d+l+t\right)  $. In particular, $p_{m}\left(  U\right)  \xi=\Gamma_{m}^{d,l}h$
define the continuous linear operators $\Gamma_{m}^{d,l}:\mathcal{O}\left(
d,l\right)  \rightarrow\mathcal{O}\left(  m\right)  $, $m>d+l$. Thus
\[
\pi\left(  h\right)  =\pi\left(  \zeta\otimes\eta\right)  =\pi_{d,l}\left(
h\right)  +\xi=\pi_{d,l}\left(  h\right)  +\sum_{m>d+l}\Gamma_{m}^{d,l}%
h\in\mathcal{O}\left(  d+l\right)  \oplus\prod\limits_{m>d+l}\mathcal{O}%
\left(  m\right)  .
\]
By Lemma \ref{lemPidl}, the mapping $\pi_{d,l}$ is reduced to $\pi
_{d,l}:\mathcal{O}_{d,l}\longrightarrow\mathcal{O}_{d+l}$, $\pi_{d,l}\left(
h\left(  z_{1},w_{1};z_{2},w_{2}\right)  \right)  =\left(  h_{1}\left(
z_{1},q^{d}z_{1}\right)  ,h_{4}\left(  q^{l}w_{2},w_{2}\right)  \right)  $ up
to the topological isomorphisms (\ref{imoo}). The operators $\Gamma_{m}^{d,l}$
can also be replaced by their isomorphic actions $\Gamma_{m}^{d,l}%
:\mathcal{O}_{d,l}\rightarrow\mathcal{O}_{m}$, $m>d+l$, so that $\pi
|\mathcal{O}_{d,l}=\pi_{dl}+\sum_{m>d+l}\Gamma_{m}^{d,l}$.

Finally, by the construction, $\pi d^{1}=0$ and $d^{1}$ has the lower
triangular matrix (\ref{d01}). If $\psi=\left(  \psi_{i,j}\right)  _{i,j}%
\in\prod\limits_{\left(  i,j\right)  \in\mathbb{Z}_{+}^{2}}\mathcal{O}%
_{i,j}^{\oplus2}$, then $d^{1}\psi\in\prod\limits_{\left(  i,j\right)
\in\mathbb{Z}_{+}^{2}}\mathcal{O}_{i,j}$ with $\left(  d^{1}\psi\right)
_{i,j}=\partial_{i,j}^{1}\psi_{i,j}+M_{i,j-1}\psi_{i,j-1}-N_{i-1,j}%
\psi_{i-1,j}$. Taking into account that $\pi_{d,l}\partial_{d,l}^{1}=0$ for
all $d,l$ (see Subsection \ref{subsecDFC1}), we deduce from $\left(  \pi
d^{1}\right)  \left(  \psi\right)  =0$ that
\[
\sum_{d+l=n}\pi_{d,l}\left(  M_{d,l-1}\psi_{d,l-1}-N_{d-1,l}\psi
_{d-1,l}\right)  +\sum_{\left(  i,j\right)  ,i+j<n}\Gamma_{n}^{i,j}\left(
\left(  d^{1}\psi\right)  _{i,j}\right)  =0\text{ in }\mathcal{O}_{n}%
\]
for every $n$. In the latter expression, the vectors $\psi_{s,t}$ of the
homogenous index $s+t$ less than $n$ have only been involved.
\end{proof}

Thus the multiplication morphism $\pi$ does not admit a triangular shape
caused mainly by the complicated structure of the sheaf $\mathcal{F}_{q}$. As
we have seen above in Subsection \ref{subsecCQxy}, the triangular shapes take
place for the sheaves $\mathcal{O}\left[  \left[  y\right]  \right]  $,
$\left[  \left[  x\right]  \right]  \mathcal{O}$ and $\mathbb{C}_{q}\left[
\left[  x,y\right]  \right]  $.

Now we can prove the following key result on the image of $d^{1}$ using the
diagonal cohomology construction from Appendix Section \ref{sectionDCG}.

\begin{proposition}
\label{propD1P}The image of the differential $d^{1}:\prod\limits_{\left(
d,l\right)  \in\mathbb{Z}_{+}^{2}}\mathcal{O}_{d,l}^{\oplus2}\longrightarrow
\prod\limits_{\left(  d,l\right)  \in\mathbb{Z}_{+}^{2}}\mathcal{O}_{d,l}$ is closed.
\end{proposition}

\begin{proof}
Take a sequence $\left\{  \psi^{\left(  k\right)  }\right\}  \subseteq
\prod\limits_{\left(  d,l\right)  \in\mathbb{Z}_{+}^{2}}\mathcal{O}%
_{d,l}^{\oplus2}$ such that $d^{1}\psi^{\left(  k\right)  }$ converges to some
$\varphi\in\prod\limits_{\left(  d,l\right)  \in\mathbb{Z}_{+}^{2}}%
\mathcal{O}_{d,l}$. Using (\ref{d01}), we derive that
\begin{equation}
\varphi_{d,l}=\lim_{k}\left(  d^{1}\psi^{\left(  k\right)  }\right)
_{d,l}=\lim_{k}\partial_{d,l}^{1}\psi_{d,l}^{\left(  k\right)  }+M_{d,l-1}%
\psi_{d,l-1}^{\left(  k\right)  }-N_{d-1,l}\psi_{d-1,l}^{\left(  k\right)  }
\label{fi1}%
\end{equation}
for all $d$, $l$. One needs to prove that $\varphi_{d,l}=\partial_{d,l}%
^{1}\theta_{d,l}+M_{d,l-1}\theta_{d,l-1}-N_{d-1,l}\theta_{d-1,l}$ for some
$\theta=\left(  \theta_{d,l}\right)  _{d,l}\in\prod\limits_{\left(
d,l\right)  \in\mathbb{Z}_{+}^{2}}\mathcal{O}_{d,l}^{\oplus2}$, that is,
$d^{1}\theta=\varphi$. Let us construct these elements by induction on the
homogeneous order $n$.

If $n=0$, then $\varphi_{0,0}=\lim_{k}\partial_{0,0}^{1}\psi_{0,0}^{\left(
k\right)  }$ in $\mathcal{O}_{0,0}$. By Theorem \ref{thMain1}, $H_{0,0}%
^{2}=\left\{  0\right\}  $ or $\operatorname{im}\left(  \partial_{0,0}%
^{1}\right)  =\ker\left(  \pi_{0,0}\right)  $ is closed. Therefore
$\varphi_{0,0}=\partial_{0,0}^{1}\theta_{0,0}$ for some $\theta_{0,0}%
\in\mathcal{O}_{0,0}^{\oplus2}$, and
\[
\lim_{k}\partial_{0,0}^{1}\omega_{0,0}^{\left(  k\right)  }=0\quad
\text{with\quad}\omega_{0,0}^{\left(  k\right)  }=\psi_{0,0}^{\left(
k\right)  }-\theta_{0,0}.
\]
Based on Corollary \ref{corDecom}, we have
\[%
\begin{array}
[c]{cc}%
\mathcal{O}_{0,0}^{\oplus2} & \\
\parallel & \\
\operatorname{im}\left(  p_{0,0}^{1}\right)  & \overset{\partial_{0,0}%
^{1}}{\longrightarrow}\\
\oplus & \\
\operatorname{im}\left(  p_{0,0}\right)  \oplus\operatorname{im}\left(
\partial_{0,0}^{0}\right)  & =\ker\left(  \partial_{0,0}^{1}\right)
\end{array}%
\begin{array}
[c]{c}%
\\
\\
\mathcal{O}_{0,0}\\
\\
\end{array}
\]
It follows that
\begin{equation}
\omega_{0,0}^{\left(  k\right)  }-s_{0,0}^{\left(  k\right)  }-\partial
_{0,0}^{0}t_{0,0}^{\left(  k\right)  }\rightarrow0 \label{fi12}%
\end{equation}
for some $\left\{  s_{0,0}^{\left(  k\right)  }\right\}  \subseteq
\operatorname{im}\left(  p_{0,0}\right)  $ and $\left\{  t_{0,0}^{\left(
k\right)  }\right\}  \subseteq\mathcal{O}_{0,0}$.

In the case of $n=1$, we have
\begin{equation}
\varphi_{0,1}=\lim_{k}\partial_{0,1}^{1}\psi_{0,1}^{\left(  k\right)
}+M_{0,0}\psi_{0,0}^{\left(  k\right)  }\quad\text{and\quad}\varphi_{1,0}%
=\lim_{k}\partial_{1,0}^{1}\psi_{1,0}^{\left(  k\right)  }-N_{0,0}\psi
_{0,0}^{\left(  k\right)  }. \label{fi11}%
\end{equation}
Then $\varphi_{0,1}-M_{00}\theta_{0,0}=\lim_{k}\partial_{0,1}^{1}\psi
_{0,1}^{\left(  k\right)  }+M_{0,0}\omega_{0,0}^{\left(  k\right)  }$. But
$\mathcal{O}_{0,1}=\operatorname{im}\left(  p_{0,1}^{0}\right)  \oplus
\operatorname{im}\left(  \partial_{0,1}^{1}\right)  $ by Corollary
\ref{corDecom}, therefore $\varphi_{0,1}-M_{00}\theta_{0,0}=\chi
_{0,1}+\partial_{0,1}^{1}\gamma_{0,1}$ with $\chi_{01}\in\operatorname{im}%
\left(  p_{0,1}^{0}\right)  $ and $\gamma_{0,1}\in\mathcal{O}_{0,1}^{\oplus2}%
$. Using (\ref{fi11}), (\ref{fi12}) and the fact that $M_{0,0}$ is continuous,
we deduce that
\begin{align*}
\chi_{0,1}  &  =\varphi_{0,1}-\partial_{0,1}^{1}\gamma_{0,1}-M_{00}%
\theta_{0,0}=\lim_{k}\partial_{0,1}^{1}\left(  \psi_{0,1}^{\left(  k\right)
}-\gamma_{0,1}\right)  +M_{0,0}\omega_{0,0}^{\left(  k\right)  }\\
&  =\lim_{k}\partial_{0,1}^{1}\left(  \psi_{0,1}^{\left(  k\right)  }%
-\gamma_{0,1}\right)  +M_{0,0}\left(  s_{0,0}^{\left(  k\right)  }%
+\partial_{0,0}^{0}t_{0,0}^{\left(  k\right)  }\right)  =\lim_{k}%
\partial_{0,1}^{1}\left(  \omega_{0,1}^{\left(  k\right)  }\right)
+M_{0,0}\left(  s_{0,0}^{\left(  k\right)  }\right)  ,
\end{align*}
where
\begin{equation}
\omega_{0,1}^{\left(  k\right)  }=\psi_{0,1}^{\left(  k\right)  }-\gamma
_{0,1}-M_{0,0}^{0}t_{0,0}^{\left(  k\right)  } \label{fi13}%
\end{equation}
(recall that $M_{0,0}\partial_{0,0}^{0}=-\partial_{0,1}^{1}M_{0,0}^{0}$ by
Lemma \ref{dmn}). In a similar way, we have $\varphi_{1,0}+N_{00}\theta
_{0,0}=\chi_{1,0}+\partial_{1,0}^{1}\gamma_{1,0}$ with $\chi_{1,0}%
\in\operatorname{im}\left(  p_{1,0}^{0}\right)  $, and
\[
\chi_{1,0}=\varphi_{1,0}-\partial_{1,0}^{1}\gamma_{1,0}+N_{00}\theta
_{0,0}=\lim_{k}\partial_{1,0}^{1}\left(  \omega_{1,0}^{\left(  k\right)
}\right)  -N_{0,0}\left(  s_{0,0}^{\left(  k\right)  }\right)  ,
\]
where $\omega_{1,0}^{\left(  k\right)  }=\psi_{1,0}^{\left(  k\right)
}-\gamma_{1,0}-N_{0,0}^{0}t_{0,0}^{\left(  k\right)  }$. By Lemma \ref{lemOMN}
(applied to the case of $n=1$ with $\chi_{0,0}=0$), we conclude that the
limits $\lim_{k}s_{0,0}^{\left(  k\right)  }=s_{0,0}$, $\lim_{k}\partial
_{0,1}^{1}\left(  \omega_{0,1}^{\left(  k\right)  }\right)  =\partial
_{0,1}^{1}\left(  \omega_{0,1}\right)  $, $\lim_{k}\partial_{1,0}^{1}\left(
\omega_{1,0}^{\left(  k\right)  }\right)  =\partial_{1,0}^{1}\left(
\omega_{1,0}\right)  $ do exist, and
\[
\chi_{0,1}=\partial_{0,1}^{1}\left(  \omega_{0,1}\right)  +M_{0,0}\left(
s_{0,0}\right)  ,\quad\chi_{1,0}=\partial_{1,0}^{1}\left(  \omega
_{1,0}\right)  -N_{0,0}\left(  s_{0,0}\right)  .
\]
It follows that $\varphi_{0,1}=\partial_{0,1}^{1}\left(  \omega_{0,1}%
+\gamma_{0,1}\right)  +M_{0,0}\left(  s_{0,0}+\theta_{0,0}\right)  $ and
$\varphi_{1,0}=\partial_{1,0}^{1}\left(  \omega_{1,0}+\gamma_{1,0}\right)
-N_{0,0}\left(  s_{0,0}+\theta_{0,0}\right)  $. Since $s_{0,0}\in\ker\left(
\partial_{0,0}^{1}\right)  $ and $\varphi_{0,0}=\partial_{0,0}^{1}\theta
_{0,0}$, one can replace $\theta_{0,0}$ by $s_{0,0}+\theta_{0,0}$, and we put
$\theta_{0,1}=\omega_{0,1}+\gamma_{0,1}$ and $\theta_{1,0}=\omega_{1,0}%
+\gamma_{1,0}$. Thus
\[
\varphi_{0,1}=\partial_{0,1}^{1}\left(  \theta_{0,1}\right)  +M_{0,0}\left(
\theta_{0,0}\right)  ,\quad\varphi_{1,0}=\partial_{1,0}^{1}\left(
\theta_{1,0}\right)  -N_{0,0}\left(  \theta_{0,0}\right)  ,
\]
and in this case, we have (see (\ref{fi13}))
\begin{align*}
\partial_{0,1}^{1}\left(  \theta_{0,1}\right)   &  =\partial_{0,1}^{1}\left(
\omega_{0,1}+\gamma_{0,1}\right)  =\lim_{k}\partial_{0,1}^{1}\left(
\omega_{0,1}^{\left(  k\right)  }+\gamma_{0,1}\right)  =\lim_{k}\partial
_{0,1}^{1}\left(  \phi_{0,1}^{\left(  k\right)  }\right)  ,\\
\partial_{1,0}^{1}\left(  \theta_{1,0}\right)   &  =\partial_{1,0}^{1}\left(
\omega_{1,0}+\gamma_{1,0}\right)  =\lim_{k}\partial_{1,0}^{1}\left(
\omega_{1,0}^{\left(  k\right)  }+\gamma_{1,0}\right)  =\lim_{k}\partial
_{1,0}^{1}\left(  \phi_{1,0}^{\left(  k\right)  }\right)  .
\end{align*}
where $\phi_{0,1}^{\left(  k\right)  }=\psi_{0,1}^{\left(  k\right)  }%
-M_{0,0}^{0}t_{0,0}^{\left(  k\right)  }$ and $\phi_{1,0}^{\left(  k\right)
}=\psi_{1,0}^{\left(  k\right)  }-N_{0,0}^{0}t_{0,0}^{\left(  k\right)  }$.
Moreover, based on Lemma \ref{dmn}, we have
\begin{align*}
M_{0,1}\phi_{0,1}^{\left(  k\right)  }  &  =M_{0,1}\psi_{0,1}^{\left(
k\right)  }-M_{0,1}M_{0,0}^{0}t_{0,0}^{\left(  k\right)  }=M_{0,1}\psi
_{0,1}^{\left(  k\right)  },\\
M_{1,0}\phi_{1,0}^{\left(  k\right)  }-N_{0,1}\phi_{0,1}^{\left(  k\right)  }
&  =M_{1,0}\psi_{1,0}^{\left(  k\right)  }-M_{1,0}N_{0,0}^{0}t_{0,0}^{\left(
k\right)  }-N_{0,1}\psi_{0,1}^{\left(  k\right)  }+N_{0,1}M_{0,0}^{0}%
t_{0,0}^{\left(  k\right)  }=M_{1,0}\psi_{1,0}^{\left(  k\right)  }%
-N_{0,1}\psi_{0,1}^{\left(  k\right)  },\\
N_{1,0}\phi_{1,0}^{\left(  k\right)  }  &  =N_{1,0}\psi_{1,0}^{\left(
k\right)  }-N_{1,0}N_{0,0}^{0}t_{0,0}^{\left(  k\right)  }=N_{1,0}\psi
_{1,0}^{\left(  k\right)  }.
\end{align*}
Thus
\begin{align*}
\varphi_{0,2}  &  =\lim_{k}\partial_{0,2}^{1}\psi_{0,2}^{\left(  k\right)
}+M_{0,1}\psi_{0,1}^{\left(  k\right)  }=\lim_{k}\partial_{0,2}^{1}\psi
_{0,2}^{\left(  k\right)  }+M_{0,1}\phi_{0,1}^{\left(  k\right)  }\\
\varphi_{1,1}  &  =\lim_{k}\partial_{1,1}^{1}\psi_{d,1}^{\left(  k\right)
}+M_{1,0}\psi_{1,0}^{\left(  k\right)  }-N_{0,1}\psi_{0,1}^{\left(  k\right)
}=\lim_{k}\partial_{1,1}^{1}\psi_{d,1}^{\left(  k\right)  }+M_{1,0}\phi
_{1,0}^{\left(  k\right)  }-N_{0,1}\phi_{0,1}^{\left(  k\right)  },\\
\varphi_{2,0}  &  =\lim_{k}\partial_{2,0}^{1}\psi_{2,0}^{\left(  k\right)
}-N_{1,0}\psi_{1,0}^{\left(  k\right)  }=\lim_{k}\partial_{2,0}^{1}\psi
_{2,0}^{\left(  k\right)  }-N_{1,0}\phi_{1,0}^{\left(  k\right)  },
\end{align*}
which means that $\phi_{0,1}^{\left(  k\right)  }$ can be replaced by
$\psi_{0,1}^{\left(  k\right)  }$, and $\phi_{1,0}^{\left(  k\right)  }$ by
$\psi_{1,0}^{\left(  k\right)  }$, and we can assume that $\lim_{k}%
\partial_{0,1}^{1}\left(  \psi_{0,1}^{\left(  k\right)  }-\theta_{0,1}\right)
=0$ and $\lim_{k}\partial_{1,0}^{1}\left(  \psi_{1,0}^{\left(  k\right)
}-\theta_{1,0}\right)  =0$ hold. Notice that $\psi_{i,j}^{\left(  k\right)  }%
$, $i+j=1$ are only used for $\varphi_{i,j}$ with $i+j=2$.

Now suppose that there exist $\theta_{i,j}\in\mathcal{O}_{i,j}^{\oplus2}$,
$i+j<n$ such that $\varphi_{i,j}=\partial_{i,j}^{1}\theta_{i,j}+M_{i,j-1}%
\theta_{i,j-1}-N_{i-1,j}\theta_{i-1,j}$ and $\lim_{k}\partial_{i,j}^{1}\left(
\psi_{i,j}^{\left(  k\right)  }-\theta_{i,j}\right)  =0$ hold for all $i+j<n$.
Put $\omega_{i,j}^{\left(  k\right)  }=\psi_{i,j}^{\left(  k\right)  }%
-\theta_{i,j}$. As above, using Corollary \ref{corDecom}, we have
\[%
\begin{array}
[c]{cc}%
\mathcal{O}_{i,j}^{\oplus2} & \\
\parallel & \\
\operatorname{im}\left(  p_{i,j}^{1}\right)  & \overset{\partial_{i,j}%
^{1}}{\longrightarrow}\\
\oplus & \\
\operatorname{im}\left(  p_{i,j}\right)  \oplus\operatorname{im}\left(
\partial_{i,j}^{0}\right)  & =\ker\left(  \partial_{i,j}^{1}\right)
\end{array}%
\begin{array}
[c]{c}%
\\
\\
\mathcal{O}_{i,j}\\
\\
\end{array}
\]
Thus, there are sequences $\left\{  s_{i,j}^{\left(  k\right)  }\right\}
\subseteq\operatorname{im}\left(  p_{i,j}\right)  $ and $\left\{
t_{i,j}^{\left(  k\right)  }\right\}  \subseteq\mathcal{O}_{i,j}$ such that
\begin{equation}
\omega_{i,j}^{\left(  k\right)  }-s_{i,j}^{\left(  k\right)  }-\partial
_{i,j}^{0}t_{i,j}^{\left(  k\right)  }\rightarrow0\quad\text{for all\quad
}i+j<n. \label{fi2}%
\end{equation}
Fix $d,l$ with $d+l\leq n$. Using (\ref{fi1}), we obtain that
\begin{equation}
\varphi_{d,l}-M_{d,l-1}\theta_{d,l-1}+N_{d-1,l}\theta_{d-1,l}=\lim_{k}%
\partial_{d,l}^{1}\psi_{d,l}^{\left(  k\right)  }+M_{d,l-1}\left(
\omega_{d,l-1}^{\left(  k\right)  }\right)  -N_{d-1,l}\left(  \omega
_{d-1,l}^{\left(  k\right)  }\right)  . \label{fi3}%
\end{equation}
Moreover, taking into account the decomposition $\mathcal{O}_{d,l}%
=\operatorname{im}\left(  p_{d,l}^{0}\right)  \oplus\operatorname{im}\left(
\partial_{d,l}^{1}\right)  $ (see Corollary \ref{corDecom}), we have
$\varphi_{d,l}-M_{d,l-1}\theta_{d,l-1}+N_{d-1,l}\theta_{d-1,l}=\chi
_{d,l}+\partial_{d,l}^{1}\gamma_{d,l}$ with $\chi_{d,l}\in\operatorname{im}%
\left(  p_{d,l}^{0}\right)  $ and $\gamma_{d,l}\in\mathcal{O}_{d,l}^{\oplus2}%
$. Put
\begin{equation}
\omega_{d,l}^{\left(  k\right)  }=\psi_{d,l}^{\left(  k\right)  }-\gamma
_{d,l}-M_{d,l-1}^{0}t_{d,l-1}^{\left(  k\right)  }-N_{d-1,l}^{0}%
t_{d-1,l}^{\left(  k\right)  }\in\mathcal{O}_{d,l}^{\oplus2}. \label{fi4}%
\end{equation}
Using (\ref{fi3}), (\ref{fi2}), and the continuity of the operators
$M_{d,l-1}$ and $N_{d-1,l}$, we deduce that
\begin{align*}
\chi_{d,l}  &  =\lim_{k}\partial_{d,l}^{1}\left(  \psi_{d,l}^{\left(
k\right)  }-\gamma_{d,l}\right)  +M_{d,l-1}\left(  \omega_{d,l-1}^{\left(
k\right)  }\right)  -N_{d-1,l}\left(  \omega_{d-1,l}^{\left(  k\right)
}\right) \\
&  =\lim_{k}\partial_{d,l}^{1}\left(  \psi_{d,l}^{\left(  k\right)  }%
-\gamma_{d,l}\right)  +M_{d,l-1}\left(  s_{d,l-1}^{\left(  k\right)
}+\partial_{d,l-1}^{0}t_{d,l-1}^{\left(  k\right)  }\right)  -N_{d-1,l}\left(
s_{d-1,l}^{\left(  k\right)  }+\partial_{d-1,l}^{0}t_{d-1,l}^{\left(
k\right)  }\right) \\
&  =\lim_{k}\partial_{d,l}^{1}\left(  \omega_{d,l}^{\left(  k\right)
}\right)  +M_{d,l-1}\left(  s_{d,l-1}^{\left(  k\right)  }\right)
-N_{d-1,l}\left(  s_{d-1,l}^{\left(  k\right)  }\right)  ,
\end{align*}
Recall again that $M_{d,l-1}\partial_{d,l-1}^{0}=-\partial_{d,l}^{1}%
M_{d,l-1}^{0}$ and $N_{d-1,l}\partial_{d-1,l}^{0}=\partial_{d,l}^{1}%
N_{d-1,l}^{0}$ thanks to Lemma \ref{dmn}. By Lemma \ref{lemOMN}, there are
limits $\lim_{k}s_{d,l-1}^{\left(  k\right)  }=s_{d,l-1}$, $\lim_{k}%
s_{d-1,l}^{\left(  k\right)  }=s_{d-1,l}$ and $\lim_{k}\partial_{d,l}%
^{1}\left(  \omega_{d,l}^{\left(  k\right)  }\right)  =\partial_{d,l}%
^{1}\left(  \omega_{d,l}\right)  $ with $\chi_{d,l}=\partial_{d,l}^{1}\left(
\omega_{d,l}\right)  +M_{d,l-1}\left(  s_{d,l-1}\right)  -N_{d-1,l}\left(
s_{d-1,l}\right)  $. It follows that
\[
\varphi_{d,l}=\partial_{d,l}^{1}\left(  \omega_{d,l}+\gamma_{d,l}\right)
+M_{d,l-1}\left(  s_{d,l-1}+\theta_{d,l-1}\right)  -N_{d-1,l}\left(
s_{d-1,l}+\theta_{d-1,l}\right)  .
\]
Now we can replace $\theta_{d,l-1}$ by $s_{d,l-1}+\theta_{d,l-1}$ (resp.,
$\theta_{d-1,l}$ by $s_{d-1,l}+\theta_{d-1,l}$) just in one degree less terms.
Indeed, $s_{d,l-1}\in\ker\left(  \partial_{d,l-1}^{1}\right)  $, $s_{d-1,l}%
\in\ker\left(  \partial_{d-1,l}^{1}\right)  $, and $\theta_{d,l-1}$,
$\theta_{d-1,l}$ occur only in the following expressions
\begin{align*}
\varphi_{d,l-1}  &  =\partial_{d,l-1}^{1}\theta_{d,l-1}+M_{d,l-2}%
\theta_{d,l-2}-N_{d-1,l-1}\theta_{d-1,l-1},\\
\varphi_{d-1,l}  &  =\partial_{d-1,l}^{1}\theta_{d-1,l}+M_{d-1,l-1}%
\theta_{d-1,l-1}-N_{d-2,l}\theta_{d-2,l}.
\end{align*}
Put $\theta_{d,l}=\omega_{d,l}+\gamma_{d,l}$. Then%
\[
\varphi_{d,l}=\partial_{d,l}^{1}\theta_{d,l}+M_{d,l-1}\theta_{d,l-1}%
-N_{d-1,l}\theta_{d-1,l}%
\]
as it was required. Moreover, using (\ref{fi4}). we have
\[
\partial_{d,l}^{1}\left(  \theta_{d,l}\right)  =\lim_{k}\partial_{d,l}%
^{1}\left(  \omega_{d,l}^{\left(  k\right)  }+\gamma_{d,l}\right)  =\lim
_{k}\partial_{d,l}^{1}\phi_{d,l}^{\left(  k\right)  },
\]
where $\phi_{d,l}^{\left(  k\right)  }=\psi_{d,l}^{\left(  k\right)
}-M_{d,l-1}^{0}t_{d,l-1}^{\left(  k\right)  }-N_{d-1,l}^{0}t_{d-1,l}^{\left(
k\right)  }$. In this case, as above using Lemma \ref{dmn}, we conclude that
\begin{align*}
M_{d,l}\phi_{d,l}^{\left(  k\right)  }-N_{d-1,l+1}\phi_{d-1,l+1}^{\left(
k\right)  }  &  =M_{d,l}\psi_{d,l}^{\left(  k\right)  }-N_{d-1,l+1}%
\psi_{d-1,l+1}^{\left(  k\right)  }\\
&  -M_{d,l}M_{d,l-1}^{0}t_{d,l-1}^{\left(  k\right)  }-M_{d,l}N_{d-1,l}%
^{0}t_{d-1,l}^{\left(  k\right)  }\\
&  +N_{d-1,l+1}M_{d-1,l}^{0}t_{d-1,l}^{\left(  k\right)  }+N_{d-1,l+1}%
N_{d-2,l+1}^{0}t_{d-2,l+1}^{\left(  k\right)  }\\
&  =M_{d,l}\psi_{d,l}^{\left(  k\right)  }-N_{d-1,l+1}\psi_{d-1,l+1}^{\left(
k\right)  }.
\end{align*}
By (\ref{fi1}), we obtain that
\[
\varphi_{d,l+1}=\lim_{k}\partial_{d,l+1}^{1}\psi_{d,l+1}^{\left(  k\right)
}+M_{d,l}\phi_{d,l}^{\left(  k\right)  }-N_{d-1,l+1}\phi_{d-1,l+1}^{\left(
k\right)  }.
\]
In a similar way, we have $M_{d+1,l-1}\phi_{d+1,l-1}^{\left(  k\right)
}-N_{d,l}\phi_{d,l}^{\left(  k\right)  }=M_{d+1,l-1}\psi_{d+1,l-1}^{\left(
k\right)  }-N_{d,l}\psi_{d,l}^{\left(  k\right)  }$ and the same limit holds
for $\varphi_{d+1,l}$. Thus we can replace $\phi_{i,j}^{\left(  k\right)  }$
by $\psi_{i,j}^{\left(  k\right)  }$ and assume that $\lim_{k}\partial
_{i,j}^{1}\left(  \psi_{i,j}^{\left(  k\right)  }-\theta_{i,j}\right)  =0$
whenever $i+j\leq n$. Hence there exists $\theta=\left(  \theta_{d,l}\right)
_{d,l}\in\prod\limits_{\left(  d,l\right)  \in\mathbb{Z}_{+}^{2}}%
\mathcal{O}_{d,l}^{\oplus2}$ such that $d^{1}\theta=\varphi$, which means that
$\varphi$ belongs to $\operatorname{im}\left(  d^{1}\right)  $.
\end{proof}

\subsection{The main result on localizations}

Now we can prove the main result of the paper.

\begin{theorem}
\label{thCENTER}Let $U\subseteq\mathbb{C}_{xy}$ be a $q$-open subset. The
canonical embeddings of the contractive quantum plane $\mathfrak{A}_{q}$ into
the Arens-Michael-Fr\'{e}chet algebras $\mathbb{C}_{q}\left[  \left[
x,y\right]  \right]  $, $\mathcal{O}\left(  U_{x}\right)  \left[  \left[
y\right]  \right]  $ and $\left[  \left[  x\right]  \right]  \mathcal{O}%
\left(  U_{y}\right)  $ are localizations. Moreover, $\mathfrak{A}%
_{q}\rightarrow\mathcal{F}_{q}\left(  U\right)  $ is a localization whenever
$U$ is a Runge $q$-open subset.
\end{theorem}

\begin{proof}
Based on Propositions \ref{propLoc1} and \ref{propLoc2}, we conclude that the
canonical morphisms of $\mathfrak{A}_{q}$ into the Arens-Michael-Fr\'{e}chet
algebras $\mathbb{C}_{q}\left[  \left[  x,y\right]  \right]  $, $\mathcal{O}%
\left(  U_{x}\right)  \left[  \left[  y\right]  \right]  $ and $\left[
\left[  x\right]  \right]  \mathcal{O}\left(  U_{y}\right)  $ are
localizations. It remains to prove the exactness of the related complex
(\ref{Rsp}) for $\mathcal{A=F}_{q}\left(  U\right)  $ whenever $U$ is a Runge
$q$-open subset. Recall that the complex (\ref{Rsp}) for $\mathcal{A=F}%
_{q}\left(  U\right)  $ is reduced to the complex (\ref{Oc}).

Based on (\ref{d01}), the differentials $d^{0}$ and $d^{1}$ of the complex
(\ref{Rsp}) have lower triangular operator matrices, that is, we come up with
a triangular complex (see Subsection \ref{subsecTCs}). The related diagonal
complexes are the cochain complexes $\mathcal{O}_{q}\left(  d,l\right)  $,
$\left(  d,l\right)  \in\mathbb{Z}_{+}^{2}$ (see (\ref{OqD})) considered in
Appendix Subsection \ref{subsecDFC1}, whose cohomology groups are denoted by
$H_{d,l}^{i}$, $0\leq i\leq3$. By Theorem \ref{thMain1}, $H_{d,l}^{0}=\left\{
0\right\}  $, $H_{d,l}^{1}=\mathcal{O}\left(  U\right)  $, $H_{d,l}%
^{2}=\left\{  0\right\}  $ and $H_{d,l}^{3}=\left\{  0\right\}  $. Since
$H_{d,l}^{0}=\left\{  0\right\}  $ and $H_{d,l}^{1}$ are Hausdorff topological
spaces, it follows that all diagonal differentials $\partial_{d,l}^{0}$ are
topologically injective linear maps. But they are diagonals of the lower
triangular operator matrix $d^{0}$. By Lemma \ref{lemTri}, we conclude that
$d^{0}$ is topologically injective.

Let us prove that $\ker\left(  d^{1}\right)  =\operatorname{im}\left(
d^{0}\right)  $, where $d^{i}$, $i=0,1$ are the differentials of the complex
(\ref{Oc}). Take a nonzero $\xi\in\ker\left(  d^{1}\right)  $, which has the
form $\xi=\left(  \xi_{i,j}\right)  \in\prod\limits_{\left(  i,j\right)
\in\mathbb{Z}_{+}^{2}}\mathcal{O}_{i,j}^{\oplus2}$. Recall that $\mathcal{O}%
_{i,j}$ is the $\left(  i,j\right)  $-isomorphic copy of $\mathcal{O}\left(
i,j\right)  $ (see Subsection \ref{subsecIO}). Assume that $\left(
d,l\right)  $ is the largest in $\mathbb{Z}_{+}^{2}$ (with respect to the
linear order from (\ref{oI})) with $\xi_{d,l}\neq0$, that is, $\xi_{i,j}=0$
for all $\left(  i,j\right)  <\left(  d,l\right)  $ (see Subsection
\ref{subsecIO}). Since $\left(  d,l-1\right)  ,\left(  d-1,l\right)  <\left(
d,l\right)  $ and $\xi_{d,l-1}=\xi_{d-1,l}=0$, it follows using (\ref{dd01})
that
\[
\partial_{d,l}^{1}\xi_{d,l}=\partial_{d,l}^{1}\xi_{d,l}+M_{d,l-1}\xi
_{d,l-1}-N_{d-1,l}\xi_{d-1,l}=\left(  d^{1}\xi\right)  _{d,l}=0,
\]
that is, $\xi_{d,l}\in\ker\left(  \partial_{d,l}^{1}\right)  $. Recall (see
Subsection \ref{subsecIO}) that
\[
M_{d,l-1}=\left[
\begin{array}
[c]{cc}%
M_{x_{2},l-1} & M_{y_{2},l-1}%
\end{array}
\right]  :\mathcal{O}_{d,l-1}^{\oplus2}\rightarrow\mathcal{O}_{d,l}\text{,
\quad}N_{d-1,l}=\left[
\begin{array}
[c]{cc}%
N_{y_{1},d-1} & N_{x_{1},d-1}%
\end{array}
\right]  :\mathcal{O}_{d-1,l}^{\oplus2}\rightarrow\mathcal{O}_{d,l},
\]
whose entries are the operators (\ref{MTN}) introduced in Subsection
\ref{subsecQRP}. Actually these entries are obtained by the tensor inflating
of the following operators
\[
N_{x,d-1},N_{y,d-1}:\mathcal{O}\left(  d-1\right)  \rightarrow\mathcal{O}%
\left(  d\right)  ,\quad M_{x,l-1},M_{y,l-1}:\mathcal{O}\left(  l-1\right)
\rightarrow\mathcal{O}\left(  l\right)
\]
considered in (\ref{Nd}). Notice that $H_{d,l}^{1}\neq\left\{  0\right\}  $.
But $M_{d,l}\xi_{d,l}\in\mathcal{O}_{d,l+1}$ and $N_{d-1,l+1}\xi_{d-1,l+1}%
\in\mathcal{O}_{d,l+1}$ with $\left(  d-1,l+1\right)  <\left(  d,l\right)  $.
It follows that $\xi_{d-1,l+1}=0$. Hence $d^{1}\xi=0$ implies that
\[
M_{d,l}\xi_{d,l}+\partial_{d,l+1}^{1}\xi_{d,l+1}=\partial_{d,l+1}^{1}%
\xi_{d,l+1}+M_{d,l}\xi_{d,l}-N_{d-1,l+1}\xi_{d-1,l+1}=\left(  d^{1}\xi\right)
_{d,l+1}=0.
\]
In the case of the opposite linear order for the homogeneous parts
$\mathbb{Z}_{+}^{2}\left(  n\right)  $, $n\in\mathbb{Z}_{+}$, we obtain that
$\partial_{d+1,l}^{1}\xi_{d+1,l}-N_{d,l}\xi_{d,l}=0$. Thus $M_{d,l}\xi
_{d,l}\in\operatorname{im}\left(  \partial_{d,l+1}^{1}\right)  $ (or
$N_{d,l}\xi_{d,l}\in\operatorname{im}\left(  \partial_{d+1,l}^{1}\right)  $).
But $\operatorname{im}\left(  \partial_{d,l+1}^{1}\right)  =\ker\left(
\pi_{d,l+1}\right)  $ or $H_{d,l+1}^{2}=\left\{  0\right\}  $ (see Theorem
\ref{thMain1}), that is, $M_{d,l}\xi_{d,l}\in\ker\left(  \pi_{d,l+1}\right)
$. Using Lemma \ref{lemDDkey}, we conclude that $\xi_{d,l}\in\operatorname{im}%
\left(  \partial_{d,l}^{0}\right)  $, that is, $\xi_{d,l}=\partial_{d,l}%
^{0}\alpha_{d,l}$ for some $\alpha_{d,l}\in\mathcal{O}_{d,l}$. Put $\alpha
\in\mathcal{F}_{q}\left(  U\right)  ^{\widehat{\otimes}2}$ with $\alpha
_{i,j}=0$ whenever $\left(  i,j\right)  \neq\left(  d,l\right)  $. Using again
(\ref{dd01}), we derive that
\[
\left(  d^{0}\alpha\right)  _{i,j}=\partial_{i,j}^{0}\alpha_{i,j}%
+M_{i,j-1}^{0}\alpha_{i,j-1}+N_{i-1,j}^{0}\alpha_{i-1,j}=\partial_{i,j}%
^{0}\alpha_{i,j}%
\]
for all $\left(  i,j\right)  \leq\left(  d,l\right)  $, where
\[
M_{i,j-1}^{0}=\left[
\begin{array}
[c]{c}%
-qM_{y_{2},j-1}\\
M_{x_{2},j-1}%
\end{array}
\right]  :\mathcal{O}_{i,j-1}\rightarrow\mathcal{O}_{i,j}^{\oplus2},\quad
N_{i-1,j}^{0}=\left[
\begin{array}
[c]{c}%
N_{x_{1},i-1}\\
-qN_{y_{1},i-1}%
\end{array}
\right]  :\mathcal{O}_{i-1,j}\rightarrow\mathcal{O}_{i,j}^{\oplus2}%
\]
are the operators considered in Subsection \ref{subsecIO}. Thus $\left(
\xi-d^{0}\alpha\right)  _{i,j}=0$ for all $\left(  i,j\right)  \leq\left(
d,l\right)  $, and we come up with a new kernel element $\xi-d^{0}\alpha
\in\ker\left(  d^{1}\right)  $ whose entries are vanishing till $\left(
d,l\right)  $ included. Namely, $\xi-d^{0}\alpha\in\prod\limits_{\left(
i,j\right)  >\left(  d,l\right)  }\mathcal{O}_{i,j}^{\oplus2}$. By induction
on $\left(  i,j\right)  $, we obtain that $\xi-\sum_{k=1}^{n}d^{0}%
\alpha^{\left(  k\right)  }$ is small enough in $\prod\limits_{\left(
i,j\right)  \in\mathbb{Z}_{+}^{2}}\mathcal{O}_{i,j}^{\oplus2}$ for a sequence
$\left\{  \alpha^{\left(  k\right)  }\right\}  \subseteq\mathcal{F}_{q}\left(
U\right)  ^{\widehat{\otimes}2}$. Hence $\xi=\sum_{k}d^{0}\alpha^{\left(
k\right)  }$ belongs to the closure of $\operatorname{im}\left(  d^{0}\right)
$. But $\operatorname{im}\left(  d^{0}\right)  $ is closed. Hence $\xi
\in\operatorname{im}\left(  d^{0}\right)  $.

Finally, the image $\operatorname{im}\left(  d^{1}\right)  $ is closed thanks
to Proposition \ref{propD1P}, and $\mathfrak{A}_{q}$ is dense in
$\mathcal{F}_{q}\left(  U\right)  $ by Proposition \ref{propRunge}. Using
Proposition \ref{propKD}, we conclude $\operatorname{im}\left(  d^{1}\right)
=\ker\left(  \pi\right)  $. The fact that $\pi$ is onto immediate. Whence
(\ref{Rsp}) for $\mathcal{A=F}_{q}\left(  U\right)  $ is exact.
\end{proof}

Thus, the noncommutative analytic space $\left(  \mathbb{C}_{xy}%
,\mathcal{F}_{q}\right)  $ possesses the localization property.

\begin{corollary}
Let $U\subseteq\mathbb{C}_{xy}$ be a Runge $q$-open subset. Then both
projections $\pi_{x}:\mathcal{F}_{q}\left(  U\right)  \rightarrow
\mathcal{O}\left(  U_{x}\right)  \left[  \left[  y\right]  \right]  $,
$\pi_{y}:\mathcal{F}_{q}\left(  U\right)  \rightarrow\left[  \left[  x\right]
\right]  \mathcal{O}\left(  U_{y}\right)  $ and the canonical morphism
$\mathcal{F}_{q}\left(  U\right)  \rightarrow\mathbb{C}_{q}\left[  \left[
x,y\right]  \right]  $ are all localizations.
\end{corollary}

\begin{proof}
One needs to use Theorem \ref{thCENTER} and \cite[Proposition 1.8]{Tay2}.
\end{proof}

\section{Applications to the joint spectra\label{secJS}}

In this final section we consider the joint spectra within the general
framework of noncommutative functional calculus considered in \cite{DosJOT10},
and prove the key results on Putinar and Taylor spectra based on our main
Theorem \ref{thCENTER} on localizations.

\subsection{Taylor spectrum}

Let $X$ be a left Fr\'{e}chet $\mathcal{O}_{q}\left(  \mathbb{C}^{2}\right)
$-module given by a couple $\left(  T,S\right)  $ of continuous linear
operator on $X$ such that $TS=q^{-1}ST$. For brevity we say that $X$ is a left
Fr\'{e}chet $q$-module. If $X$ is a left Banach $\mathfrak{A}_{q}$-module,
then the homomorphism $\mathfrak{A}_{q}\rightarrow\mathcal{L}\left(  X\right)
$, $x\mapsto T$, $y\mapsto S$ into the Banach algebra $\mathcal{L}\left(
X\right)  $ has a unique continuous (algebra homomorphism) extension
$\mathcal{O}_{q}\left(  \mathbb{C}^{2}\right)  \rightarrow\mathcal{L}\left(
X\right)  $ to its Arens-Michael envelope $\mathcal{O}_{q}\left(
\mathbb{C}^{2}\right)  $ (see Subsection \ref{subsecAME}). It turns out that
$X$ is a left Fr\'{e}chet $q$-module automatically. Recall \cite{Dosi242} that
the resolvent set $\operatorname{res}\left(  T,S\right)  $ of these operators
(or the $q$-module $X$) is defined as a set of those $\gamma\in\mathbb{C}%
_{xy}$ such that the transversality relation $\mathbb{C}\left(  \gamma\right)
\perp X$ holds, that is, $\operatorname{Tor}_{k}^{\mathfrak{A}_{q}}\left(
\mathbb{C}\left(  \gamma\right)  ,X\right)  =\left\{  0\right\}  $ for all
$k\geq0$. The complement
\[
\sigma\left(  T,S\right)  =\mathbb{C}_{xy}^{2}\backslash\operatorname{res}%
\left(  T,S\right)
\]
defines the Taylor spectrum of the operator couple $\left(  T,S\right)  $. If
$U\subseteq\mathbb{C}_{xy}$ is a nonempty $q$-open subset such that
$U\cap\sigma\left(  T,S\right)  =\varnothing$ (or $U\subseteq
\operatorname{res}\left(  T,S\right)  $), then the transversality relation
$\mathcal{F}_{q}\left(  U\right)  \perp X$ holds whenever $X$ is a left Banach
$q$-module \cite[Theorem 6.5]{Dosi25}. But the reverse implication is strongly
depends on the localization property of the noncommutative analytic space
$\left(  \mathbb{C}_{xy},\mathcal{F}_{q}\right)  $ (see Definition
\ref{def3}). Based on Theorem \ref{thCENTER}, we conclude that $\mathcal{F}%
_{q}\left(  U\right)  \perp X$ holds iff $U\cap\sigma\left(  T,S\right)
=\varnothing$, whenever $U$ is a Runge $q$-open subset and $X$ is a left
Banach $q$-module. The same statement holds for the noncommutative analytic
spaces $\left(  \left\{  \left(  0,0\right)  \right\}  ,\mathbb{C}_{q}\left[
\left[  x,y\right]  \right]  \right)  $, $\left(  \mathbb{C}_{x}%
,\mathcal{O}\left[  \left[  y\right]  \right]  \right)  $ and $\left(
\mathbb{C}_{y},\left[  \left[  x\right]  \right]  \mathcal{O}\right)  $ (see
Definitions \ref{def1} and \ref{def2}) too.

\subsection{Putinar spectrum}

The Putinar spectrum of a left Fr\'{e}chet module with respect to a
Fr\'{e}chet algebra sheaf was introduced in \cite{DosiMS10}, \cite{DosJOT10},
that plays a central role in noncommutative functional calculus problems.
Based on \cite[Definition 4.2]{DosJOT10}, we define \textit{the resolvent set
}$\operatorname*{res}\left(  \mathcal{F}_{q},X\right)  $ of a left Fr\'{e}chet
$q$-module $X$ with respect to the sheaf $\mathcal{F}_{q}$ to be a subset of
those $\gamma\in\mathbb{C}_{xy}$ such that there exists a $q$-open
neighborhood $U_{\gamma}$ of $\gamma$ with the transversality property
$\mathcal{F}_{q}\left(  V\right)  \perp X$ for each $q$-open subset
$V\subseteq U_{\gamma}$. Actually, in \cite[Definition 4.2]{DosJOT10} it is
assumed that $\mathcal{F}_{q}\left(  V\right)  \perp X$ holds only for every
$q$-open $\mathcal{F}_{q}$-acyclic subset $V\subseteq U_{\gamma}$. But we can
skip this moment taking into account that every $q$-open subset of
$\mathbb{C}_{xy}$ is $\mathcal{F}_{q}$-acyclic \cite{DosiSS}. By its very
definition, $\operatorname*{res}\left(  \mathcal{F}_{q},X\right)  $ is a
$q$-open subset whereas that is not the case for $\operatorname{res}\left(
T,S\right)  $ (see below Example). The complement%
\[
\sigma\left(  \mathcal{F}_{q},X\right)  =\mathbb{C}_{xy}\backslash
\operatorname*{res}\left(  \mathcal{F}_{q},X\right)
\]
is called \textit{the Putinar spectrum }of the left Fr\'{e}chet $q$-module $X$
with respect to the sheaf $\mathcal{F}_{q}$.

\begin{lemma}
\label{corFUV}Let $U\subseteq\mathbb{C}_{xy}$ be a Runge $q$-open subset,
$V\subseteq U$ its $q$-open subset, and let $X$ be a left Fr\'{e}chet
$q$-module. If $\mathcal{F}_{q}\left(  U\right)  \perp X$ then $\mathcal{F}%
_{q}\left(  V\right)  \perp X$ holds too.
\end{lemma}

\begin{proof}
Put $A=\mathcal{O}_{q}\left(  \mathbb{C}_{xy}\right)  $. As we have mentioned
above (see Subsection \ref{subsecBRQP}) the canonical map $\mathfrak{A}%
_{q}\rightarrow A$ is an absolute localization \cite{Pir}. It follows that $A$
has the finite free $A$-bimodule resolution $\mathcal{R}\left(
A^{\widehat{\otimes}2}\right)  $, which in turn implies that $A$ is of finite
type in the sense of \cite[Definition 2.4]{Tay2}. By the sheaf property of
$\mathcal{F}_{q}$ (see Subsection \ref{subsecNCCP}), we deduce that
$\mathcal{F}_{q}\left(  V\right)  $ is a Fr\'{e}chet $\mathcal{F}_{q}\left(
U\right)  $-bimodule. Actually, we have the dominance property $\mathcal{F}%
_{q}\left(  U\right)  \gg\mathcal{F}_{q}\left(  V\right)  $ \cite[Definition
2.2]{Tay2} in the following sense that $\operatorname{Tor}_{n}^{A}\left(
\mathcal{F}_{q}\left(  U\right)  ,\mathcal{F}_{q}\left(  V\right)  \right)
=\left\{  0\right\}  $, $n>0$, and the natural mapping $\mathcal{F}_{q}\left(
V\right)  \rightarrow\mathcal{F}_{q}\left(  U\right)  \widehat{\otimes}%
_{A}\mathcal{F}_{q}\left(  V\right)  $, $\zeta\mapsto1\otimes\zeta$ is a
topological isomorphism. Indeed, using the fact that $A\rightarrow
\mathcal{F}_{q}\left(  U\right)  $ is a localization (Theorem \ref{thCENTER}),
and the assertion from \cite[Proposition 2.1]{Tay2}, we deduce that
$\mathcal{F}_{q}\left(  U\right)  $ is dominating over $\mathcal{F}_{q}\left(
V\right)  $. Taking into account the nuclearity property of $\mathcal{F}%
_{q}\left(  U\right)  $, we conclude that $\mathcal{F}_{q}\left(  U\right)
\perp X$ and $\mathcal{F}_{q}\left(  U\right)  \gg\mathcal{F}_{q}\left(
V\right)  $ imply $\mathcal{F}_{q}\left(  V\right)  \perp X$ \cite[Proposition
2.7 (b)]{Tay2}.
\end{proof}

Based on Lemma \ref{corFUV}, we obtain that $\gamma\in\operatorname*{res}%
\left(  \mathcal{F}_{q},X\right)  $ if and only if $\mathcal{F}_{q}\left(
U_{\gamma}\right)  \perp X$ for a certain Runge $q$-open subset $U_{\gamma}$
containing $\gamma$. Recall that the Runge $q$-open subsets form a topology
base of the $q$-topology in $\mathbb{C}_{xy}$. Thus $\sigma\left(
\mathcal{F}_{q},X\right)  $ is a $q$-closed subset of $\mathbb{C}_{xy}$.

\begin{theorem}
\label{corPT}Let $X$ be a left Fr\'{e}chet $q$-module given by an operator
couple $\left(  T,S\right)  $. Then $\sigma\left(  T,S\right)  ^{-q}%
\subseteq\sigma\left(  \mathcal{F}_{q},X\right)  $ holds, where $\sigma\left(
T,S\right)  ^{-q}$ is the $q$-closure of the Taylor spectrum $\sigma\left(
T,S\right)  $. If $X$ is a left Banach $q$-module, then
\[
\sigma\left(  \mathcal{F}_{q},X\right)  =\sigma\left(  T,S\right)  ^{-q}.
\]

\end{theorem}

\begin{proof}
First take $\gamma\in\operatorname*{res}\left(  \mathcal{F}_{q},X\right)  $.
Then $\mathcal{F}_{q}\left(  U_{\gamma}\right)  \perp X$ for a certain Runge
$q$-open subset $U_{\gamma}$ containing $\gamma$. Since $\mathbb{C}\left(
\gamma\right)  $ is a Fr\'{e}chet $\mathcal{F}_{q}\left(  U_{\gamma}\right)
$-bimodule and $\mathfrak{A}_{q}\rightarrow\mathcal{F}_{q}\left(  U\right)  $
is a localization (Theorem \ref{thCENTER}), it follows that $\mathbb{C}\left(
\gamma\right)  \perp X$ holds too (see \cite[Proposition 2.7 (b)]{Tay2}). Thus
$\operatorname*{res}\left(  \mathcal{F}_{q},X\right)  \subseteq
\operatorname{res}\left(  T,S\right)  $ or $\sigma\left(  T,S\right)
\subseteq\sigma\left(  \mathcal{F}_{q},X\right)  $. But $\sigma\left(
\mathcal{F}_{q},X\right)  $ is a $q$-closed subset (Lemma \ref{corFUV}),
therefore $\sigma\left(  T,S\right)  ^{-q}\subseteq\sigma\left(
\mathcal{F}_{q},X\right)  $.

Further, assume that $X$ is a left Banach $q$-module, and take $\gamma
\notin\sigma\left(  T,S\right)  ^{-q}$. Then $U_{\gamma}\cap\sigma\left(
T,S\right)  =\varnothing$ for a certain Runge $q$-open subset $U_{\gamma}$
containing $\gamma$. It follows that $\mathbb{C}\left(  \lambda\right)  \perp
X$ for all $\lambda\in U_{\gamma}$. Using \cite[Theorem 6.2]{Dosi25}, we
deduce that $\mathcal{F}_{q}\left(  U_{\gamma}\right)  \perp X$. By Lemma
\ref{corFUV}, we conclude that $\gamma\in\operatorname*{res}\left(
\mathcal{F}_{q},X\right)  $. Whence $\sigma\left(  \mathcal{F}_{q},X\right)
=\sigma\left(  T,S\right)  ^{-q}$.
\end{proof}

Now let us illustrate the statement of Theorem \ref{corPT} in the following
example to see the gap between these spectra.

\textbf{Example. }Put $X=\ell_{p}\left(  \mathbb{Z}_{+}\right)  $, $1\leq
p<\infty$ with its standard topological basis $\left\{  e_{n}\right\}
_{n\in\mathbb{Z}_{+}}$, and consider the unilateral shift operator $T$ and the
diagonal $q$-operator $S$ on $\ell_{p}$. Thus $T,S\in\mathcal{B}\left(
X\right)  $ with $T\left(  e_{n}\right)  =e_{n+1}$ and $S\left(  e_{n}\right)
=q^{n}e_{n}$ for all $n\in\mathbb{Z}_{+}$. In the case of $\left\vert
q\right\vert <1$, the operator $S$ is compact with its spectrum $\sigma\left(
S\right)  =\left\{  q^{n}:n\in\mathbb{Z}_{+}\right\}  \cup\left\{  0\right\}
$, which is the $q$-hull of $\left\{  1\right\}  $, and $\sigma\left(
T\right)  =\mathbb{D}_{1}$ is the closed unit disk, which is $q$-dense in
$\mathbb{C}$. Since $TS=q^{-1}ST$, it follows that $X$ is a left Banach
$q$-module. It is proved (see \cite{Dosi242}) that
\[
\sigma\left(  T,S\right)  =\left(  \left(  \mathbb{D}_{\left\vert q\right\vert
^{-1}}\backslash\mathbb{D}_{1}^{\circ}\right)  \times\left\{  0\right\}
\right)  \cup\left\{  \left(  0,1\right)  \right\}  ,
\]
where $\mathbb{D}_{1}^{\circ}=B\left(  0,1\right)  $. The $q$-closure of
$\sigma\left(  T,S\right)  \cap\mathbb{C}_{x}$ equals to the complement of the
disk $\mathbb{D}_{1}^{\circ}\times\left\{  0\right\}  $. Indeed, for every
$z\in\mathbb{C}$ with $\left\vert z\right\vert >\left\vert q\right\vert ^{-1}$
there is a positive integer $n\geq1$ such that $\left\vert q\right\vert
^{-n}<\left\vert z\right\vert \leq\left\vert q\right\vert ^{-n-1}$. It follows
that $q^{n}z\in\mathbb{D}_{\left\vert q\right\vert ^{-1}}\backslash
\mathbb{D}_{1}^{\circ}$ or $\left(  q^{n}z,0\right)  \in\sigma\left(
T,S\right)  $. Thus $\left(  z,0\right)  =q^{-n}\left(  q^{n}z,0\right)
\in\sigma\left(  T,S\right)  ^{-q}$ (see Subsection \ref{subsecQT}). Moreover,
the $q$-closure of $\sigma\left(  T,S\right)  \cap\mathbb{C}_{y}$ equals to
$\left\{  \left(  0,q^{-n}\right)  :n\geq0\right\}  $. By Theorem \ref{corPT},
we conclude that%
\begin{align*}
\sigma\left(  \mathcal{F}_{q},X\right)   &  =\sigma\left(  T,S\right)
^{-q}=\left(  \left(  \mathbb{D}_{\left\vert q\right\vert ^{-1}}%
\backslash\mathbb{D}_{1}^{\circ}\right)  \times\left\{  0\right\}  \right)
^{-q}\cup\left\{  \left(  0,1\right)  \right\}  ^{-q}\\
&  =\left(  \mathbb{C}\backslash\mathbb{D}_{1}^{\circ}\times\left\{
0\right\}  \right)  \cup\left\{  \left(  0,q^{-n}\right)  :n\geq0\right\}  .
\end{align*}

\textbf{The concluding remark. }It turns out that the Taylor spectrum
$\sigma\left(  T,S\right)  $ in the $q$-theory could admit a large $q$-closure
that results in the Putinar spectrum. This moment explains the lack of the
conventional projection property for the Taylor spectrum observed in
\cite{Dosi242}. Since Putinar spectrum plays a central role in noncommutative
functional calculus problem and the spectral mapping property \cite{DosJOT10},
we are expecting some key properties for the $q$-closure of $\sigma\left(
T,S\right)  $ rather than itself. Certainly, this serious gap between spectra
is mainly caused by the weak, non-Hausdorff $q$-topology of $\mathbb{C}_{xy}$.
This is a new phenomenon of noncommutative analytic geometry, which makes the
$q$-theory more attractive and interesting.

\section{Appendix: The diagonal cohomology groups\label{sectionDCG}}

In this section we provide the details of the diagonal Fr\'{e}chet space
complexes occurred in the localization problem for $\mathcal{F}_{q}\left(
U\right)  $ and their cohomology groups. For brevity, we say that they are
\textit{the diagonal cohomology groups.}

\subsection{The operators on the compatible quadruples\label{subsecQRP}}

As above, let us take two copies $U_{1}=U_{z_{1}}\cup U_{w_{1}}$ and
$U_{2}=U_{z_{2}}\cup U_{w_{2}}$ of the same $q$-open subset $U\subseteq
\mathbb{C}_{xy}$. For brevity, we use the notations $\mathcal{O}\left(
z_{i}\right)  $, $\mathcal{O}\left(  w_{j}\right)  $ instead of $\mathcal{O}%
\left(  U_{z_{i}}\right)  $, $\mathcal{O}\left(  U_{w_{j}}\right)  $,
respectively. Moreover, $\mathcal{O}\left(  z_{i},w_{j}\right)  $ (with their
possible combinations) denotes $\mathcal{O}\left(  U_{z_{i}}\times U_{w_{j}%
}\right)  $ (or the related presheaf). As above we use the notation
$\mathcal{O}_{d,l}$ instead of $\mathcal{O}\widehat{\otimes}\mathcal{O}$ to
indicate its $\left(  d,l\right)  $-isomorphic copy.

\begin{lemma}
\label{lemTech2}The projective tensor product $\mathcal{O}\widehat{\otimes
}\mathcal{O}$ is identified with the fibered product
\[
\left(  \mathcal{O}\left(  z_{1},z_{2}\right)  \underset{\mathcal{O}\left(
z_{1}\right)  }{\times}\mathcal{O}\left(  z_{1},w_{2}\right)  \right)
\underset{\mathcal{O}\left(  z_{2},w_{2}\right)  }{\times}\left(
\mathcal{O}\left(  w_{1},z_{2}\right)  \underset{\mathcal{O}\left(
w_{1}\right)  }{\times}\mathcal{O}\left(  w_{1},w_{2}\right)  \right)
\]
up to the canonical topological isomorphism of the Fr\'{e}chet spaces (or
presheaves). In particular, every $\mathcal{O}_{d,l}$ consists of the
following compatible quadruples $\zeta=\zeta\left(  z_{1},w_{1};z_{2}%
,w_{2}\right)  $ such that
\[
\zeta=\left(  \left(  \zeta_{1}\left(  z_{1},z_{2}\right)  ,\zeta_{2}\left(
z_{1},w_{2}\right)  \right)  ,\left(  \zeta_{3}\left(  w_{1},z_{2}\right)
,\zeta_{4}\left(  w_{1},w_{2}\right)  \right)  \right)  ,
\]%
\begin{align*}
\zeta_{1}\left(  z_{1},0\right)   &  =\zeta_{2}\left(  z_{1},0\right)
,\quad\zeta_{3}\left(  w_{1},0\right)  =\zeta_{4}\left(  w_{1},0\right)  ,\\
\zeta_{1}\left(  0,z_{2}\right)   &  =\zeta_{3}\left(  0,z_{2}\right)
,\quad\zeta_{2}\left(  0,w_{2}\right)  =\zeta_{4}\left(  0,w_{2}\right)  .
\end{align*}

\end{lemma}

\begin{proof}
Taking into account that $\mathcal{O}\left(  U_{i}\right)  =\mathcal{O}\left(
z_{i}\right)  \underset{\mathbb{C}}{\times}\mathcal{O}\left(  w_{i}\right)  $,
$i=1,2$, we deduce by using \cite[Proposition 3.2]{Dosi25} that
\begin{align*}
\mathcal{O}\left(  U_{1}\right)  \widehat{\otimes}\mathcal{O}\left(
U_{2}\right)   &  =\left(  \mathcal{O}\left(  z_{1}\right)
\underset{\mathbb{C}}{\times}\mathcal{O}\left(  w_{1}\right)  \right)
\widehat{\otimes}\mathcal{O}\left(  U_{2}\right)  =\mathcal{O}\left(
z_{1}\right)  \widehat{\otimes}\mathcal{O}\left(  U_{2}\right)
\underset{\mathcal{O}\left(  U_{2}\right)  }{\times}\mathcal{O}\left(
w_{1}\right)  \widehat{\otimes}\mathcal{O}\left(  U_{2}\right) \\
&  =\left(  \mathcal{O}\left(  z_{1},z_{2}\right)  \underset{\mathcal{O}%
\left(  z_{1}\right)  }{\times}\mathcal{O}\left(  z_{1},w_{2}\right)  \right)
\underset{\mathcal{O}\left(  U_{2}\right)  }{\times}\left(  \mathcal{O}\left(
w_{1},z_{2}\right)  \underset{\mathcal{O}\left(  w_{1}\right)  }{\times
}\mathcal{O}\left(  w_{1},w_{2}\right)  \right)
\end{align*}
up to the canonical topological isomorphisms of the Fr\'{e}chet spaces. Thus
every $\zeta\in\mathcal{O}_{d,l}$ has the following components
\[
\zeta=\left(  \left(  \zeta_{1},\zeta_{2}\right)  ,\left(  \zeta_{3},\zeta
_{4}\right)  \right)  \text{, \quad}\zeta_{1}=\zeta_{1}\left(  z_{1}%
,z_{2}\right)  ,\text{ }\zeta_{2}=\zeta_{2}\left(  z_{1},w_{2}\right)  \text{,
}\zeta_{3}=\zeta_{3}\left(  w_{1},z_{2}\right)  \text{, }\zeta_{4}=\zeta
_{4}\left(  w_{1},w_{2}\right)  .
\]
In this case, $\left(  \zeta_{1},\zeta_{2}\right)  $ and $\left(  \zeta
_{3},\zeta_{4}\right)  $ are compatible elements of the fibered products over
$\mathcal{O}\left(  z_{1}\right)  $ and $\mathcal{O}\left(  w_{1}\right)  $,
respectively. Therefore
\[
\zeta_{1}\left(  z_{1},0\right)  =\zeta_{2}\left(  z_{1},0\right)  \text{ in
}\mathcal{O}\left(  z_{1}\right)  ,\quad\zeta_{3}\left(  w_{1},0\right)
=\zeta_{4}\left(  w_{1},0\right)  \text{ in }\mathcal{O}\left(  w_{1}\right)
.
\]
Moreover, the couples $\left(  \zeta_{1},\zeta_{2}\right)  $ and $\left(
\zeta_{3},\zeta_{4}\right)  $ should be compatible over $\mathcal{O}\left(
U_{2}\right)  $, which means that $\left(  \zeta_{1}\left(  0,z_{2}\right)
,\zeta_{2}\left(  0,w_{2}\right)  \right)  =\left(  \zeta_{3}\left(
0,z_{2}\right)  ,\zeta_{4}\left(  0,w_{2}\right)  \right)  $ in $\mathcal{O}%
\left(  U_{2}\right)  $, that is,
\[
\zeta_{1}\left(  0,z_{2}\right)  =\zeta_{3}\left(  0,z_{2}\right)  ,\quad
\zeta_{2}\left(  0,w_{2}\right)  =\zeta_{4}\left(  0,w_{2}\right)  .
\]
Thus $\zeta$ does satisfy all indicated conditions.
\end{proof}

The multiplication operators by $z_{i}$, $w_{j}$ are lifted to $\mathcal{O}%
_{d,l}$ too, and we have $z_{1}=z\otimes1$, $z_{2}=1\otimes z$, $w_{1}%
=w\otimes1$ and $w_{2}=1\otimes w$ (see (\ref{RzRw1})).

\begin{corollary}
\label{corTech1}If $\zeta\in\mathcal{O}_{d,l}$ is a compatible quadruple then
\begin{align*}
z_{1}\left(  \zeta\right)   &  =\left(  \left(  z_{1}\zeta_{1}\left(
z_{1},z_{2}\right)  ,z_{1}\zeta_{2}\left(  z_{1},w_{2}\right)  \right)
,\left(  0,0\right)  \right)  ,\quad z_{2}\left(  \zeta\right)  =\left(
\left(  z_{2}\zeta_{1}\left(  z_{1},z_{2}\right)  ,0\right)  ,\left(
z_{2}\zeta_{3}\left(  w_{1},z_{2}\right)  ,0\right)  \right)  ,\\
w_{1}\left(  \zeta\right)   &  =\left(  \left(  0,0\right)  ,\left(
w_{1}\zeta_{3}\left(  w_{1},z_{2}\right)  ,w_{1}\zeta_{4}\left(  w_{1}%
,w_{2}\right)  \right)  \right)  ,\quad w_{2}\left(  \zeta\right)  =\left(
\left(  0,w_{2}\zeta_{2}\left(  z_{1},w_{2}\right)  \right)  ,\left(
0,w_{2}\zeta_{4}\left(  w_{1},w_{2}\right)  \right)  \right)  .
\end{align*}

\end{corollary}

\begin{proof}
One needs to use (\ref{RzRw1}) and Lemma \ref{lemTech2}, and verify the
equalities over the elementary tensors $\zeta=\alpha_{1}\otimes\alpha_{2}$
(see below to the proof of Lemma \ref{lemTech3}).
\end{proof}

The operators (\ref{Nd}) act on $\mathcal{O}_{d,l}$ too. Namely, we only
consider \ the following operators
\begin{align}
N_{x_{1},d}  &  =N_{x,d}\otimes1,\quad N_{y_{1},d}=N_{y,d}\otimes
1:\mathcal{O}_{d,l}\rightarrow\mathcal{O}_{d+1,l},\label{MTN}\\
M_{x_{2},l}  &  =1\otimes M_{x,l},\quad M_{y_{2},l}=1\otimes M_{y,l}%
:\mathcal{O}_{d,l}\rightarrow\mathcal{O}_{d,l+1},\nonumber
\end{align}
where $l,d\in\mathbb{Z}_{+}$.

\begin{lemma}
\label{lemTech3}If $\zeta\in\mathcal{O}_{d,l}$ is a compatible quadruple, then
the following identities
\begin{align*}
M_{x_{2},l}\left(  \zeta\right)   &  =\left(  \left(  \frac{\partial\zeta_{2}%
}{\partial w_{2}}\left(  z_{1},0\right)  ,\left(  \zeta_{2}\right)  _{w_{2}%
}\left(  z_{1},w_{2}\right)  \right)  ,\left(  \frac{\partial\zeta_{4}%
}{\partial w_{2}}\left(  w_{1},0\right)  ,\left(  \zeta_{4}\right)  _{w_{2}%
}\left(  w_{1},w_{2}\right)  \right)  \right)  ,\\
M_{y_{2},l}\left(  \zeta\right)   &  =q^{l+1}\left(  \left(  \left(  \zeta
_{1}\right)  _{z_{2}}\left(  z_{1},qz_{2}\right)  ,\frac{\partial\zeta_{1}%
}{\partial z_{2}}\left(  z_{1},0\right)  \right)  ,\left(  \left(  \zeta
_{3}\right)  _{z_{2}}\left(  w_{1},qz_{2}\right)  ,\frac{\partial\zeta_{3}%
}{\partial z_{2}}\left(  w_{1},0\right)  \right)  \right)  ,\\
N_{x_{1},d}\left(  \zeta\right)   &  =\left(  \left(  \left(  \zeta
_{1}\right)  _{z_{1}}\left(  z_{1},z_{2}\right)  ,\left(  \zeta_{2}\right)
_{z_{1}}\left(  z_{1},w_{2}\right)  \right)  ,\left(  \frac{\partial\zeta_{1}%
}{\partial z_{1}}\left(  0,z_{2}\right)  ,\frac{\partial\zeta_{2}}{\partial
z_{1}}\left(  0,w_{2}\right)  \right)  \right)  ,\\
N_{y_{1},d}\left(  \zeta\right)   &  =q^{d+1}\left(  \left(  \frac
{\partial\zeta_{3}}{\partial w_{1}}\left(  0,z_{2}\right)  ,\frac
{\partial\zeta_{4}}{\partial w_{1}}\left(  0,w_{2}\right)  \right)  ,\left(
\left(  \zeta_{3}\right)  _{w_{1}}\left(  qw_{1},z_{2}\right)  ,\left(
\zeta_{4}\right)  _{w_{1}}\left(  qw_{1},w_{2}\right)  \right)  \right)
\end{align*}
hold.
\end{lemma}

\begin{proof}
One suffices to prove the equality over elementary tensors $\zeta=\alpha
_{1}\otimes\alpha_{2}\in\mathcal{O}\left(  U_{1}\right)  \widehat{\otimes
}\mathcal{O}\left(  U_{2}\right)  $ corresponding to $\alpha_{i}=\left(
f_{i}\left(  z_{i}\right)  ,g_{i}\left(  w_{i}\right)  \right)  $, $i=1,2$. By
Lemma \ref{lemTech2}, $\zeta$ is identified with the quadruple such that
$\zeta_{1}=f_{1}\left(  z_{1}\right)  f_{2}\left(  z_{2}\right)  $, $\zeta
_{2}=f_{1}\left(  z_{1}\right)  g_{2}\left(  w_{2}\right)  $, $\zeta_{3}%
=g_{1}\left(  w_{1}\right)  f_{2}\left(  z_{2}\right)  $ and $\zeta_{4}%
=g_{1}\left(  w_{1}\right)  g_{2}\left(  w_{2}\right)  $. Using Lemma
\ref{lemTech1}, we derive that
\[
M_{x_{2},l}\left(  \zeta\right)  =\alpha_{1}\otimes M_{x,l}\left(  \alpha
_{2}\right)  =\alpha_{1}\otimes\left(  g_{2}^{\prime}\left(  0\right)
,\frac{g_{2}\left(  w_{2}\right)  -g_{2}\left(  0\right)  }{w_{2}}\right)
\]%
\begin{align*}
&  =\left(  \left(  f_{1}\left(  z_{1}\right)  g_{2}^{\prime}\left(  0\right)
,f_{1}\left(  z_{1}\right)  \frac{g_{2}\left(  w_{2}\right)  -g_{2}\left(
0\right)  }{w_{2}}\right)  ,\left(  g_{1}\left(  w_{1}\right)  g_{2}^{\prime
}\left(  0\right)  ,g_{1}\left(  w_{1}\right)  \frac{g_{2}\left(
w_{2}\right)  -g_{2}\left(  0\right)  }{w_{2}}\right)  \right) \\
&  =\left(  \left(  \frac{\partial\zeta_{2}}{\partial w_{2}}\left(
z_{1},0\right)  ,\frac{\zeta_{2}\left(  z_{1},w_{2}\right)  -\zeta_{2}\left(
z_{1},0\right)  }{w_{2}}\right)  ,\left(  \frac{\partial\zeta_{4}}{\partial
w_{2}}\left(  w_{1},0\right)  ,\frac{\zeta_{4}\left(  w_{1},w_{2}\right)
-\zeta_{4}\left(  w_{1},0\right)  }{w_{2}}\right)  \right)  .
\end{align*}
In the case of $M_{y_{2},l}$, we have%
\[
q^{-l-1}M_{y_{2},l}\left(  \zeta\right)  =q^{-l-1}\alpha_{1}\otimes
M_{y,l}\left(  \alpha_{2}\right)  =\alpha_{1}\otimes\left(  \frac{f_{2}\left(
qz_{2}\right)  -f_{2}\left(  0\right)  }{qz_{2}},f_{2}^{\prime}\left(
0\right)  \right)
\]%
\begin{align*}
&  =\left(  \left(  f_{1}\left(  z_{1}\right)  \frac{f_{2}\left(
qz_{2}\right)  -f_{2}\left(  0\right)  }{qz_{2}},f_{1}\left(  z_{1}\right)
f_{2}^{\prime}\left(  0\right)  \right)  ,\left(  g_{1}\left(  w_{1}\right)
\frac{f_{2}\left(  qz_{2}\right)  -f_{2}\left(  0\right)  }{qz_{2}}%
,g_{1}\left(  w_{1}\right)  f_{2}^{\prime}\left(  0\right)  \right)  \right)
\\
&  =\left(  \left(  \frac{\zeta_{1}\left(  z_{1},qz_{2}\right)  -\zeta
_{1}\left(  z_{1},0\right)  }{qz_{2}},\frac{\partial\zeta_{1}}{\partial z_{2}%
}\left(  z_{1},0\right)  \right)  ,\left(  \frac{\zeta_{3}\left(  w_{1}%
,qz_{2}\right)  -\zeta_{3}\left(  w_{1},0\right)  }{qz_{2}},\frac
{\partial\zeta_{3}}{\partial z_{2}}\left(  w_{1},0\right)  \right)  \right)  .
\end{align*}
Now consider the operators $N_{x_{1},d}$ and $N_{y_{1},d}$. As above, we have%
\[
N_{x_{1},d}\left(  \zeta\right)  =N_{x,d}\left(  \alpha_{1}\right)
\otimes\alpha_{2}=\left(  \frac{f_{1}\left(  z_{1}\right)  -f_{1}\left(
0\right)  }{z_{1}},f_{1}^{\prime}\left(  0\right)  \right)  \otimes\alpha_{2}%
\]%
\begin{align*}
&  =\left(  \left(  \frac{f_{1}\left(  z_{1}\right)  -f_{1}\left(  0\right)
}{z_{1}}f_{2}\left(  z_{2}\right)  ,\frac{f_{1}\left(  z_{1}\right)
-f_{1}\left(  0\right)  }{z_{1}}g_{2}\left(  w_{2}\right)  \right)  ,\left(
f_{1}^{\prime}\left(  0\right)  f_{2}\left(  z_{2}\right)  ,f_{1}^{\prime
}\left(  0\right)  g_{2}\left(  w_{2}\right)  \right)  \right) \\
&  =\left(  \left(  \frac{\zeta_{1}\left(  z_{1},z_{2}\right)  -\zeta
_{1}\left(  0,z_{2}\right)  }{z_{1}},\frac{\zeta_{2}\left(  z_{1}%
,w_{2}\right)  -\zeta_{2}\left(  0,w_{2}\right)  }{z_{1}}\right)  ,\left(
\frac{\partial\zeta_{1}}{\partial z_{1}}\left(  0,z_{2}\right)  ,\frac
{\partial\zeta_{2}}{\partial z_{1}}\left(  0,w_{2}\right)  \right)  \right)  .
\end{align*}
Finally, we derive that%
\[
q^{-d-1}N_{y_{1},d}\left(  \zeta\right)  =q^{-d-1}N_{y,d}\left(  \alpha
_{1}\right)  \otimes\alpha_{2}=\left(  g_{1}^{\prime}\left(  0\right)
,\frac{g_{1}\left(  qw_{1}\right)  -g_{1}\left(  0\right)  }{qw_{1}}\right)
\otimes\alpha_{2}%
\]%
\begin{align*}
&  =\left(  \left(  g_{1}^{\prime}\left(  0\right)  f_{2}\left(  z_{2}\right)
,g_{1}^{\prime}\left(  0\right)  g_{2}\left(  w_{2}\right)  \right)  ,\left(
\frac{g_{1}\left(  qw_{1}\right)  -g_{1}\left(  0\right)  }{qw_{1}}%
f_{2}\left(  z_{2}\right)  ,\frac{g_{1}\left(  qw_{1}\right)  -g_{1}\left(
0\right)  }{qw_{1}}g_{2}\left(  w_{2}\right)  \right)  \right) \\
&  =\left(  \left(  \frac{\partial\zeta_{3}}{\partial w_{1}}\left(
0,z_{2}\right)  ,\frac{\partial\zeta_{4}}{\partial w_{1}}\left(
0,w_{2}\right)  \right)  ,\left(  \frac{\zeta_{3}\left(  qw_{1},z_{2}\right)
-\zeta_{3}\left(  0,z_{2}\right)  }{qw_{1}},\frac{\zeta_{4}\left(
qw_{1},w_{2}\right)  -\zeta_{4}\left(  0,w_{2}\right)  }{qw_{1}}\right)
\right)  .
\end{align*}
Thus all required identities hold.
\end{proof}

Let us also consider the following continuous linear mapping
\[
\pi_{d,l}:\mathcal{O}\left(  d,l\right)  \longrightarrow\mathcal{O}\left(
d+l\right)  ,\quad\pi_{d,l}\left(  \zeta\otimes\eta\right)  =\left(  f\left(
z\right)  u\left(  q^{d}z\right)  ,g\left(  q^{l}w\right)  v\left(  w\right)
\right)  ,
\]
whenever $\zeta=\left(  f,g\right)  \in\mathcal{O}\left(  d\right)  $ and
$\eta=\left(  u,v\right)  \in\mathcal{O}\left(  l\right)  $. Notice that
\begin{align*}
\frac{1}{\left(  d+l\right)  !}\frac{d^{d+l}f\left(  z\right)  u\left(
q^{d}z\right)  }{dz^{l+d}}|_{z=0}  &  =q^{dl}\frac{f^{\left(  d\right)
}\left(  0\right)  }{d!}\frac{u^{\left(  l\right)  }\left(  0\right)  }%
{l!}=q^{ld}\frac{g^{\left(  d\right)  }\left(  0\right)  }{d!}\frac{v^{\left(
l\right)  }\left(  0\right)  }{l!}\\
&  =\frac{1}{\left(  d+l\right)  !}\frac{d^{l+d}g\left(  q^{l}w\right)
v\left(  w\right)  }{dw^{l+d}}|_{w=0},
\end{align*}
which means that $\pi_{d,l}$ is a well defined mapping.

\begin{lemma}
\label{lemPidl}The mapping $\pi_{d,l}$ is reduced to the following continuous
linear mapping
\[
\pi_{d,l}:\mathcal{O}_{d,l}\longrightarrow\mathcal{O}_{d+l},\quad\pi
_{d,l}\left(  h\left(  z_{1},w_{1};z_{2},w_{2}\right)  \right)  =q^{dl}\left(
h_{1}\left(  z_{1},q^{d}z_{1}\right)  ,h_{4}\left(  q^{l}w_{2},w_{2}\right)
\right)
\]
up to the topological ($\left(  d,l\right)  $ and $d+l$) isomorphisms
(\ref{imoo}).
\end{lemma}

\begin{proof}
One needs to prove the statement over elementary tensors $h=\zeta\otimes\eta$
with $\zeta=\left(  f,g\right)  \in\mathcal{O}_{d}$ and $\eta=\left(
u,v\right)  \in\mathcal{O}_{l}$. As above, we have $h_{1}=f\left(
z_{1}\right)  u\left(  z_{2}\right)  $, $h_{2}=f\left(  z_{1}\right)  v\left(
w_{2}\right)  $, $h_{3}=g\left(  w_{1}\right)  u\left(  z_{2}\right)  $ and
$h_{4}=g\left(  w_{1}\right)  v\left(  w_{2}\right)  $. If $t_{d,l}=\left(
z^{-d}\times_{1}w^{-d}\right)  \otimes\left(  z^{-l}\times_{1}w^{-l}\right)  $
is the $\left(  d,l\right)  $-isomorphism $\mathcal{O}\left(  d,l\right)
\rightarrow\mathcal{O}_{d,l}$, then
\begin{align*}
\pi_{d,l}t_{d,l}^{-1}\left(  h\right)   &  =\pi_{d,l}\left(  \left(
z^{d}f\left(  z\right)  ,w^{d}g\left(  w\right)  \right)  \otimes\left(
z^{l}u\left(  z\right)  ,w^{l}v\left(  w\right)  \right)  \right) \\
&  =q^{dl}\left(  z^{d+l}f\left(  z\right)  u\left(  q^{d}z\right)
,w^{d+l}g\left(  q^{l}w\right)  v\left(  w\right)  \right)  .
\end{align*}
It follows that
\[
\left(  z^{-d-l}\times_{1}w^{-d-l}\right)  \pi_{d,l}t_{d,l}^{-1}\left(
h\right)  =q^{dl}\left(  f\left(  z\right)  u\left(  q^{d}z\right)  ,g\left(
q^{l}w\right)  v\left(  w\right)  \right)  =q^{dl}\left(  h_{1}\left(
z,q^{d}z\right)  ,h_{4}\left(  q^{l}w,w\right)  \right)  .
\]
The rest is clear.
\end{proof}

Thus the coefficient $q^{dl}$ can be ignored and we can assume that $\pi
_{d,l}$ has the action $h\mapsto h\left(  z,q^{l}w;q^{d}z,w\right)  $ over the
quadruples $h$.

\subsection{The diagonal Fr\'{e}chet space complexes\label{subsecDFC1}}

As above we fix a $q$-open subset $U\subseteq\mathbb{C}_{xy}$, and use the
$\left(  d,l\right)  $-copies $\mathcal{O}_{d,l}$ of $\mathcal{O}\left(
d,l\right)  $, $d,l\in\mathbb{Z}_{+}$. Based on the multiplication operators
$z_{i}$, $w_{j}$ acting on $\mathcal{O}_{d,l}$, we define the following
cochain complex
\begin{equation}
\mathcal{O}_{q}\left(  d,l\right)  :0\rightarrow\mathcal{O}_{d,l}%
\overset{\partial_{d,l}^{0}}{\longrightarrow}%
\begin{array}
[c]{c}%
\mathcal{O}_{d,l}\\
\oplus\\
\mathcal{O}_{d,l}%
\end{array}
\overset{\partial_{d,l}^{1}}{\longrightarrow}\mathcal{O}_{d,l}\overset{\pi
_{d,l}}{\longrightarrow}\mathcal{O}_{d+l}\rightarrow0, \label{OqD}%
\end{equation}
whose differentials are acting in the following ways
\[
\partial_{d,l}^{0}=\left[
\begin{array}
[c]{c}%
w_{1}-q^{l+1}w_{2}\\
z_{2}-q^{d+1}z_{1}%
\end{array}
\right]  ,\quad\partial_{d,l}^{1}=\left[
\begin{array}
[c]{cc}%
z_{2}-q^{d}z_{1} & q^{l}w_{2}-w_{1}%
\end{array}
\right]  ,
\]%
\[
\pi_{d,l}\left(  h\left(  z_{1},w_{1};z_{2},w_{2}\right)  \right)  =h\left(
z,q^{l}w;q^{d}z,w\right)  .
\]
One can easily verify that $\partial_{d,l}^{1}\partial_{d,l}^{0}=0$ and
$\pi_{d,l}\partial_{d,l}^{1}=0$. Notice that
\[
\partial_{d,l}^{0}=\left[
\begin{array}
[c]{c}%
w\otimes1-q^{l+1}\otimes w\\
1\otimes z-z\otimes q^{d+1}%
\end{array}
\right]  ,\quad\partial_{d,l}^{1}=\left[
\begin{array}
[c]{cc}%
1\otimes z-z\otimes q^{d} & q^{l}\otimes w-w\otimes1
\end{array}
\right]
\]
whereas $\pi_{d,l}\left(  h\left(  z_{1},w_{1};z_{2},w_{2}\right)  \right)
=\left(  h_{1}\left(  z_{1},q^{d}z_{1}\right)  ,h_{4}\left(  q^{l}w_{2}%
,w_{2}\right)  \right)  $. The cohomology groups of $\mathcal{O}_{q}\left(
d,l\right)  $ are denoted by $H_{d,l}^{i}$, $0\leq i\leq3$.

\subsection{The operator $T$ on the free quadruples\label{subsecOT}}

Let us introduce the free quadruples to be the elements of the following
Fr\'{e}chet space
\[
\mathcal{Z}\left(  U\right)  =\mathcal{O}\left(  z_{1},z_{2}\right)
\times\mathcal{O}\left(  z_{1},w_{2}\right)  \times\mathcal{O}\left(
w_{1},z_{2}\right)  \times\mathcal{O}\left(  w_{1},w_{2}\right)  ,
\]
which contains $\mathcal{O}_{d,l}$ as a closed subspace (see Lemma
\ref{lemTech2}). The following continuous linear operator
\[
T:\mathcal{Z}\left(  U\right)  \longrightarrow%
\begin{array}
[c]{c}%
\mathcal{O}_{d,l}\\
\oplus\\
\mathcal{O}_{d,l}%
\end{array}
,\quad T\left(  \theta\right)  =\left[
\begin{array}
[c]{c}%
\zeta\\
\eta
\end{array}
\right]  =\left[
\begin{array}
[c]{c}%
\left(  \left(  0,-q^{l+1}w_{2}\theta_{2}\left(  z_{1},w_{2}\right)  \right)
,\left(  w_{1}\theta_{3}\left(  w_{1},z_{2}\right)  ,\zeta_{4}\right)  \right)
\\
\left(  \left(  \eta_{1},-q^{d+1}z_{1}\theta_{2}\left(  z_{1},w_{2}\right)
\right)  ,\left(  z_{2}\theta_{3}\left(  w_{1},z_{2}\right)  ,0\right)
\right)
\end{array}
\right]
\]
plays a key role in the description of the cohomology groups $H_{d,l}^{i}$,
where $\zeta_{4}$ and $\eta_{1}$ are obtained by
\begin{align*}
\zeta_{4}\left(  w_{1},w_{2}\right)   &  =w_{1}\theta_{3}\left(
w_{1},0\right)  -q^{l+1}w_{2}\theta_{2}\left(  0,w_{2}\right)  +w_{1}%
w_{2}\theta_{4}\left(  w_{1},w_{2}\right)  ,\\
\eta_{1}\left(  z_{1},z_{2}\right)   &  =-q^{d+1}z_{1}\theta_{2}\left(
z_{1},0\right)  +z_{2}\theta_{3}\left(  0,z_{2}\right)  +z_{1}z_{2}\theta
_{1}\left(  z_{1},z_{2}\right)  .
\end{align*}
By Lemma \ref{lemTech2}, one can easily verify that $\zeta$ and $\eta$ are
compatible quadruples from $\mathcal{O}_{d,l}$. Indeed, we have
\begin{align*}
\zeta_{1}\left(  z_{1},0\right)   &  =0=\zeta_{2}\left(  z_{1},0\right)
,\quad\zeta_{3}\left(  w_{1},0\right)  =w_{1}\theta_{3}\left(  w_{1},0\right)
=\zeta_{4}\left(  w_{1},0\right)  \text{, }\\
\zeta_{1}\left(  0,z_{2}\right)   &  =0=\zeta_{3}\left(  0,z_{2}\right)
\text{,\quad}\zeta_{2}\left(  0,w_{2}\right)  =-q^{l+1}w_{2}\theta_{2}\left(
0,w_{2}\right)  =\zeta_{4}\left(  0,w_{2}\right)  ,
\end{align*}%
\begin{align*}
\eta_{1}\left(  z_{1},0\right)   &  =-q^{d+1}z_{1}\theta_{2}\left(
z_{1},0\right)  =\eta_{2}\left(  z_{1},0\right)  ,\quad\eta_{3}\left(
w_{1},0\right)  =0=\eta_{4}\left(  w_{1},0\right)  ,\\
\eta_{1}\left(  0,z_{2}\right)   &  =z_{2}\theta_{3}\left(  0,z_{2}\right)
=\eta_{3}\left(  0,z_{2}\right)  ,\quad\eta_{2}\left(  0,w_{2}\right)
=0=\eta_{4}\left(  0,w_{2}\right)  .
\end{align*}
Thus $T$ is a well defined continuous linear operator.

\subsection{The diagonal cohomology group $H_{d,l}^{1}$\label{subsecDC}}

It turns out that the image of $T$ provides the kernel of the differential
$\partial_{d,l}^{1}$ of the complex $\mathcal{O}_{q}\left(  d,l\right)  $.

\begin{lemma}
\label{lemTech4}The range of the operator $T$ equals to $\ker\left(
\partial_{d,l}^{1}\right)  $, and $T$ implements a topological isomorphism
$\mathcal{Z}\left(  U\right)  \overset{\simeq}{\longrightarrow}\ker\left(
\partial_{d,l}^{1}\right)  $ of the Fr\'{e}chet spaces.
\end{lemma}

\begin{proof}
Take an element $\left[
\begin{array}
[c]{c}%
\zeta\\
\eta
\end{array}
\right]  =T\left(  \theta\right)  \in\operatorname{im}\left(  T\right)  $ with
$\theta\in\mathcal{Z}\left(  U\right)  $, and put $\vartheta=\partial
_{d,l}^{1}\left(  T\left(  \theta\right)  \right)  $. Using Corollary
\ref{corTech1}, and the fact that $\zeta_{1}=0$, $\eta_{4}=0$, we deduce that
\begin{align*}
\vartheta &  =\partial_{d,l}^{1}\left[
\begin{array}
[c]{c}%
\zeta\\
\eta
\end{array}
\right]  =\left(  z_{2}-q^{d}z_{1}\right)  \zeta+\left(  q^{l}w_{2}%
-w_{1}\right)  \eta\\
&  =\left(  \left(  0,-q^{d}z_{1}\zeta_{2}\left(  z_{1},w_{2}\right)  \right)
,\left(  z_{2}\zeta_{3}\left(  w_{1},z_{2}\right)  ,0\right)  \right)
+\left(  \left(  0,q^{l}w_{2}\eta_{2}\left(  z_{1},w_{2}\right)  \right)
,\left(  -w_{1}\eta_{3}\left(  w_{1},z_{2}\right)  ,0\right)  \right) \\
&  =\left(  \left(  0,q^{d+l+1}z_{1}w_{2}\theta_{2}\left(  z_{1},w_{2}\right)
\right)  ,\left(  z_{2}w_{1}\theta_{3}\left(  w_{1},z_{2}\right)  ,0\right)
\right) \\
&  +\left(  \left(  0,-q^{l+d+1}w_{2}z_{1}\theta_{2}\left(  z_{1}%
,w_{2}\right)  \right)  ,\left(  -w_{1}z_{2}\theta_{3}\left(  w_{1}%
,z_{2}\right)  ,0\right)  \right)  =0,
\end{align*}
that is, $T\left(  \theta\right)  \in\ker\left(  \partial_{d,l}^{1}\right)  $.
Conversely, take $\omega=\left[
\begin{array}
[c]{c}%
\zeta\\
\eta
\end{array}
\right]  \in\ker\left(  \partial_{d,l}^{1}\right)  $. Then $\left(
z_{2}-q^{d}z_{1}\right)  \zeta=\left(  w_{1}-q^{l}w_{2}\right)  \eta$, which
means that (see Corollary \ref{corTech1})%
\begin{align*}
\left(  z_{2}-q^{d}z_{1}\right)  \zeta_{1}\left(  z_{1},z_{2}\right)   &
=\left(  w_{1}-q^{l}w_{2}\right)  \eta_{1}\left(  z_{1},z_{2}\right)  =0,\\
-q^{d}z_{1}\zeta_{2}\left(  z_{1},w_{2}\right)   &  =\left(  z_{2}-q^{d}%
z_{1}\right)  \zeta_{2}\left(  z_{1},w_{2}\right)  =\left(  w_{1}-q^{l}%
w_{2}\right)  \eta_{2}\left(  z_{1},w_{2}\right)  =-q^{l}w_{2}\eta_{2}\left(
z_{1},w_{2}\right)  ,\\
z_{2}\zeta_{3}\left(  w_{1},z_{2}\right)   &  =\left(  z_{2}-q^{d}%
z_{1}\right)  \zeta_{3}\left(  w_{1},z_{2}\right)  =\left(  w_{1}-q^{l}%
w_{2}\right)  \eta_{3}\left(  w_{1},z_{2}\right)  =w_{1}\eta_{3}\left(
w_{1},z_{2}\right)  ,\\
0  &  =\left(  z_{2}-q^{d}z_{1}\right)  \zeta_{4}\left(  w_{1},w_{2}\right)
=\left(  w_{1}-q^{l}w_{2}\right)  \eta_{4}\left(  w_{1},w_{2}\right)  .
\end{align*}
It follows that (removable singularities)
\[
\zeta_{1}\left(  z_{1},z_{2}\right)  =0,\quad\eta_{4}\left(  w_{1}%
,w_{2}\right)  =0,\quad\zeta_{2}\left(  z_{1},w_{2}\right)  =-q^{l+1}%
w_{2}\theta_{2}\left(  z_{1},w_{2}\right)  ,\quad\zeta_{3}\left(  w_{1}%
,z_{2}\right)  =w_{1}\theta_{3}\left(  w_{1},z_{2}\right)
\]
for some $\theta_{2}\in\mathcal{O}\left(  z_{1},w_{2}\right)  $ and
$\theta_{3}\in\mathcal{O}\left(  w_{1},z_{2}\right)  $ (optionally one can use
Analytic Nullstellensatz, see below to the proof of Lemma \ref{lemTech5}). In
particular, $\eta_{2}\left(  z_{1},w_{2}\right)  =-q^{d+1}z_{1}\theta
_{2}\left(  z_{1},w_{2}\right)  $ and $\eta_{3}\left(  w_{1},z_{2}\right)
=z_{2}\theta_{3}\left(  w_{1},z_{2}\right)  $. Further, using Lemma
\ref{lemTech2}, we derive that $\zeta_{4}\left(  w_{1},0\right)  =\zeta
_{3}\left(  w_{1},0\right)  =w_{1}\theta_{3}\left(  w_{1},0\right)  $,
$\zeta_{4}\left(  0,w_{2}\right)  =\zeta_{2}\left(  0,w_{2}\right)
=-q^{l+1}w_{2}\theta_{2}\left(  0,w_{2}\right)  $. In particular, $\zeta
_{4}\left(  0,0\right)  =0$ and
\[
\zeta_{4}\left(  w_{1},w_{2}\right)  =w_{1}\theta_{3}\left(  w_{1},0\right)
-q^{l+1}w_{2}\theta_{2}\left(  0,w_{2}\right)  +w_{1}w_{2}\theta_{4}\left(
w_{1},w_{2}\right)
\]
for some $\theta_{4}\in\mathcal{O}\left(  w_{1},w_{2}\right)  $ (see below
Remark \ref{remTech1}). In a similar way, we obtain that
\begin{align*}
\eta_{1}\left(  z_{1},z_{2}\right)   &  =\eta_{1}\left(  z_{1},0\right)
+\eta_{1}\left(  0,z_{2}\right)  +z_{1}z_{2}\theta_{1}\left(  z_{1}%
,z_{2}\right)  =\eta_{2}\left(  z_{1},0\right)  +\eta_{3}\left(
0,z_{2}\right)  +z_{1}z_{2}\theta_{1}\left(  z_{1},z_{2}\right) \\
&  =-q^{d+1}z_{1}\theta_{2}\left(  z_{1},0\right)  +z_{2}\theta_{3}\left(
0,z_{2}\right)  +z_{1}z_{2}\theta_{1}\left(  z_{1},z_{2}\right)
\end{align*}
for some $\theta_{1}\in\mathcal{O}\left(  z_{1},z_{2}\right)  $. If
$\theta=\left(  \theta_{1},\theta_{2},\theta_{3},\theta_{4}\right)
\in\mathcal{Z}\left(  U\right)  $ is the free quadruple obtained, then
\[
T\left(  \theta\right)  =\left[
\begin{array}
[c]{c}%
\left(  \left(  0,-q^{l+1}w_{2}\theta_{2}\left(  z_{1},w_{2}\right)  \right)
,\left(  w_{1}\theta_{3}\left(  w_{1},z_{2}\right)  ,\zeta_{4}\right)  \right)
\\
\left(  \left(  \eta_{1},-q^{d+1}z_{1}\theta_{2}\left(  z_{1},w_{2}\right)
\right)  ,\left(  z_{2}\theta_{3}\left(  w_{1},z_{2}\right)  ,0\right)
\right)
\end{array}
\right]  =\left[
\begin{array}
[c]{c}%
\zeta\\
\eta
\end{array}
\right]  .
\]
Thus $\operatorname{im}\left(  T\right)  =\ker\left(  \partial_{d,l}%
^{1}\right)  $, and in particular, $\operatorname{im}\left(  T\right)  $ is
closed. One can easily see that $\ker\left(  T\right)  =\left\{  0\right\}  $.
By Open Mapping Theorem, we conclude that $T$ implements a topological
isomorphism of $\mathcal{Z}\left(  U\right)  $ onto $\ker\left(
\partial_{d,l}^{1}\right)  $.
\end{proof}

\begin{remark}
\label{remTech1}If $\zeta\left(  z,w\right)  $ is a holomorphic function
defined on a polydomain containing the origin, and $\zeta\left(  0,0\right)
=0$, then
\begin{align*}
\zeta\left(  z,w\right)   &  =\zeta_{z}\left(  z,w\right)  z+\zeta\left(
0,w\right)  =\zeta_{z}\left(  z,0\right)  z+\left(  \zeta_{z}\right)
_{w}\left(  z,w\right)  wz+\zeta\left(  0,w\right) \\
&  =z\frac{\zeta\left(  z,0\right)  -\zeta\left(  0,0\right)  }{z}%
+\zeta\left(  0,w\right)  +zw\left(  \zeta_{z}\right)  _{w}\left(  z,w\right)
=\zeta\left(  z,0\right)  +\zeta\left(  0,w\right)  +zw\zeta_{zw}\left(
z,w\right)  ,
\end{align*}
that is, $\zeta\left(  z,w\right)  =\zeta\left(  z,0\right)  +\zeta\left(
0,w\right)  +zw\zeta_{zw}\left(  z,w\right)  $.
\end{remark}

For the description of $\operatorname{im}\left(  \partial_{d,l}^{0}\right)  $
in terms of the quadruples from Lemma \ref{lemTech4}, we introduce the
following subspace $\mathcal{I}\left(  U\right)  \subseteq\mathcal{Z}\left(
U\right)  $ for those (free) quadruples $\theta$ such that $\theta_{2}\left(
0,0\right)  =\theta_{3}\left(  0,0\right)  $ and
\begin{align}
\theta_{1}\left(  z_{1},q^{d+1}z_{1}\right)   &  =\frac{\theta_{2}\left(
z_{1},0\right)  -\theta_{3}\left(  0,q^{d+1}z_{1}\right)  }{z_{1}%
},\label{Theta}\\
\theta_{4}\left(  q^{l+1}w_{2},w_{2}\right)   &  =\frac{\theta_{2}\left(
0,w_{2}\right)  -\theta_{3}\left(  q^{l+1}w_{2},0\right)  }{w_{2}},\nonumber
\end{align}
Thus $\mathcal{I}\left(  U\right)  $ is identified with the fibered product
$\mathcal{O}\left(  z_{1},w_{2}\right)  \underset{\mathbb{C}}{\times
}\mathcal{O}\left(  w_{1},z_{2}\right)  $ up to a topological isomorphism.
Every couple $\left(  \theta_{2},\theta_{3}\right)  $ from the fibered product
defines uniquely the related quadruple $\theta$ (see to the proof of the
forthcoming lemma).

\begin{lemma}
\label{lemTech5}The continuous linear operator
\[
T^{-1}\partial_{d,l}^{0}:\mathcal{O}_{d,l}\longrightarrow\mathcal{O}\left(
z_{1},w_{2}\right)  \underset{\mathbb{C}}{\times}\mathcal{O}\left(
w_{1},z_{2}\right)
\]
is transforming a compatible quadruple $\alpha\in\mathcal{O}_{d,l}$ into a
couple $\left(  \theta_{2},\theta_{3}\right)  $ whose quadruple is given by
\[
\theta=\left(  \left(  \alpha_{1}\right)  _{z_{1}}\left(  z_{1},z_{2}\right)
-q^{d+1}\left(  \alpha_{1}\right)  _{z_{2}}\left(  z_{1},z_{2}\right)
,\alpha_{2},\alpha_{3},\left(  \alpha_{4}\right)  _{w_{2}}\left(  w_{1}%
,w_{2}\right)  -q^{l+1}\left(  \alpha_{4}\right)  _{w_{1}}\left(  w_{1}%
,w_{2}\right)  \right)  ,
\]
where $\theta_{2}=\alpha_{2}$ and $\theta_{3}=\alpha_{3}$. Moreover,
$T^{-1}\partial_{d,l}^{0}$ implements a topological isomorphism onto.
\end{lemma}

\begin{proof}
First by Lemma \ref{lemTech2}, $\alpha$ is identified with a compatible
quadruple
\[
\alpha=\left(  \left(  \alpha_{1}\left(  z_{1},z_{2}\right)  ,\alpha
_{2}\left(  z_{1},w_{2}\right)  \right)  ,\left(  \alpha_{3}\left(
w_{1},z_{2}\right)  ,\alpha_{4}\left(  w_{1},w_{2}\right)  \right)  \right)
\]
with the relations
\begin{align}
\alpha_{1}\left(  z_{1},0\right)   &  =\alpha_{2}\left(  z_{1},0\right)
,\quad\alpha_{3}\left(  w_{1},0\right)  =\alpha_{4}\left(  w_{1},0\right)
,\label{Alph}\\
\alpha_{1}\left(  0,z_{2}\right)   &  =\alpha_{3}\left(  0,z_{2}\right)
,\quad\alpha_{2}\left(  0,w_{2}\right)  =\alpha_{4}\left(  0,w_{2}\right)
.\nonumber
\end{align}
In this case, we have
\begin{align*}
\partial_{d,l}^{0}\left(  \alpha\right)   &  =\left[
\begin{array}
[c]{c}%
\left(  w_{1}-q^{l+1}w_{2}\right)  \alpha\\
\left(  z_{2}-q^{d+1}z_{1}\right)  \alpha
\end{array}
\right] \\
&  =\left[
\begin{array}
[c]{c}%
\left(  \left(  0,-q^{l+1}w_{2}\alpha_{2}\left(  z_{1},w_{2}\right)  \right)
,\left(  w_{1}\alpha_{3}\left(  w_{1},z_{2}\right)  ,\left(  w_{1}%
-q^{l+1}w_{2}\right)  \alpha_{4}\left(  w_{1},w_{2}\right)  \right)  \right)
\\
\left(  \left(  \left(  z_{2}-q^{d+1}z_{1}\right)  \alpha_{1}\left(
z_{1},z_{2}\right)  ,-q^{d+1}z_{1}\alpha_{2}\left(  z_{1},w_{2}\right)
\right)  ,\left(  z_{2}\alpha_{3}\left(  w_{1},z_{2}\right)  ,0\right)
\right)
\end{array}
\right]  .
\end{align*}
By Lemma \ref{lemTech4}, we conclude that $\partial_{d,l}^{0}\left(
\alpha\right)  =T\left(  \theta\right)  $ for a uniquely defined free
quadruple $\theta$. Therefore $\theta_{2}=\alpha_{2}$ and $\theta_{3}%
=\alpha_{3}$ with $\theta_{2}\left(  0,0\right)  =\alpha_{2}\left(
0,0\right)  =\alpha_{1}\left(  0,0\right)  =\alpha_{3}\left(  0,0\right)
=\theta_{3}\left(  0,0\right)  $. Moreover, using (\ref{Alph}), we obtain
that
\begin{align*}
\left(  w_{1}-q^{l+1}w_{2}\right)  \alpha_{4}\left(  w_{1},w_{2}\right)   &
=w_{1}\theta_{3}\left(  w_{1},0\right)  -q^{l+1}w_{2}\theta_{2}\left(
0,w_{2}\right)  +w_{1}w_{2}\theta_{4}\left(  w_{1},w_{2}\right) \\
&  =w_{1}\alpha_{3}\left(  w_{1},0\right)  -q^{l+1}w_{2}\alpha_{2}\left(
0,w_{2}\right)  +w_{1}w_{2}\theta_{4}\left(  w_{1},w_{2}\right) \\
&  =w_{1}\alpha_{4}\left(  w_{1},0\right)  -q^{l+1}w_{2}\alpha_{4}\left(
0,w_{2}\right)  +w_{1}w_{2}\theta_{4}\left(  w_{1},w_{2}\right)  ,
\end{align*}
which in turn implies that
\begin{align*}
w_{1}w_{2}\theta_{4}\left(  w_{1},w_{2}\right)   &  =w_{1}\left(  \alpha
_{4}\left(  w_{1},w_{2}\right)  -\alpha_{4}\left(  w_{1},0\right)  \right)
-q^{l+1}w_{2}\left(  \alpha_{4}\left(  w_{1},w_{2}\right)  -\alpha_{4}\left(
0,w_{2}\right)  \right) \\
&  =w_{1}w_{2}\left(  \alpha_{4}\right)  _{w_{2}}\left(  w_{1},w_{2}\right)
-q^{l+1}w_{2}w_{1}\left(  \alpha_{4}\right)  _{w_{1}}\left(  w_{1}%
,w_{2}\right)  ,
\end{align*}
that is, $\theta_{4}\left(  w_{1},w_{2}\right)  =\left(  \alpha_{4}\right)
_{w_{2}}\left(  w_{1},w_{2}\right)  -q^{l+1}\left(  \alpha_{4}\right)
_{w_{1}}\left(  w_{1},w_{2}\right)  $. In a similar way, we have (see
(\ref{Alph}))
\begin{align*}
\left(  z_{2}-q^{d+1}z_{1}\right)  \alpha_{1}\left(  z_{1},z_{2}\right)   &
=-q^{d+1}z_{1}\theta_{2}\left(  z_{1},0\right)  +z_{2}\theta_{3}\left(
0,z_{2}\right)  +z_{1}z_{2}\theta_{1}\left(  z_{1},z_{2}\right) \\
&  =-q^{d+1}z_{1}\alpha_{2}\left(  z_{1},0\right)  +z_{2}\alpha_{3}\left(
0,z_{2}\right)  +z_{1}z_{2}\theta_{1}\left(  z_{1},z_{2}\right) \\
&  =-q^{d+1}z_{1}\alpha_{1}\left(  z_{1},0\right)  +z_{2}\alpha_{1}\left(
0,z_{2}\right)  +z_{1}z_{2}\theta_{1}\left(  z_{1},z_{2}\right)  ,
\end{align*}
which in turn implies that $\theta_{1}\left(  z_{1},z_{2}\right)  =\left(
\alpha_{1}\right)  _{z_{1}}\left(  z_{1},z_{2}\right)  -q^{d+1}\left(
\alpha_{1}\right)  _{z_{2}}\left(  z_{1},z_{2}\right)  $. Thus
\begin{align*}
\theta &  =T^{-1}\partial_{d,l}^{0}\left(  \alpha\right) \\
&  =\left(  \left(  \alpha_{1}\right)  _{z_{1}}\left(  z_{1},z_{2}\right)
-q^{d+1}\left(  \alpha_{1}\right)  _{z_{2}}\left(  z_{1},z_{2}\right)
,\alpha_{2},\alpha_{3},\left(  \alpha_{4}\right)  _{w_{2}}\left(  w_{1}%
,w_{2}\right)  -q^{l+1}\left(  \alpha_{4}\right)  _{w_{1}}\left(  w_{1}%
,w_{2}\right)  \right)  .
\end{align*}
is the related quadruple. Moreover,
\begin{align*}
\theta_{1}\left(  z_{1},q^{d+1}z_{1}\right)   &  =\left(  \alpha_{1}\right)
_{z_{1}}\left(  z_{1},q^{d+1}z_{1}\right)  -q^{d+1}\left(  \alpha_{1}\right)
_{z_{2}}\left(  z_{1},q^{d+1}z_{1}\right)  =\frac{\alpha_{1}\left(
z_{1},0\right)  -\alpha_{1}\left(  0,q^{d+1}z_{1}\right)  }{z_{1}}\\
&  =\frac{\alpha_{2}\left(  z_{1},0\right)  -\alpha_{3}\left(  0,q^{d+1}%
z_{1}\right)  }{z_{1}}=\frac{\theta_{2}\left(  z_{1},0\right)  -\theta
_{3}\left(  0,q^{d+1}z_{1}\right)  }{z_{1}}%
\end{align*}
(see (\ref{Alph})). In a similar way, we have\
\begin{align*}
\theta_{4}\left(  q^{l+1}w_{2},w_{2}\right)   &  =\left(  \alpha_{4}\right)
_{w_{2}}\left(  q^{l+1}w_{2},w_{2}\right)  -q^{l+1}\left(  \alpha_{4}\right)
_{w_{1}}\left(  q^{l+1}w_{2},w_{2}\right)  =\frac{\alpha_{4}\left(
0,w_{2}\right)  -\alpha_{4}\left(  q^{l+1}w_{2},0\right)  }{w_{2}}\\
&  =\frac{\alpha_{2}\left(  0,w_{2}\right)  -\alpha_{3}\left(  q^{l+1}%
w_{2},0\right)  }{w_{2}}=\frac{\theta_{2}\left(  0,w_{2}\right)  -\theta
_{3}\left(  q^{l+1}w_{2},0\right)  }{w_{2}},
\end{align*}
which means that $\theta\in\mathcal{I}\left(  U\right)  $ or it is given by
the couple $\left(  \theta_{2},\theta_{3}\right)  \in\mathcal{O}\left(
z_{1},w_{2}\right)  \underset{\mathbb{C}}{\times}\mathcal{O}\left(
w_{1},z_{2}\right)  $.

Conversely, take a quadruple $\theta\in\mathcal{I}\left(  U\right)  $ related
to a couple $\left(  \theta_{2},\theta_{3}\right)  $ from the fibered product.
Put $f\left(  z_{1},z_{2}\right)  =q^{d+1}z_{1}\theta_{2}\left(
z_{1},0\right)  -z_{2}\theta_{3}\left(  0,z_{2}\right)  -z_{1}z_{2}\theta
_{1}\left(  z_{1},z_{2}\right)  $ to be a holomorphic function on the
polydomain $U_{x_{1}}\times U_{x_{2}}$, and it is vanishing over the
irreducible analytic set $z_{2}=q^{d+1}z_{1}$. Using Analytic Nullstellensatz
\cite[II.E.18]{GRo}, we conclude that $f$ is a multiply of the (irreducible)
polynomial $z_{2}-q^{d+1}z_{1}$, that is, $\left(  q^{d+1}z_{1}-z_{2}\right)
\alpha_{1}\left(  z_{1},z_{2}\right)  =f\left(  z_{1},z_{2}\right)  $ for some
analytic germ $\alpha_{1}$ at the origin. But $\alpha_{1}\left(  z_{1}%
,z_{2}\right)  =\left(  q^{d+1}z_{1}-z_{2}\right)  ^{-1}f\left(  z_{1}%
,z_{2}\right)  $ is holomorphic on $U_{x_{1}}\times U_{x_{2}}$ apart from the
set $z_{2}=q^{d+1}z_{1}$, which can be removed. Hence $\alpha_{1}%
\in\mathcal{O}\left(  z_{1},z_{2}\right)  $ and
\[
\left(  z_{2}-q^{d+1}z_{1}\right)  \alpha_{1}\left(  z_{1},z_{2}\right)
=-q^{d+1}z_{1}\theta_{2}\left(  z_{1},0\right)  +z_{2}\theta_{3}\left(
0,z_{2}\right)  +z_{1}z_{2}\theta_{1}\left(  z_{1},z_{2}\right)
\]
holds. In a similar way,
\[
\left(  w_{1}-q^{l+1}w_{2}\right)  \alpha_{4}\left(  w_{1},w_{2}\right)
=w_{1}\theta_{3}\left(  w_{1},0\right)  -q^{l+1}w_{2}\theta_{2}\left(
0,w_{2}\right)  +w_{1}w_{2}\theta_{4}\left(  w_{1},w_{2}\right)  .
\]
By assuming that $\alpha_{2}=\theta_{2}$ and $\alpha_{3}=\theta_{3}$, one can
easily verify the relations from (\ref{Alph}). For example, $\alpha_{4}\left(
0,w_{2}\right)  =\theta_{2}\left(  0,w_{2}\right)  =\alpha_{2}\left(
0,w_{2}\right)  $. It follows that $\alpha$ is a compatible quadruple and
$T^{-1}\partial_{d,l}^{0}\left(  \alpha\right)  =\theta$. It remains to use
Lemma \ref{lemTech4}.
\end{proof}

Now we can prove the following key result on the first diagonal cohomology group.

\begin{proposition}
\label{propTechE}The canonical identification $H_{d,l}^{1}=\mathcal{O}\left(
U\right)  $ holds up to a topological isomorphism. In this case, $\ker\left(
\partial_{d,l}^{1}\right)  =\operatorname{im}\left(  \partial_{d,l}%
^{1}\right)  \oplus\mathcal{O}\left(  U\right)  $ is a topological direct sum
of the Fr\'{e}chet spaces given by a continuous projection $p_{d,l}%
\in\mathcal{L}\left(  \ker\left(  \partial_{d,l}^{1}\right)  \right)  $ whose
image being identified with $\mathcal{O}\left(  U\right)  $ consists of the
following free quadruples
\[
\operatorname{im}\left(  p_{d,l}\right)  =\left\{  \left(  f\left(
z_{1}\right)  ,0,\lambda,g\left(  w_{2}\right)  \right)  :f\in\mathcal{O}%
\left(  z_{1}\right)  ,\lambda\in\mathbb{C},g\in\mathcal{O}\left(
w_{2}\right)  \right\}  \subseteq\mathcal{Z}\left(  U\right)  .
\]

\end{proposition}

\begin{proof}
We define the following continuous linear mapping
\begin{align*}
\Phi &  :\mathcal{Z}\left(  U\right)  \longrightarrow\mathcal{O}\left(
z\right)  \underset{\mathbb{C}}{\times}\mathcal{O}\left(  w\right)  ,\\
\Phi\left(  \theta\right)   &  =\left(  z\theta_{1}\left(  z,q^{d+1}z\right)
-\theta_{2}\left(  z,0\right)  +\theta_{3}\left(  0,q^{d+1}z\right)
,w\theta_{4}\left(  q^{l+1}w,w\right)  -\theta_{2}\left(  0,w\right)
+\theta_{3}\left(  q^{l+1}w,0\right)  \right)  .
\end{align*}
Notice that $\Phi\left(  \theta\right)  =\left(  f\left(  z\right)  ,g\left(
w\right)  \right)  \in\mathcal{O}\left(  z\right)  \oplus\mathcal{O}\left(
w\right)  $ with $f\left(  0\right)  =-\theta_{2}\left(  0,0\right)
+\theta_{3}\left(  0,0\right)  =g\left(  0\right)  $, that is, $\Phi\left(
\theta\right)  \in\mathcal{O}\left(  z\right)  \underset{\mathbb{C}}{\times
}\mathcal{O}\left(  w\right)  $. Further, $\theta\in\ker\left(  \Phi\right)  $
if and only if $\theta\in\mathcal{I}\left(  U\right)  $.

Now take $\left(  f\left(  z\right)  ,g\left(  w\right)  \right)
\in\mathcal{O}\left(  z\right)  \underset{\mathbb{C}}{\times}\mathcal{O}%
\left(  w\right)  $ with $\lambda=f\left(  0\right)  =g\left(  0\right)  $. It
follows that $f\left(  z\right)  =\lambda+zf_{1}\left(  z\right)  $ and
$g\left(  w\right)  =\lambda+wg_{1}\left(  w\right)  $ for some $f_{1}%
\in\mathcal{O}\left(  z\right)  $ and $g_{1}\in\mathcal{O}\left(  w\right)  $.
Put $\theta_{1}\left(  z_{1},z_{2}\right)  =f_{1}\left(  z_{1}\right)  $,
$\theta_{2}=0$, $\theta_{3}=\lambda$, $\theta_{4}\left(  w_{1},w_{2}\right)
=g_{1}\left(  w_{2}\right)  $, which defines a free quadruple $\theta$ such
that $\Phi\left(  \theta\right)  =\left(  zf_{1}\left(  z\right)
+\lambda,wg_{1}\left(  w\right)  +\lambda\right)  =\left(  f\left(  z\right)
,g\left(  w\right)  \right)  $. Thus
\[
\Psi:\mathcal{O}\left(  U\right)  \longrightarrow\mathcal{Z}\left(  U\right)
,\quad\Psi\left(  f,g\right)  =\left(  f_{z}\left(  z_{1}\right)  ,0,f\left(
0\right)  ,g_{w}\left(  w_{2}\right)  \right)
\]
is a continuous right inverse of $\Phi$. It follows that $p_{d,l}=\Psi\Phi$ is
a continuous projection on $\mathcal{Z}\left(  U\right)  $ such that
$\mathcal{Z}\left(  U\right)  =\mathcal{I}\left(  U\right)  \oplus
\operatorname{im}\left(  p_{d,l}\right)  $ is a topological direct sum, and
$\operatorname{im}\left(  p_{d,l}\right)  =\mathcal{O}\left(  U\right)  $ up
to a topological isomorphism of the Fr\'{e}chet spaces. Using Lemma
\ref{lemTech4} and Lemma \ref{lemTech5}, we deduce that
\[
H_{d,l}^{1}=\mathcal{Z}\left(  U\right)  /\mathcal{I}\left(  U\right)
=\mathcal{O}\left(  z\right)  \underset{\mathbb{C}}{\times}\mathcal{O}\left(
w\right)  =\mathcal{O}\left(  U\right)
\]
up to the topological isomorphisms, that is, $H_{d,l}^{1}=\mathcal{O}\left(
U\right)  $ holds.
\end{proof}

\subsection{The main result on the diagonal cohomology}

Now we can prove the central result on the cohomology groups of the diagonal
complex $\mathcal{O}_{q}\left(  d,l\right)  $.

\begin{theorem}
\label{thMain1}If $H_{d,l}^{p}$, $0\leq p\leq3$ are the cohomology groups of
the diagonal complex $\mathcal{O}_{q}\left(  d,l\right)  $, then $H_{d,l}%
^{0}=\left\{  0\right\}  $, $H_{d,l}^{1}=\mathcal{O}\left(  U\right)  $,
$H_{d,l}^{2}=\left\{  0\right\}  $ and $H_{d,l}^{3}=\left\{  0\right\}  $.
\end{theorem}

\begin{proof}
Take $h=h\left(  z_{1},w_{1};z_{2},w_{2}\right)  \in\mathcal{O}_{d,l}$ with
$\partial_{d,l}^{0}h=0$. It follows that
\[
\left(  w_{1}-q^{l+1}w_{2}\right)  h\left(  z_{1},w_{1};z_{2},w_{2}\right)
=\left(  z_{2}-q^{d+1}z_{1}\right)  h\left(  z_{1},w_{1};z_{2},w_{2}\right)
=0.
\]
Notice that $w_{1}=q^{l+1}w_{2}$ defines an algebraic (or thin) subset
$\Xi\subseteq\mathbb{C}^{2}$ to be zero set of the polynomial function
$w_{1}-q^{l+1}w_{2}$. The function $h$ is vanishing out of $\Xi\cap\left(
U_{_{1}}\times U_{_{2}}\right)  $. By Removable Singularity Theorem
\cite[I.C.3]{GRo}, we deduce that $h$ is vanishing over all $U_{1}\times
U_{2}$. Therefore $H_{d,l}^{0}=\left\{  0\right\}  $. The fact $H_{d,l}%
^{1}=\mathcal{O}\left(  U\right)  $ follows from Proposition \ref{propTechE}.

Now let us prove that $H_{d,l}^{2}=\left\{  0\right\}  $. Assume that
$\pi_{d,l}\left(  h\right)  =h\left(  z,q^{l}w;q^{d}z.w\right)  =0$ for some
$h\in\mathcal{O}_{d,l}$. As a compatible quadruple, we obtain that
\[
h_{1}\left(  z_{1},q^{d}z_{1}\right)  =0,\quad h_{4}\left(  q^{l}w_{2}%
,w_{2}\right)  =0.
\]
As above in the proof of Lemma \ref{lemTech5}, using Analytic Nullstellensatz,
we conclude that $h_{1}=\left(  z_{2}-q^{d}z_{1}\right)  f_{1}\left(
z_{1},z_{2}\right)  $ and $h_{4}=\left(  q^{l}w_{2}-w_{1}\right)  f_{4}\left(
w_{1},w_{2}\right)  $ for some $f_{1}\in\mathcal{O}\left(  z_{1},z_{2}\right)
$ and $f_{4}\in\mathcal{O}\left(  w_{1},w_{2}\right)  $. Based on the
compatibility conditions from Lemma \ref{lemTech2}, we deduce that%
\begin{align*}
-q^{d}z_{1}f_{1}\left(  z_{1},0\right)   &  =h_{2}\left(  z_{1},0\right)
,\quad h_{3}\left(  w_{1},0\right)  =-w_{1}f_{4}\left(  w_{1},0\right)  ,\\
z_{2}f_{1}\left(  0,z_{2}\right)   &  =h_{3}\left(  0,z_{2}\right)  ,\quad
h_{2}\left(  0,w_{2}\right)  =q^{l}w_{2}f_{4}\left(  0,w_{2}\right)  ,
\end{align*}
which invokes that $h_{2}\left(  0,0\right)  =h_{3}\left(  0,0\right)  =0$. It
follows that (see Remark \ref{remTech1})
\begin{align*}
h_{2}\left(  z_{1},w_{2}\right)   &  =h_{2}\left(  z_{1},0\right)
+h_{2}\left(  0,w_{2}\right)  +z_{1}w_{2}g_{2}\left(  z_{1},w_{2}\right) \\
&  =-q^{d}z_{1}f_{1}\left(  z_{1},0\right)  +q^{l}w_{2}f_{4}\left(
0,w_{2}\right)  +z_{1}w_{2}g_{2}\left(  z_{1},w_{2}\right)  ,\\
h_{3}\left(  w_{1},z_{2}\right)   &  =h_{3}\left(  w_{1},0\right)
+h_{3}\left(  0,z_{2}\right)  +w_{1}z_{2}g_{3}\left(  w_{1},z_{2}\right) \\
&  =-w_{1}f_{4}\left(  w_{1},0\right)  +z_{2}f_{1}\left(  0,z_{2}\right)
+w_{1}z_{2}g_{3}\left(  w_{1},z_{2}\right)
\end{align*}
for some $g_{2}\in\mathcal{O}\left(  z_{1},w_{2}\right)  $ and $g_{3}%
\in\mathcal{O}\left(  w_{1},z_{2}\right)  $. We define
\begin{align*}
\zeta &  =\left(  \left(  f_{1}\left(  z_{1},z_{2}\right)  ,f_{1}\left(
z_{1},0\right)  \right)  ,\left(  f_{1}\left(  0,z_{2}\right)  +w_{1}%
g_{3}\left(  w_{1},z_{2}\right)  ,f_{1}\left(  0,0\right)  +w_{1}g_{3}\left(
w_{1},0\right)  \right)  \right)  ,\\
\eta &  =\left(  \left(  f_{4}\left(  0,0\right)  +q^{-l}z_{1}g_{2}\left(
z_{1},0\right)  ,f_{4}\left(  0,w_{2}\right)  +q^{-l}z_{1}g_{2}\left(
z_{1},w_{2}\right)  \right)  ,\left(  f_{4}\left(  w_{1},0\right)
,f_{4}\left(  w_{1},w_{2}\right)  \right)  \right)
\end{align*}
to be the elements of $\mathcal{O}_{d,l}$. One can easily verify that the
compatibility conditions from Lemma \ref{lemTech2} are satisfied. If
$\psi=\left[
\begin{array}
[c]{c}%
\zeta\\
\eta
\end{array}
\right]  \in\mathcal{O}_{d,l}^{\oplus2}$, then (see Corollary \ref{corTech1})
\begin{align*}
\partial_{d,l}^{1}\psi &  =\left(  z_{2}-q^{d}z_{1}\right)  \zeta+\left(
q^{l}w_{2}-w_{1}\right)  \eta\\
&  =\left(  \left(  h_{1}\left(  z_{1},z_{2}\right)  ,-q^{d}z_{1}f_{1}\left(
z_{1},0\right)  \right)  ,\left(  z_{2}f_{1}\left(  0,z_{2}\right)
+z_{2}w_{1}g_{3}\left(  w_{1},z_{2}\right)  ,0\right)  \right) \\
&  +\left(  \left(  0,q^{l}w_{2}f_{4}\left(  0,w_{2}\right)  +z_{1}w_{2}%
g_{2}\left(  z_{1},w_{2}\right)  \right)  ,\left(  -w_{1}f_{4}\left(
w_{1},0\right)  ,h_{4}\left(  w_{1},w_{2}\right)  \right)  \right) \\
&  =\left(  \left(  h_{1}\left(  z_{1},z_{2}\right)  ,h_{2}\left(  z_{1}%
,w_{2}\right)  \right)  ,\left(  h_{3}\left(  w_{1},z_{2}\right)
,h_{4}\left(  w_{1},w_{2}\right)  \right)  \right)  =h.
\end{align*}
Hence $H_{d,l}^{2}=\left\{  0\right\}  $. In particular, $\operatorname{im}%
\partial_{d,l}^{1}$ is closed.

Finally, let us prove that $H_{d,l}^{3}=\left\{  0\right\}  $. Take $\left(
f,g\right)  \in\mathcal{O}_{d+l}$, that is, $f\left(  0\right)  =g\left(
0\right)  $. Put $h=\left(  \left(  f\left(  z_{1}\right)  ,-f\left(
0\right)  +f\left(  z_{1}\right)  +g\left(  w_{2}\right)  \right)  ,\left(
g\left(  0\right)  ,g\left(  w_{2}\right)  \right)  \right)  $ to be a
compatible quadruple from $\mathcal{O}_{d,l}$. Then $\pi_{d,l}\left(
h\right)  =\left(  f,g\right)  $, that is, $\pi_{d,l}$ is onto.
\end{proof}

Thus the exactness fluctuations in $\mathcal{O}_{q}\left(  d,l\right)  $ occur
only in the first cohomology group. Actually, $\mathcal{O}_{q}\left(
d,l\right)  $ almost splits in the sense of the forthcoming assertion. Let us
consider the following continuous linear maps%
\begin{align*}
\tau_{d,l}^{0}  &  :\mathcal{O}_{d+l}\longrightarrow\mathcal{O}_{d,l}%
,\quad\tau_{d,l}^{0}\left(  f,g\right)  =\left(  \left(  f\left(
z_{1}\right)  ,-f\left(  0\right)  +f\left(  z_{1}\right)  +g\left(
w_{2}\right)  \right)  ,\left(  g\left(  0\right)  ,g\left(  w_{2}\right)
\right)  \right)  ,\\
\tau_{d,l}^{1}  &  :\mathcal{O}_{d,l}\longrightarrow%
\begin{array}
[c]{c}%
\mathcal{O}_{d,l}\\
\oplus\\
\mathcal{O}_{d,l}%
\end{array}
,\quad\tau_{d,l}^{1}\left(  \zeta\right)  =\left[
\begin{array}
[c]{c}%
\alpha\\
\beta
\end{array}
\right]
\end{align*}
with
\begin{align*}
\alpha_{1}\left(  z_{1},z_{2}\right)   &  =\frac{\zeta_{1}\left(  z_{1}%
,z_{2}\right)  -\zeta_{1}\left(  z_{1},q^{d}z_{1}\right)  }{z_{2}-q^{d}z_{1}%
},\quad\alpha_{2}\left(  z_{1},w_{2}\right)  =\alpha_{1}\left(  z_{1}%
,0\right)  -q^{-d}w_{2}\left(  \zeta_{2}\right)  _{z_{1}w_{2}}\left(
z_{1},w_{2}\right)  ,\\
\alpha_{3}\left(  w_{1},z_{2}\right)   &  =\alpha_{1}\left(  0,z_{2}\right)
,\quad\alpha_{4}\left(  w_{1},w_{2}\right)  =\alpha_{1}\left(  0,0\right)
-q^{-d}w_{2}\left(  \zeta_{2}\right)  _{z_{1}w_{2}}\left(  0,w_{2}\right)  ,
\end{align*}
and%
\begin{align*}
\beta_{1}\left(  z_{1},z_{2}\right)   &  =\beta_{4}\left(  0,0\right)
-z_{2}\left(  \zeta_{3}\right)  _{w_{1}z_{2}}\left(  0,z_{2}\right)
,\quad\beta_{2}\left(  z_{1},w_{2}\right)  =\beta_{4}\left(  0,w_{2}\right)
,\\
\beta_{3}\left(  w_{1},z_{2}\right)   &  =\beta_{4}\left(  w_{1},0\right)
-z_{2}\left(  \zeta_{3}\right)  _{w_{1}z_{2}}\left(  w_{1},z_{2}\right)
,\quad\beta_{4}\left(  w_{1},w_{2}\right)  =\frac{\zeta_{4}\left(  w_{1}%
,w_{2}\right)  -\zeta_{4}\left(  q^{l}w_{2},w_{2}\right)  }{q^{l}w_{2}-w_{1}}.
\end{align*}
They are occurred in the proofs of the previous assertions implicitly. We
resume all these maps and the related projections including $p_{d,l}$ from
Proposition \ref{propTechE} in the following assertion.

\begin{corollary}
\label{corDecom}The identities $\pi_{d,l}\tau_{d,l}^{0}=1$ and $\tau_{d,l}%
^{0}\pi_{d,l}+\partial_{d,l}^{1}\tau_{d,l}^{1}=1$ hold. In this case,
$p_{d,l}^{0}=\tau_{d,l}^{0}\pi_{d,l}\in\mathcal{L}\left(  \mathcal{O}%
_{d,l}\right)  $ and $p_{d,l}^{1}=\tau_{d,l}^{1}\partial_{d,l}^{1}%
\in\mathcal{L}\left(  \mathcal{O}_{d,l}^{\oplus2}\right)  $ are the continuous
projections such that the complex $\mathcal{O}_{q}\left(  d,l\right)  $ is
decomposed in the following way
\[%
\begin{array}
[c]{ccccccc}%
0\rightarrow\mathcal{O}_{d,l} & \overset{\partial_{d,l}^{0}}{\longrightarrow}
& \mathcal{O}_{d,l}^{\oplus2} & \overset{\partial_{d,l}^{1}}{\longrightarrow}
& \mathcal{O}_{d,l} & \overset{\pi_{d,l}}{\longrightarrow} & \mathcal{O}%
_{d+l}\rightarrow0\\
&  & \parallel &  & \parallel &  & \\%
\begin{array}
[c]{c}%
\\
\\
\\
\\
0\rightarrow\mathcal{O}_{d,l}%
\end{array}
&
\begin{array}
[c]{c}%
\\
\\
\\
\\
\overset{\partial_{d,l}^{0}}{\longleftrightarrow}%
\end{array}
&
\begin{array}
[c]{c}%
\operatorname{im}\left(  p_{d,l}^{1}\right) \\
\oplus\\
\operatorname{im}\left(  p_{d,l}\right) \\
\oplus\\
\operatorname{im}\left(  \partial_{d,l}^{0}\right)
\end{array}
&
\begin{array}
[c]{c}%
\\
_{\tau_{d,l}^{1}}\nwarrow\searrow^{\partial_{d,l}^{1}}\\
\\
\\
\end{array}
&
\begin{array}
[c]{c}%
\operatorname{im}\left(  p_{d,l}^{0}\right) \\
\oplus\\
\operatorname{im}\left(  \partial_{d,l}^{1}\right) \\
\\
\end{array}
&
\begin{array}
[c]{c}%
\\
_{\tau_{d,l}^{0}}\nwarrow\searrow^{\pi_{d,l}}\\
\\
\\
\end{array}
&
\begin{array}
[c]{c}%
\\
\\
\mathcal{O}_{d+l}\rightarrow0\\
\\
\end{array}
\end{array}
\]

\end{corollary}

\begin{proof}
Take $\zeta\in\mathcal{O}_{d,l}$ and put $f\left(  z\right)  =\zeta_{1}\left(
z,q^{d}z\right)  $, $g\left(  w\right)  =\zeta_{4}\left(  q^{l}w,w\right)  $
with $\pi_{d,l}\left(  \zeta\right)  =\left(  f\left(  z\right)  ,g\left(
w\right)  \right)  $. In this case,
\[
\tau_{d,l}^{0}\pi_{d,l}\left(  \zeta\right)  =\left(  \left(  f\left(
z_{1}\right)  ,f\left(  z_{1}\right)  -\lambda+g\left(  w_{2}\right)  \right)
,\left(  \lambda,g\left(  w_{2}\right)  \right)  \right)
\]
with $\lambda=f\left(  0\right)  =g\left(  0\right)  $.

Let us choose $\alpha_{1}\in\mathcal{O}\left(  z_{1},z_{2}\right)  $ and
$\beta_{4}\in\mathcal{O}\left(  w_{1},w_{2}\right)  $ so that $\zeta
_{1}\left(  z_{1},z_{2}\right)  =f\left(  z_{1}\right)  +\left(  z_{2}%
-q^{d}z_{1}\right)  \alpha_{1}\left(  z_{1},z_{2}\right)  $ and $\zeta
_{4}\left(  w_{1},w_{2}\right)  =g\left(  w_{2}\right)  +\left(  q^{l}%
w_{2}-w_{1}\right)  \beta_{4}\left(  w_{1},w_{2}\right)  $. Put
\begin{align*}
\alpha_{2}\left(  z_{1},w_{2}\right)   &  =\alpha_{1}\left(  z_{1},0\right)
-q^{-d}w_{2}\left(  \zeta_{2}\right)  _{z_{1}w_{2}}\left(  z_{1},w_{2}\right)
,\quad\alpha_{3}\left(  w_{1},z_{2}\right)  =\alpha_{1}\left(  0,z_{2}\right)
,\\
\alpha_{4}\left(  w_{1},w_{2}\right)   &  =\alpha_{1}\left(  0,0\right)
-q^{-d}w_{2}\left(  \zeta_{2}\right)  _{z_{1}w_{2}}\left(  0,w_{2}\right)  ,
\end{align*}
and%
\begin{align*}
\beta_{1}\left(  z_{1},z_{2}\right)   &  =\beta_{4}\left(  0,0\right)
-z_{2}\left(  \zeta_{3}\right)  _{w_{1}z_{2}}\left(  0,z_{2}\right)
,\quad\beta_{2}\left(  z_{1},w_{2}\right)  =\beta_{4}\left(  0,w_{2}\right)
,\\
\beta_{3}\left(  w_{1},z_{2}\right)   &  =\beta_{4}\left(  w_{1},0\right)
-z_{2}\left(  \zeta_{3}\right)  _{w_{1}z_{2}}\left(  w_{1},z_{2}\right)  ,
\end{align*}
which result in the quadruples $\alpha$ and $\beta$. Based on Lemma
\ref{lemTech2}, one can easily verify that they are compatible quadruples from
$\mathcal{O}_{d,l}$ \ indeed. Notice that $\tau_{d,l}^{1}\left(  \zeta\right)
=\left[
\begin{array}
[c]{c}%
\alpha\\
\beta
\end{array}
\right]  $. For brevity we put $a=\left(  z_{2}-q^{d}z_{1}\right)  \alpha$ and
$b=\left(  q^{l}w_{2}-w_{1}\right)  \beta$. By Corollary \ref{corTech1}, we
have
\begin{align*}
a  &  =\left(  \left(  \left(  z_{2}-q^{d}z_{1}\right)  \alpha_{1}\left(
z_{1},z_{2}\right)  ,-q^{d}z_{1}\alpha_{1}\left(  z_{1},0\right)  +z_{1}%
w_{2}\left(  \zeta_{2}\right)  _{z_{1}w_{2}}\left(  z_{1},w_{2}\right)
\right)  ,\left(  z_{2}\alpha_{1}\left(  0,z_{2}\right)  ,0\right)  \right)
,\\
b  &  =\left(  \left(  0,q^{l}w_{2}\beta_{4}\left(  0,w_{2}\right)  \right)
,\left(  -w_{1}\beta_{4}\left(  w_{1},0\right)  +w_{1}z_{2}\left(  \zeta
_{3}\right)  _{w_{1}z_{2}}\left(  w_{1},z_{2}\right)  ,\left(  q^{l}%
w_{2}-w_{1}\right)  \beta_{4}\left(  w_{1},w_{2}\right)  \right)  \right)  .
\end{align*}
Moreover, $\partial_{d,l}^{1}\tau_{d,l}^{1}\left(  \zeta\right)
=\partial_{d,l}^{1}\left[
\begin{array}
[c]{c}%
\alpha\\
\beta
\end{array}
\right]  =a+b$. But $\zeta_{1}\left(  z_{1},0\right)  =\zeta_{2}\left(
z_{1},0\right)  $ and $\zeta_{4}\left(  0,w_{2}\right)  =\zeta_{2}\left(
0,w_{2}\right)  $, therefore (see Lemma \ref{lemTech2} and Remark
\ref{remTech1})
\begin{align*}
a_{2}+b_{2}  &  =-q^{d}z_{1}\alpha_{1}\left(  z_{1},0\right)  +z_{1}%
w_{2}\left(  \zeta_{2}\right)  _{z_{1}w_{2}}\left(  z_{1},w_{2}\right)
+q^{l}w_{2}\beta_{4}\left(  0,w_{2}\right) \\
&  =\zeta_{1}\left(  z_{1},0\right)  -f\left(  z_{1}\right)  +z_{1}%
w_{2}\left(  \zeta_{2}\right)  _{z_{1}w_{2}}\left(  z_{1},w_{2}\right)
+\zeta_{4}\left(  0,w_{2}\right)  -g\left(  w_{2}\right) \\
&  =\zeta_{2}\left(  z_{1},0\right)  +\zeta_{2}\left(  0,w_{2}\right)
+z_{1}w_{2}\left(  \zeta_{2}\right)  _{z_{1}w_{2}}\left(  z_{1},w_{2}\right)
-f\left(  z_{1}\right)  -g\left(  w_{2}\right) \\
&  =\zeta_{2}\left(  z_{1},z_{2}\right)  +\zeta_{2}\left(  0,0\right)
-f\left(  z_{1}\right)  -g\left(  w_{2}\right)  =\zeta_{2}\left(  z_{1}%
,z_{2}\right)  +\lambda-f\left(  z_{1}\right)  -g\left(  w_{2}\right)
\end{align*}
In a similar way, we have
\begin{align*}
a_{3}+b_{3}  &  =z_{2}\alpha_{1}\left(  0,z_{2}\right)  -w_{1}\beta_{4}\left(
w_{1},0\right)  +w_{1}z_{2}\left(  \zeta_{3}\right)  _{w_{1}z_{2}}\left(
w_{1},z_{2}\right) \\
&  =\zeta_{1}\left(  0,z_{2}\right)  -f\left(  0\right)  +\zeta_{4}\left(
w_{1},0\right)  -g\left(  0\right)  +w_{1}z_{2}\left(  \zeta_{3}\right)
_{w_{1}z_{2}}\left(  w_{1},z_{2}\right) \\
&  =\zeta_{3}\left(  0,z_{2}\right)  -\zeta_{3}\left(  0,0\right)  +\zeta
_{3}\left(  w_{1},0\right)  -\zeta_{3}\left(  0,0\right)  +w_{1}z_{2}\left(
\zeta_{3}\right)  _{w_{1}z_{2}}\left(  w_{1},z_{2}\right) \\
&  =\zeta_{3}\left(  w_{1},z_{2}\right)  -\zeta_{3}\left(  0,0\right)
=\zeta_{3}\left(  w_{1},z_{2}\right)  -\lambda.
\end{align*}
It follows that
\begin{align*}
a+b  &  =\left(  \zeta_{1}\left(  z_{1},z_{2}\right)  -f\left(  z_{1}\right)
,\zeta_{2}\left(  z_{1},z_{2}\right)  +\lambda-f\left(  z_{1}\right)
-g\left(  w_{2}\right)  ,\zeta_{3}\left(  w_{1},z_{2}\right)  -\lambda
,\zeta_{4}\left(  w_{1},w_{2}\right)  -g\left(  w_{2}\right)  \right) \\
&  =\zeta-\left(  \left(  f\left(  z_{1}\right)  ,-\lambda+f\left(
z_{1}\right)  +g\left(  w_{2}\right)  \right)  ,\left(  \lambda,g\left(
w_{2}\right)  \right)  \right)  =\zeta-\tau_{d,l}^{0}\pi_{d,l}\left(
\zeta\right)  ,
\end{align*}
that is, $\partial_{d,l}^{1}\tau_{d,l}^{1}\left(  \zeta\right)  +\tau
_{d,l}^{0}\pi_{d,l}\left(  \zeta\right)  =\zeta$. Thus $p_{d,l}^{0}=\tau
_{d,l}^{0}\pi_{d,l}$ is a continuous projection with
\[
p_{d,l}^{0}\left(  \zeta\right)  =\left(  \left(  f\left(  z_{1}\right)
,f\left(  z_{1}\right)  -\lambda+g\left(  w_{2}\right)  \right)  ,\left(
\lambda,g\left(  w_{2}\right)  \right)  \right)  \quad\text{and\quad
}\mathcal{O}_{d,l}=\operatorname{im}\left(  p_{d,l}^{0}\right)  \oplus
\operatorname{im}\left(  \partial_{d,l}^{1}\right)  .
\]
Moreover, $p_{d,l}^{1}=\tau_{d,l}^{1}\partial_{d,l}^{1}$ is a projection too,
and $\operatorname{im}\left(  1-p_{d,l}^{1}\right)  =\ker\partial_{d,l}^{1}$.
Hence $\mathcal{O}_{d,l}^{\oplus2}=\operatorname{im}\left(  p_{d,l}%
^{1}\right)  \oplus\ker\partial_{d,l}^{1}$. The rest follows from Proposition
\ref{propTechE} and Theorem \ref{thMain1}.
\end{proof}

\subsection{The partial differential and difference operators $M$,
$N$\label{subsecPDE}}

The operators $M_{x_{i},l}$, $M_{y_{i},l}$, $N_{x_{j},d}$ and $N_{y_{j},d}$
from (\ref{Nd}) considered above in Subsection \ref{subsecQRP} (see Lemma
\ref{lemTech3}), in turn define the following partial differential and
difference operators
\[
M_{d,l}=\left[
\begin{array}
[c]{cc}%
M_{x_{2},l} & M_{y_{2},l}%
\end{array}
\right]  :\mathcal{O}_{d,l}^{\oplus2}\rightarrow\mathcal{O}_{d,l+1}%
\text{,\quad}N_{d,l}=\left[
\begin{array}
[c]{cc}%
N_{y_{1},d} & N_{x_{1},d}%
\end{array}
\right]  :\mathcal{O}_{d,l}^{\oplus2}\rightarrow\mathcal{O}_{d+1,l}.
\]
The following technical result plays a key role in the proof of the main
result (Theorem \ref{thCENTER}).

\begin{lemma}
\label{lemDDkey}If $\beta\in\ker\left(  \partial_{d,l}^{1}\right)  $ with
$M_{d,l}\beta\in\ker\left(  \pi_{d,l+1}\right)  $ or $N_{d,l}\beta\in
\ker\left(  \pi_{d+1,l}\right)  $, then $\beta\in\operatorname{im}\left(
\partial_{d,l}^{0}\right)  $. Thus
\[
\left(  M_{d,l}^{-1}\left(  \ker\left(  \pi_{d,l+1}\right)  \right)  \cup
N_{d,l}^{-1}\left(  \ker\left(  \pi_{d+1,l}\right)  \right)  \right)  \cap
\ker\left(  \partial_{d,l}^{1}\right)  \subseteq\operatorname{im}\left(
\partial_{d,l}^{0}\right)
\]
for all $d,l$.
\end{lemma}

\begin{proof}
By Lemma \ref{lemTech4}, we have $\beta=\left[
\begin{array}
[c]{c}%
\zeta\\
\eta
\end{array}
\right]  =T\left(  \theta\right)  $ for a uniquely defined free quadruple
$\theta\in\mathcal{Z}\left(  U\right)  $. Put $h=M_{d,l}\beta=\overline{\zeta
}+\overline{\eta}$ with $\overline{\zeta}=M_{x_{2},l}\zeta$, $\overline{\eta
}=M_{y_{2},l}\eta$. First assume that $\pi_{d,l+1}\left(  h\right)  =0$. Using
Lemma \ref{lemTech3}, we deduce that (see to the definition of $T\left(
\theta\right)  $ in Subsection \ref{subsecOT})%
\begin{align*}
\overline{\zeta}\left(  z_{1},w_{1};z_{2},w_{2}\right)   &  =\left(  \left(
\overline{\zeta}_{1}\left(  z_{1},z_{2}\right)  ,\overline{\zeta}_{2}\left(
z_{1},w_{2}\right)  \right)  ,\left(  \overline{\zeta}_{3}\left(  w_{1}%
,z_{2}\right)  ,\overline{\zeta}_{4}\left(  w_{1},w_{2}\right)  \right)
\right)  ,\\
\overline{\zeta}_{1}\left(  z_{1},z_{2}\right)   &  =-q^{l+1}\theta_{2}\left(
z_{1},0\right)  ,\text{\quad}\overline{\zeta}_{2}\left(  z_{1},w_{2}\right)
=-q^{l+1}\theta_{2}\left(  z_{1},w_{2}\right)  ,\\
\overline{\zeta}_{3}\left(  z_{1},w_{2}\right)   &  =-q^{l+1}\theta_{2}\left(
0,0\right)  +w_{1}\theta_{4}\left(  w_{1},0\right)  ,\\
\overline{\zeta}_{4}\left(  z_{1},w_{2}\right)   &  =-q^{l+1}\theta_{2}\left(
0,w_{2}\right)  +w_{1}\theta_{4}\left(  w_{1},w_{2}\right)  ,
\end{align*}
and $q^{-l-1}\overline{\eta}$ has the following components
\[
q^{-l-1}\overline{\eta}=\left(  \left(  \theta_{3}\left(  0,qz_{2}\right)
+z_{1}\theta_{1}\left(  z_{1},qz_{2}\right)  ,\theta_{3}\left(  0,0\right)
+z_{1}\theta_{1}\left(  z_{1},0\right)  \right)  ,\left(  \theta_{3}\left(
w_{1},qz_{2}\right)  ,\theta_{3}\left(  w_{1},0\right)  \right)  \right)  .
\]
Since $h=h\left(  z_{1},w_{1};z_{2},w_{2}\right)  =\overline{\zeta}%
+\overline{\eta}$, it follows that%
\begin{align*}
h_{1}\left(  z_{1},z_{2}\right)   &  =q^{l+1}\left(  z_{1}\theta_{1}\left(
z_{1},qz_{2}\right)  -\theta_{2}\left(  z_{1},0\right)  +\theta_{3}\left(
0,qz_{2}\right)  \right)  ,\\
h_{2}\left(  z_{1},w_{2}\right)   &  =q^{l+1}\left(  z_{1}\theta_{1}\left(
z_{1},0\right)  -\theta_{2}\left(  z_{1},w_{2}\right)  +\theta_{3}\left(
0,0\right)  \right)  ,\\
h_{3}\left(  w_{1},z_{2}\right)   &  =-q^{l+1}\theta_{2}\left(  0,0\right)
+q^{l+1}\theta_{3}\left(  w_{1},qz_{2}\right)  +w_{1}\theta_{4}\left(
w_{1},0\right)  ,\\
h_{4}\left(  w_{1},w_{2}\right)   &  =-q^{l+1}\theta_{2}\left(  0,w_{2}%
\right)  +q^{l+1}\theta_{3}\left(  w_{1},0\right)  +w_{1}\theta_{4}\left(
w_{1},w_{2}\right)  .
\end{align*}
Moreover, the assumption $h\left(  z_{1},q^{l+1}w_{2};q^{d}z_{1},w_{2}\right)
=\pi_{d,l+1}\left(  h\right)  =0$ invokes that $h_{1}\left(  z_{1},q^{d}%
z_{1}\right)  =0$ and $h_{4}\left(  q^{l+1}w_{2},w_{2}\right)  =0$. In
particular, $\theta_{2}\left(  0,0\right)  =\theta_{3}\left(  0,0\right)  $,
$\theta_{1}$ and $\theta_{4}$ do satisfy (\ref{Theta}), which means that
$\theta\in\mathcal{I}\left(  U\right)  $. By Lemma \ref{lemTech5}, we derive
that $\theta=T^{-1}\partial_{d,l}^{0}\left(  \alpha\right)  $ for a certain
$\alpha\in\mathcal{O}_{d,l}$, that is, $\beta=T\left(  \theta\right)
=\partial_{d,l}^{0}\left(  \alpha\right)  \in\operatorname{im}\left(
\partial_{d,l}^{0}\right)  $.

Now assume that $h=N_{d,l}\beta=\overline{\zeta}+\overline{\eta}$ with
$\overline{\zeta}=N_{y_{1},d}\zeta$, $\overline{\eta}=N_{x_{1},d}\eta$, and
$\pi_{d+1,l}\left(  h\right)  =0$. Using again Lemma \ref{lemTech3}, we deduce
that%
\[
q^{-d-1}\overline{\zeta}=\left(  \left(  \theta_{3}\left(  0,z_{2}\right)
,\theta_{3}\left(  0,0\right)  +w_{2}\theta_{4}\left(  0,w_{2}\right)
\right)  ,\left(  \theta_{3}\left(  qw_{1},z_{2}\right)  ,\theta_{3}\left(
qw_{1},0\right)  +w_{2}\theta_{4}\left(  qw_{1},w_{2}\right)  \right)
\right)  ,
\]
whereas for $\overline{\eta}$ we have
\begin{align*}
\overline{\eta}_{1}\left(  z_{1},z_{2}\right)   &  =-q^{d+1}\theta_{2}\left(
z_{1},0\right)  +z_{2}\theta_{1}\left(  z_{1},z_{2}\right)  ,\quad
\overline{\eta}_{2}\left(  z_{1},w_{2}\right)  =-q^{d+1}\theta_{2}\left(
z_{1},w_{2}\right)  ,\\
\overline{\eta}_{3}\left(  z_{1},w_{2}\right)   &  =-q^{d+1}\theta_{2}\left(
0,0\right)  +z_{2}\theta_{1}\left(  0,z_{2}\right)  ,\quad\overline{\eta}%
_{4}\left(  z_{1},w_{2}\right)  =-q^{d+1}\theta_{2}\left(  0,w_{2}\right)  .
\end{align*}
Taking into account that $h=\overline{\zeta}+\overline{\eta}$, we derive that
\begin{align*}
h_{1}\left(  z_{1},z_{2}\right)   &  =z_{2}\theta_{1}\left(  z_{1}%
,z_{2}\right)  -q^{d+1}\theta_{2}\left(  z_{1},0\right)  +q^{d+1}\theta
_{3}\left(  0,z_{2}\right)  ,\\
h_{2}\left(  z_{1},w_{2}\right)   &  =q^{d+1}\left(  -\theta_{2}\left(
z_{1},w_{2}\right)  +\theta_{3}\left(  0,0\right)  +w_{2}\theta_{4}\left(
0,w_{2}\right)  \right)  ,\\
h_{3}\left(  w_{1},z_{2}\right)   &  =z_{2}\theta_{1}\left(  0,z_{2}\right)
-q^{d+1}\theta_{2}\left(  0,0\right)  +q^{d+1}\theta_{3}\left(  qw_{1}%
,z_{2}\right)  ,\\
h_{4}\left(  w_{1},w_{2}\right)   &  =q^{d+1}\left(  -\theta_{2}\left(
0,w_{2}\right)  +\theta_{3}\left(  qw_{1},0\right)  +w_{2}\theta_{4}\left(
qw_{1},w_{2}\right)  \right)  ,
\end{align*}
and $h\left(  z_{1},q^{l}w_{2};q^{d+1}z_{1},w_{2}\right)  =\pi_{d+1,l}\left(
h\right)  =0$. In particular, we\ have $h_{1}\left(  z_{1},q^{d+1}%
z_{1}\right)  =0$ and $h_{4}\left(  q^{l}w_{2},w_{2}\right)  =0$, which result
in (\ref{Theta}). The rest is clear.
\end{proof}

In the following assertion, the upper indices $\left(  k\right)  $ should not
be mixed up with the derivatives.

\begin{lemma}
\label{lemOMN}Let $\left\{  s_{i,j}^{\left(  k\right)  }\right\}
\subseteq\operatorname{im}\left(  p_{i,j}\right)  $ and $\left\{  \omega
_{i,j}^{\left(  k\right)  }\right\}  \subseteq\mathcal{O}_{i,j}^{\oplus2}$ be
sequences such that the limit
\[
\chi_{i,j}=\lim_{k}\partial_{i,j}^{1}\omega_{i,j}^{\left(  k\right)
}+M_{i,j-1}s_{i,j-1}^{\left(  k\right)  }-N_{i-1,j}s_{i-1,j}^{\left(
k\right)  }\quad\text{in\quad}\mathcal{O}_{i,j}%
\]
converges to some $\chi_{i,j}\in\operatorname{im}\left(  p_{i,j}^{0}\right)  $
for all $i$, $j$, $i+j\leq n$. Then there are limits $\lim_{k}s_{i,j}^{\left(
k\right)  }=s_{i,j}$, $i+j<n$ and $\lim_{k}\partial_{i,j}^{1}\omega
_{i,j}^{\left(  k\right)  }=\partial_{i,j}^{1}\omega_{i,j}$, $i+j\leq n$ such
that
\[
\chi_{i,j}=\partial_{i,j}^{1}\omega_{i,j}+M_{i,j-1}s_{i,j-1}-N_{i-1,j}%
s_{i-1,j}%
\]
holds for all $i,j$, $i+j\leq n$.
\end{lemma}

\begin{proof}
First prove that $\lim_{k}s_{i,j-1}^{\left(  k\right)  }=s_{i,j-1}$ converges
iff so is $\lim_{k}s_{i-1,j}^{\left(  k\right)  }=s_{i-1,j}$ whenever $i+j\leq
n$. In this case, we would have
\[
\chi_{i,j}-M_{i,j-1}s_{i,j-1}+N_{i-1,j}s_{i-1,j}=\lim_{k}\partial_{i,j}%
^{1}\omega_{i,j}^{\left(  k\right)  }%
\]
out of the continuity of the operators $M_{i,j-1}$ and $N_{i-1,j}$. Taking
into account that $\operatorname{im}\left(  \partial_{i,j}^{1}\right)  $ is
closed (see Theorem \ref{thMain1}), we obtain that $\lim_{k}\partial_{i,j}%
^{1}\omega_{i,j}^{\left(  k\right)  }=\partial_{i,j}^{1}\omega_{i,j}$ for some
$\omega_{i,j}$, and the result $\chi_{i,j}=\partial_{i,j}^{1}\omega
_{i,j}+M_{i,j-1}s_{i,j-1}-N_{i-1,j}s_{i-1,j}$ follows.

Now by assumption, we have
\begin{align*}
s_{i,j-1}^{\left(  k\right)  }  &  =p_{i,j-1}\left(  s_{i,j-1}^{\left(
k\right)  }\right)  =\left(  f^{\left(  k\right)  }\left(  z_{1}\right)
,0,\lambda^{\left(  k\right)  },g^{\left(  k\right)  }\left(  w_{2}\right)
\right)  ,\\
s_{i-1,j}^{\left(  k\right)  }  &  =p_{i-1,j}\left(  s_{i-1,j}^{\left(
k\right)  }\right)  =\left(  u^{\left(  k\right)  }\left(  z_{1}\right)
,0,\mu^{\left(  k\right)  },v^{\left(  k\right)  }\left(  w_{2}\right)
\right)
\end{align*}
are free quadruples thanks to Proposition \ref{propTechE}, where $\left\{
f^{\left(  k\right)  }\right\}  \subseteq\mathcal{O}\left(  z\right)  $,
$\left\{  g^{\left(  k\right)  }\right\}  \subseteq\mathcal{O}\left(
w\right)  $ and $\left\{  \lambda^{\left(  k\right)  }\right\}  \subseteq
\mathbb{C}$. By Lemma \ref{lemTech4}, these free quadruples are in turn
identified with the following compatible ones
\begin{align*}
T\left(  s_{i,j-1}^{\left(  k\right)  }\right)   &  =\left[
\begin{array}
[c]{c}%
\zeta_{i,j-1}^{\left(  k\right)  }\\
\eta_{i,j-1}^{\left(  k\right)  }%
\end{array}
\right]  =\left[
\begin{array}
[c]{c}%
\left(  \left(  0,0\right)  ,\left(  w_{1}\lambda^{\left(  k\right)  }%
,\lambda^{\left(  k\right)  }w_{1}+w_{1}w_{2}g^{\left(  k\right)  }\left(
w_{2}\right)  \right)  \right) \\
\left(  \left(  \lambda^{\left(  k\right)  }z_{2}+z_{1}z_{2}f^{\left(
k\right)  }\left(  z_{1}\right)  ,0\right)  ,\left(  \lambda^{\left(
k\right)  }z_{2},0\right)  \right)
\end{array}
\right]  \in\ker\left(  \partial_{i,j-1}^{1}\right) \\
T\left(  s_{i-1,j}^{\left(  k\right)  }\right)   &  =\left[
\begin{array}
[c]{c}%
\zeta_{i-1,j}^{\left(  k\right)  }\\
\eta_{i-1,j}^{\left(  k\right)  }%
\end{array}
\right]  =\left[
\begin{array}
[c]{c}%
\left(  \left(  0,0\right)  ,\left(  \mu^{\left(  k\right)  }w_{1}%
,\mu^{\left(  k\right)  }w_{1}+w_{1}w_{2}v^{\left(  k\right)  }\left(
w_{2}\right)  \right)  \right) \\
\left(  \left(  \mu^{\left(  k\right)  }z_{2}+z_{1}z_{2}u^{\left(  k\right)
}\left(  z_{1}\right)  ,0\right)  ,\left(  \mu^{\left(  k\right)  }%
z_{2},0\right)  \right)
\end{array}
\right]  \in\ker\left(  \partial_{i-1,j}^{1}\right)
\end{align*}
in $\mathcal{O}_{i,j-1}$ and $\mathcal{O}_{i-1,j}$, respectively. Moreover,
based on Corollary \ref{corDecom}, we have
\[
\chi_{i,j}=p_{i,j}^{0}\left(  \chi_{i,j}\right)  =\tau_{i,j}^{0}\left(
f_{ij},g_{ij}\right)  =\left(  \left(  f_{ij}\left(  z_{1}\right)
,-\lambda_{ij}+f_{ij}\left(  z_{1}\right)  +g_{ij}\left(  w_{2}\right)
\right)  ,\left(  \lambda_{ij},g_{ij}\left(  w_{2}\right)  \right)  \right)
\]
in $\mathcal{O}_{i,j}$ for some $\left(  f_{i,j},g_{i,j}\right)
\in\mathcal{O}_{i+j}$ with $f_{ij}\left(  0\right)  =\lambda_{ij}%
=g_{ij}\left(  0\right)  $. Using Lemma \ref{lemTech3}, we deduce that
\begin{align*}
M_{i,j-1}s_{i,j-1}^{\left(  k\right)  }  &  =M_{x_{2},j-1}\zeta_{i,j-1}%
^{\left(  k\right)  }+M_{y_{2},j-1}\eta_{i,j-1}^{\left(  k\right)  }=\left(
\left(  0,0\right)  ,\left(  w_{1}g^{\left(  k\right)  }\left(  0\right)
,w_{1}g^{\left(  k\right)  }\left(  w_{2}\right)  \right)  \right) \\
&  +q^{j}\left(  \left(  \lambda^{\left(  k\right)  }+z_{1}f^{\left(
k\right)  }\left(  z_{1}\right)  ,\lambda^{\left(  k\right)  }+z_{1}f^{\left(
k\right)  }\left(  z_{1}\right)  \right)  ,\left(  \lambda^{\left(  k\right)
},\lambda^{\left(  k\right)  }\right)  \right)
\end{align*}%
\[
=\left(  \left(  q^{j}\left(  \lambda^{\left(  k\right)  }+z_{1}f^{\left(
k\right)  }\left(  z_{1}\right)  \right)  ,q^{j}\left(  \lambda^{\left(
k\right)  }+z_{1}f^{\left(  k\right)  }\left(  z_{1}\right)  \right)  \right)
,\left(  q^{j}\lambda^{\left(  k\right)  }+w_{1}g^{\left(  k\right)  }\left(
0\right)  ,q^{j}\lambda^{\left(  k\right)  }+w_{1}g^{\left(  k\right)
}\left(  w_{2}\right)  \right)  \right)  .
\]
In a similar way, we deduce that
\begin{align*}
N_{i-1,j}s_{i-1,j}^{\left(  k\right)  }  &  =N_{y_{1},i-1}\zeta_{i-1,j}%
^{\left(  k\right)  }+N_{x_{1},i-1}\eta_{i-1,j}^{\left(  k\right)  }%
=q^{i}\left(  \left(  \mu^{\left(  k\right)  },\mu^{\left(  k\right)  }%
+w_{2}v^{\left(  k\right)  }\left(  w_{2}\right)  \right)  ,\left(
\mu^{\left(  k\right)  },\mu^{\left(  k\right)  }+w_{2}v^{\left(  k\right)
}\left(  w_{2}\right)  \right)  \right) \\
&  +\left(  \left(  z_{2}u^{\left(  k\right)  }\left(  z_{1}\right)
,0\right)  ,\left(  z_{2}u^{\left(  k\right)  }\left(  0\right)  ,0\right)
\right)
\end{align*}%
\[
=\left(  \left(  q^{i}\mu^{\left(  k\right)  }+z_{2}u^{\left(  k\right)
}\left(  z_{1}\right)  ,q^{i}\left(  \mu^{\left(  k\right)  }+w_{2}v^{\left(
k\right)  }\left(  w_{2}\right)  \right)  \right)  ,\left(  q^{i}\mu^{\left(
k\right)  }+z_{2}u^{\left(  k\right)  }\left(  0\right)  ,q^{i}\left(
\mu^{\left(  k\right)  }+w_{2}v^{\left(  k\right)  }\left(  w_{2}\right)
\right)  \right)  \right)  .
\]
If $\omega_{i,j}^{\left(  k\right)  }=\left[
\begin{array}
[c]{c}%
a^{\left(  k\right)  }\\
b^{\left(  k\right)  }%
\end{array}
\right]  \in\mathcal{O}_{i,j}^{\oplus2}$ for some compactible quadruples
$a^{\left(  k\right)  }$ and $b^{\left(  k\right)  }$, then $\partial
_{i,j}^{1}\omega_{i,j}^{\left(  k\right)  }=\left(  z_{2}-q^{i}z_{1}\right)
a^{\left(  k\right)  }+\left(  q^{j}w_{2}-w_{1}\right)  b^{\left(  k\right)
}$. Using Corollary \ref{corTech1}, we obtain that
\[
\partial_{i,j}^{1}\omega_{i,j}^{\left(  k\right)  }=\left(  \left(  \left(
z_{2}-q^{i}z_{1}\right)  a_{1}^{\left(  k\right)  },-q^{i}z_{1}a_{2}^{\left(
k\right)  }+q^{j}w_{2}b_{2}^{\left(  k\right)  }\right)  ,\left(  z_{2}%
a_{3}^{\left(  k\right)  }-w_{1}b_{3}^{\left(  k\right)  },\left(  q^{j}%
w_{2}-w_{1}\right)  b_{4}^{\left(  k\right)  }\right)  \right)  .
\]
Put $\chi^{\left(  k\right)  }=\partial_{i,j}^{1}\omega_{i,j}^{\left(
k\right)  }+M_{i,j-1}s_{i,j-1}^{\left(  k\right)  }-N_{i-1,j}s_{i-1,j}%
^{\left(  k\right)  }$ to be convergent to $\chi_{i,j}$. By adding up all
those compatible quadruples obtained above, we have
\begin{align*}
\chi_{1}^{\left(  k\right)  }  &  =\left(  z_{2}-q^{i}z_{1}\right)
a_{1}^{\left(  k\right)  }+q^{j}\left(  \lambda^{\left(  k\right)  }%
+z_{1}f^{\left(  k\right)  }\left(  z_{1}\right)  \right)  -q^{i}\mu^{\left(
k\right)  }-z_{2}u^{\left(  k\right)  }\left(  z_{1}\right)  \rightarrow
f_{ij}\left(  z_{1}\right)  ,\\
\chi_{2}^{\left(  k\right)  }  &  =-q^{i}z_{1}a_{2}^{\left(  k\right)  }%
+q^{j}w_{2}b_{2}^{\left(  k\right)  }+q^{j}\left(  \lambda^{\left(  k\right)
}+z_{1}f^{\left(  k\right)  }\left(  z_{1}\right)  \right)  -q^{i}\left(
\mu^{\left(  k\right)  }+w_{2}v^{\left(  k\right)  }\left(  w_{2}\right)
\right) \\
&  \rightarrow-\lambda_{ij}+f_{ij}\left(  z_{1}\right)  +g_{ij}\left(
w_{2}\right)  ,\\
\chi_{3}^{\left(  k\right)  }  &  =z_{2}a_{3}^{\left(  k\right)  }-w_{1}%
b_{3}^{\left(  k\right)  }+q^{j}\lambda^{\left(  k\right)  }+w_{1}g^{\left(
k\right)  }\left(  0\right)  -q^{i}\mu^{\left(  k\right)  }-z_{2}u^{\left(
k\right)  }\left(  0\right)  \rightarrow\lambda_{ij},\\
\chi_{4}^{\left(  k\right)  }  &  =\left(  q^{j}w_{2}-w_{1}\right)
b_{4}^{\left(  k\right)  }+q^{j}\lambda^{\left(  k\right)  }+w_{1}g^{\left(
k\right)  }\left(  w_{2}\right)  -q^{i}\left(  \mu^{\left(  k\right)  }%
+w_{2}v^{\left(  k\right)  }\left(  w_{2}\right)  \right)  \rightarrow
g_{ij}\left(  w_{2}\right)  .
\end{align*}
For $w_{1}=0$, $z_{2}=0$ in $\chi_{3}^{\left(  k\right)  }\rightarrow
\lambda_{ij}$, we derive that $q^{j}\lambda^{\left(  k\right)  }-q^{i}%
\mu^{\left(  k\right)  }\rightarrow\lambda_{ij}$, Along the line $z_{2}%
=q^{i}z_{1}$ in $\chi_{1}^{\left(  k\right)  }\rightarrow f_{ij}\left(
z_{1}\right)  $, we obtain that $q^{j}f^{\left(  k\right)  }\left(
z_{1}\right)  -q^{i}u^{\left(  k\right)  }\left(  z_{1}\right)  \rightarrow
\left(  f_{ij}\right)  _{z}\left(  z_{1}\right)  $. Along the line
$w_{1}=q^{j}w_{2}$ in $\chi_{4}^{\left(  k\right)  }\rightarrow g_{ij}\left(
w_{2}\right)  $, we obtain that $q^{j}g^{\left(  k\right)  }\left(
w_{2}\right)  -q^{i}v^{\left(  k\right)  }\left(  w_{2}\right)  \rightarrow
\left(  g_{ij}\right)  _{w}\left(  w_{2}\right)  $. It follows that
\begin{align*}
q^{j}s_{i,j-1}^{\left(  k\right)  }-q^{i}s_{i-1,j}^{\left(  k\right)  }  &
=\left(  q^{j}f^{\left(  k\right)  }\left(  z_{1}\right)  -q^{i}u^{\left(
k\right)  }\left(  z_{1}\right)  ,0,q^{j}\lambda^{\left(  k\right)  }-q^{i}%
\mu^{\left(  k\right)  },q^{j}g^{\left(  k\right)  }\left(  w_{2}\right)
-q^{i}v^{\left(  k\right)  }\left(  w_{2}\right)  \right) \\
&  \rightarrow\left(  \left(  f_{ij}\right)  _{z}\left(  z_{1}\right)
,0,\lambda_{ij},\left(  g_{ij}\right)  _{w}\left(  w_{2}\right)  \right)  ,
\end{align*}
that is, the presence of $\lim_{k}s_{i,j-1}^{\left(  k\right)  }=s_{i,j-1}$
implies $\lim_{k}s_{i-1,j}^{\left(  k\right)  }=s_{i-1,j}$ and vise-versa. In
particular, if $j=0$ then the limit $\lim_{k}s_{i-1,0}^{\left(  k\right)
}=s_{i-1,0}$ ($s_{i,-1}^{\left(  k\right)  }=0$) does exist for every $i$. In
a similar way, $\lim_{k}s_{0,j-1}^{\left(  k\right)  }=s_{0,j-1}$
($s_{-1,j}^{\left(  k\right)  }=0$) exists for every $j$. Thus $s_{i,0}$,
$s_{0,j}$ do exist, and $\chi_{i,0}=\partial_{i,0}^{1}\omega_{i,0}%
-N_{i-1,0}s_{i-1,0}$, $\chi_{0,j}=\partial_{0,j}^{1}\omega_{0,j}%
+M_{0,j-1}s_{0,j-1}$ hold.

Finally, fix $m\leq n$. Then $s_{m-1,0}=\lim_{k}s_{m-1,0}^{\left(  k\right)
}$ implies that $\lim_{k}s_{m-2,1}^{\left(  k\right)  }=s_{m-2,1}$ converges
too. By iterating, we deduce that all $\lim_{k}s_{i,j}^{\left(  k\right)
}=s_{i,j}$, $i+j=m-1$ converge and
\begin{align*}
\chi_{m-i,i}  &  =\lim_{k}\partial_{m-i,i}^{1}\omega_{m-i,i}^{\left(
k\right)  }+M_{m-i,i-1}s_{m-i,i-1}^{\left(  k\right)  }-N_{m-i-1,i}%
s_{m-i-1,i}^{\left(  k\right)  }\\
&  =\partial_{m-i,i}^{1}\omega_{m-i,i}+M_{m-i,i-1}s_{m-i,i-1}-N_{m-i-1,i}%
s_{m-i-1,i}%
\end{align*}
for all $i$. Whence $\lim_{k}s_{i,j}^{\left(  k\right)  }=s_{i,j}$ exists for
all $i,j$, $i+j<n$, $\lim_{k}\partial_{i,j}^{1}\omega_{i,j}^{\left(  k\right)
}=\partial_{i,j}^{1}\omega_{i,j}$ and $\chi_{i,j}=\partial_{i,j}^{1}%
\omega_{i,j}+M_{i,j-1}s_{i,j-1}-N_{i-1,j}s_{i-1,j}$ hold for all $i,j$,
$i+j\leq n$.
\end{proof}

\end{document}